\documentclass[10pt,a4paper,openany,oneside]{article}

\usepackage{amsfonts}
\usepackage{fancyhdr}
\usepackage{verbatim}
\usepackage{graphicx,color}
\usepackage{authblk}
\usepackage{mathscinet}

\usepackage{times}
\usepackage[latin1]{inputenc}
\usepackage[english]{babel}
\usepackage{mathrsfs}
\usepackage{stmaryrd}
\usepackage{amssymb,amsmath,amsthm}
\usepackage{array}
\usepackage{hyperref}
\usepackage{graphicx}
\usepackage{subcaption}
\usepackage{eurosym}
\usepackage{cancel}

\usepackage{filecontents}

\usepackage{amssymb,amsfonts}
\usepackage{xcolor,dsfont}

\numberwithin{equation}{section}

\newtheorem{theorem}{Theorem}[section]
\newtheorem*{theorem*}{Theorem}
\newtheorem{thm}[theorem]{Theorem}

\newtheorem{lem}[theorem]{Lemma}
\newtheorem{proposition}[theorem]{Proposition}

\newtheorem{corollary}[theorem]{Corollary}

\newtheorem{definition}[theorem]{Definition}

\newtheorem{example}[theorem]{Example}
\newtheorem{remark}[theorem]{Remark}

\newtheorem{innercustomthm}{Theorem}
\newenvironment{customthm}[1]
{\renewcommand\theinnercustomthm{#1}\innercustomthm}
{\endinnercustomthm}

\newcommand{\be}{\begin{equation}}
\newcommand{\ee}{\end{equation}}
\newcommand{\beq}{\begin{equation*}}
\newcommand{\eeq}{\end{equation*}}
\newcommand{\enq}{\end{equation}}
\newcommand{\ben}{\begin{eqnarray}}
\newcommand{\een}{\end{eqnarray}}
\newcommand{\bea}{\begin{eqnarray*}}
\newcommand{\eea}{\end{eqnarray*}}

\newcommand{\cF}{ {\mathcal{F}}}

\newcommand{\cD}{ {\mathcal{D}}}
\newcommand{\cL}{ {\mathcal{L}}}

\newcommand{\Sc}{ {\mathcal{S}}}

\newcommand{\cT}{ {\mathcal{T}}}
\newcommand{\cR}{ {\mathcal{R}}}
\newcommand{\cN}{ {\mathcal{N}}}
\newcommand{\cM}{ {\mathcal{M}}}
\newcommand{\cU}{ {\mathcal{U}}}

\newcommand{\sign}{\mbox{\rm sign}}

\newcommand{\dist}{\mbox{\rm dist}}

\def\ker{{\mathrm{ker\,}}}
\def\Ran{{\mathrm{Ran\,}}}

\newcommand{\R}{{\mathbb{R}}}
\newcommand{\N}{{\mathbb{N}}}
\newcommand{\C}{{\mathbb{C}}}
\newcommand{\Z}{{\mathbb{Z}}}

\newcommand{\eps}{\varepsilon}

\newcommand{\supp}{\mbox{\rm supp}}

\def\ri{{\rm i}}

\DeclareMathOperator\perm{\boldsymbol{\epsilon}}

 \def\dd{\, {\rm d}}
\newcommand{\pa}{\partial}

\DeclareMathOperator{\Curl}{curl}

\DeclareMathOperator\Hdiv{H_{div}}
\DeclareMathOperator\Hcurl{H_{curl}}
\DeclareMathOperator\tr{Tr}

\def\Real{\mathop{\rm Re}}
\def\Imag{\mathop{\rm Im}}
\def\c{\mathord{\rm const}}

\DeclareMathOperator\sweyl{\sigma_{Weyl}}

\newcommand{\bspm}{\left(\begin{smallmatrix}}\newcommand{\espm}{\end{smallmatrix}\right)}
\newcommand{\bpm}{\begin{pmatrix}}\newcommand{\epm}{\end{pmatrix}}
\def\blem{\begin{lem}}\def\elem{\end{lem}}
\def\bthm{\begin{thm}}\def\ethm{\end{thm}}
\def\bprop{\begin{proposition}}\def\eprop{\end{proposition}}
\def\bcor{\begin{corollary}}\def\ecor{\end{corollary}}
\def\beq{\begin{equation}}\def\eeq{\end{equation}}
\def\beqq{\begin{equation*}}\def\eeqq{\end{equation*}}
\def\bal{\begin{align}}\def\eal{\end{align}}
\def\bpf{\begin{proof}}\def\epf{\end{proof}}
\def\bex{\begin{example}}\def\eex{\end{example}}
\def\brem{\begin{remark}}\def\erem{\end{remark}}

\newcommand{\curl}[1]{\nabla \times #1 }

\newcommand{\ba}{\begin{aligned}}
\newcommand{\ea}{\end{aligned}}

\def\C{\mathbb C}
\def\R{\mathbb R}
\def\N{\mathbb N}
\def\Z{\mathbb Z}

\DeclareFontFamily{U}{mathx}{\hyphenchar\font45}
\DeclareFontShape{U}{mathx}{m}{n}{
<5> <6> <7> <8> <9> <10>
<10.95> <12> <14.4> <17.28> <20.74> <24.88>
mathx10
}{}
\DeclareSymbolFont{mathx}{U}{mathx}{m}{n}
\DeclareFontSubstitution{U}{mathx}{m}{n}
\DeclareMathAccent{\widecheck}{0}{mathx}{"71}

\newcolumntype{D}[1]{>{\displaystyle} #1}

\makeatletter
\def\@fnsymbol#1{\ensuremath{\ifcase#1\or \dagger\or \ddagger\or
\mathsection\or \mathparagraph\or \|\or **\or \dagger\dagger
\or \ddagger\ddagger \else\@ctrerr\fi}}
\makeatother

\newcommand{\Addresses}{{
\bigskip
\footnotesize

T.~Dohnal, \textsc{Martin Luther University Halle-Wittenberg, Institute of Mathematics, Theodor Lieser Str.5, 06120 Halle, Germany}\par\nopagebreak
\textit{E-mail address}, T.~Dohnal: \texttt{tomas.dohnal@mathematik.uni-halle.de}

\medskip

M.~Plum, \textsc{Karlsruher Institut f\"ur Technologie, Institute for Analysis, Englerstrasse 2, 76131 Karlsruhe, Germany}\par\nopagebreak
\textit{E-mail address}, M.~Plum: \texttt{michael.plum@kit.edu}
\medskip

K.M.~Schmidt, \textsc{School of Mathematics, Cardiff University, Senghennydd Road, Cardiff CF24 4AG, UK}\par\nopagebreak
\textit{E-mail address}, K.M.~Schmidt: \texttt{schmidtkm@cardiff.ac.uk}
\medskip

I.~Wood, \textsc{School of Engineering, Mathematics and Physics, Cornwallis South, University of Kent, Canterbury, CT2 7NF, UK}\par\nopagebreak
\textit{E-mail address}, I.~Wood: \texttt{i.wood@kent.ac.uk}
}}

\begin{document}
\title {Spectrum of the Maxwell Equations for a Flat Interface between Non-Homogeneous Dispersive Media in 2D and 3D}

\author{Tom\'a\v{s} Dohnal, Michael Plum, Karl M. Schmidt, and Ian Wood}

\maketitle
\Addresses


\abstract{The present work concerns the time-harmonic Maxwell equations in two- and three-dimensional space, divided into two half-spaces by a flat interface. The two half-spaces are filled with media whose electric permittivity is frequency-dependent and varies as a function of the distance from the interface only; this dependence is assumed to satisfy a regularity condition in each half-space but may be discontinuous at the interface. No specific model for the frequency dependence is assumed. For the associated operator pencil, we characterise subsets of the resolvent set and separate subsets of the Weyl spectrum corresponding to radiation away from the interface and along the interface, respectively. The characterisation of the sets is via conditions on the fundamental solutions.
If the media are periodic in the direction orthogonal to the interface, a more explicit description of these sets can be given in terms of Floquet theory of related Sturm-Liouville equations.
These results generalise earlier work in one and two dimensions where the media were assumed to be homogeneous in each half-space.
}


\section{Introduction}

We study the spectrum of an operator pencil corresponding to the time-harmonic Maxwell problem in $\R^2$ and $\R^3$ with an interface between two non-homogeneous and dispersive (i.e. frequency dependent) materials. The interface is given by the plane $x_1=0$ and the two media's electromagnetic properties are dependent only on $x_1$ (and the frequency $\omega$) but independent of $x_2$ and $x_3$. This paper is a continuation of \cite{BDPW24} where the problem in $\R$ and $\R^2$ with an interface between two homogeneous ($x-$independent) media was analyzed.

In the second-order formulation for the spatial profile $E$ of the electric field, and setting the permeability equal to one, Maxwell's equations for dispersive materials in $\R^N, N=2,3,$ in the frequency domain are
\begin{equation}\label{E:Maxw-2ndord}
	\nabla\times\nabla\times E-\omega^2\perm(x_1,\omega)E=0, \quad \nabla \cdot (\perm E)=0, \quad x\in \R^N,
\end{equation}
where in our interface setting
\beq\label{E:perm}
\perm(x_1,\omega)=\begin{cases}
	\perm_-(x_1,\omega), & x_1<0,\\
	\perm_+(x_1,\omega), & x_1>0.
\end{cases}
\eeq
We assume 
\beq\label{E:ass-perm}
D(\perm):=\{ \omega \in \C: \perm_\pm(\cdot,\omega)\in W^{3,\infty}(\R_\pm)  \} \neq \emptyset,
\eeq
where $\R_+:=(0,\infty)$ and $\R_-:=(-\infty,0)$.
In the case $N=2$ we additionally assume that $E$ is independent of $x_3$. Note that possible singularities of $\omega \mapsto \perm(\cdot,\omega)$ lie outside of $D(\perm)$. Such singularities are common in physical models, like the Drude or Lorentz models, see \eqref{E:Drude} and \eqref{E:Lorentz}. Our spectral analysis covers the case of $\perm$ with singularities.

We view the left hand side of \eqref{E:Maxw-2ndord} as the  operator pencil $L(\omega):=-\omega^2\perm(\cdot,\omega)+\nabla\times\nabla\times$ acting on $E$. Note that in \cite{BDPW24} we had $\perm_\pm$ independent of $x$ and studied the cases $x\in \R$ and $x\in \R^2$. We make no assumptions on the $\omega$-dependence of $\perm_\pm$. In particular, the dependence can be nonlinear and $\perm_\pm$ can be complex-valued, which is the case in non-conservative (typically dissipative) media. The resulting operator pencil is then non-self-adjoint and nonlinear in the spectral parameter $\omega$. 

The interface  induces localization of the field, see the well-known effect of surface plasmons \cite{Zayats05,Pitarke_2007,Grieser14,AMRZ17}, where the first two references are from the physics literature. However, full two- or three-dimensional localization is not expected as there is no mechanism for localization in directions parallel to the interface.
By allowing the permittivities $\perm_\pm$ to depend on $x_1$, we make it possible to consider, e.g., the interface between two materials of which one is a photonic crystal. Wave propagation in photonic crystals built from dispersive  materials is of physical interest, see, e.g., \cite{SHCL2014,AD2024}. Also, surface plasmons are often studied on surfaces of gratings or crystals made of dispersive materials, see, e.g., chapter 3 in \cite{maier2007}.

The spectrum of the dispersive Maxwell problem has been previously studied also by other authors. A linear non-self-adjoint pencil problem was considered in \cite{Lassas98,ABMW19}. A nonlinear non-self-adjoint pencil in a cavity was studied in \cite{HP20}.  Recent works on the spectrum of the Maxwell operator with dispersive media include \cite{CJ2026-I,E2024,FM2023}. For more background on operator pencils, their use in modeling physical problems and their spectra, we refer to \cite{Markus1988,MoellerPivo} as well as the introduction in \cite{BDPW24} and the references therein. Closely related to our work are the papers \cite{CHJ2017,CHJ2022,CJ2026-II} where the wave propagation at a flat interface between a dispersive medium and vacuum is investigated in $\R^2$. In \cite{CJ2026-II} also a dispersive slab inside vacuum is considered. In all cases of \cite{CHJ2017,CHJ2022,CJ2026-II} the media are, however, homogeneous, and $\perm_\pm(\omega)$ is given by a specific non-dissipative Drude model, leading to a self-adjoint problem.

In \cite{BDPW24}, where $\perm_\pm(\omega)$ were constants, we divided the spectrum into the part outside the singular set 
\beq\label{Omega0} \Omega_0:=\{\omega\in D(\perm):\omega^2\perm_+(\omega)=0 \ \text{or} \  \omega^2\perm_-(\omega)=0\}\eeq
 and the part in $\Omega_0$. In $\Omega_0$ the pencil $L(\omega)$ reduces to the operator $\nabla\times \nabla\times$ and features special spectral properties. Also, the set 
\beq\label{M} M:=\{\omega \in D(\perm): \omega^2\perm_+(\omega)\in \R_+ \ \text{or} \ \omega^2\perm_-(\omega)\in \R_+ \}\eeq
 plays a special role in \cite{BDPW24} - it lies in the essential spectrum since plane wave solutions of \eqref{E:Maxw-2ndord} exist if $\omega \in M$.

In the current paper, where $\perm_\pm$ depend on $x_1$, the coefficients $\omega^2\perm_\pm(\cdot,\omega)$ can be in $[0,\infty)$ or arbitrarily close to $[0,\infty)$ on a subset of $\R_\pm$. It is not clear to us what such values of $\omega$ mean for the spectrum. 
In most of the paper we work in the complement of the following exceptional sets
\beq\label{Omegaa}
\begin{aligned}
\Omega^+_a = \{ \omega \in D(\perm) : \dist(\omega^2 \perm_+( \R_+ , \omega), \{a\}) = 0 \}, \\
\Omega^-_a = \{ \omega \in D(\perm) : \dist(\omega^2  \perm_-( \R_-, \omega), \{a\}) = 0  \}
\end{aligned}
\eeq
for some suitably chosen $a\geq 0$ or  in the complement of the set
\beq\label{Omega} 
\Omega := \{ \omega \in D(\perm) : \dist(\omega^2 \perm(\R,\omega), \R_+) = 0\}.
\eeq
Note that the exceptional set $\Omega$ is a generalization of $\Omega_0 \cup M$.
Only for the example of $\perm_\pm$ independent of $x$, we determine the spectrum on the whole set $\C\setminus \Omega_0$.

\subsection{The Mathematical Setting}\label{S:math-setting}

Here, we  explain the origin of equation \eqref{E:Maxw-2ndord} and the functional analytic setting in which we study the spectrum. We also review the definitions of the spectrum and some of its parts for the case of operator pencils. For a more detailed discussion of this see \cite{BDPW24}.

The time-harmonic Maxwell equations in $\R^N, N=2,3,$ in the absence of free charges and free currents and with a constant permeability take the form
\begin{equation}\label{E:Maxw-1stord}
\begin{aligned}
\nabla\times E & = \ri\omega \boldsymbol{\mu} H, \\
\nabla\times H & = -\ri\omega D, \quad D =\perm(x_1,\omega)E,\\
\nabla \cdot D & =0, \quad \nabla \cdot H =0,
\end{aligned}
\end{equation}
where 
$$\nabla := (\pa_{x_1},\pa_{x_2},\pa_{x_3})^\top \ \text{if} \ N=3 \qquad \text{and} \qquad  \nabla := (\pa_{x_1},\pa_{x_2},0)^\top \ \text{if} \ N=2.$$
$\perm$ is the electric permittivity of the medium and $\boldsymbol{\mu}>0$ is the magnetic permeability of the material.  In the rest of the paper we set, for the sake of brevity,
$$\boldsymbol{\mu}=1.$$
This can be achieved by rescaling $H$ and the variable $x$. This simplification has no effect on the results. We assume that $\perm$ depends on space only through the variable $x_1$.

Equation \eqref{E:Maxw-1stord} arises from the time-dependent dispersive Maxwell equations by the generalized time-harmonic ansatz
$(\mathcal E, \mathcal H)(x,t)=(E,H)(x)e^{-\ri\omega t}+(\overline{E},\overline{H})(x)e^{\ri\overline{\omega} t}, \omega \in\C,$ for the electric field $\mathcal E$ and the magnetic field $\mathcal H$. In detail, the displacement field $\cD$ in the time domain is
$$\cD(x,t)=\perm_0\left(\mathcal{E}(x,t)+\int_{-\infty}^t\chi(x_1,t-s)\mathcal{E}(x,s)\dd s\right),$$
where $\chi:\R^2\to\R$ is called the response function or susceptibility.  The above time-harmonic ansatz produces $\cD(x,t)=D(x)e^{-\ri \omega t}+\bar{D}(x)e^{\ri \overline{\omega}t}$ with
$D =\perm_0 (1+\cF_t(\chi)(x_1,\omega))E,$
where $\cF_t$ denotes the temporal Fourier transform, i.e.,
$$\cF_t(\chi)(x_1,\omega):=\int_\R \chi(x_1,t)e^{\ri\omega t}\dd t.$$
Setting
$$\perm(x_1,\omega):=\perm_0 (1+ \cF_t(\chi)(x_1,\omega)),$$
one recovers $D$ in \eqref{E:Maxw-1stord}.

In the case of a non-dispersive dielectric in $x_1>0$ and a dispersive metal in $x_1<0$ the following is a simple example,
	\begin{equation}\label{E:chi}
		\cF_t(\chi)(x_1,\omega):=
		\begin{cases}
			\cF_t(\chi_m)(\omega) & \text{for } x_1<0, \\
			\eta & \text{for } x_1>0,
		\end{cases}
	\end{equation}
	where $\eta>0$ and where physically relevant frequency dependent permittivities $\chi_m$ are, e.g., the Drude model
		\beq\label{E:Drude}
		\cF_t(\chi_m)(\omega)=-\frac{c_D}{\omega^2+\ri\gamma\omega}
		\eeq
		with $ c_D,\gamma>0$ and the Lorentz model
		\beq\label{E:Lorentz}
		\cF_t(\chi_m)(\omega)=-\frac{c_L}{\omega^2+\ri\gamma\omega-\omega_*^2}
		\eeq
		with $c_L,\gamma,\omega_*>0$, see e.g. \cite{ACL2018,CJM2023, Pitarke_2007}.
For an inhomogeneous structure, e.g., a crystal, the constants in \eqref{E:chi}, \eqref{E:Drude} and \eqref{E:Lorentz} may be $x$-dependent.

In the second-order formulation for the $E$-field, system \eqref{E:Maxw-1stord} translates to \eqref{E:Maxw-2ndord}, which is our object of study. The geometry we are considering consists of the two half-spaces
$$\R^N_\pm :=\R_\pm \times \R^{N-1}, N=2,3, \quad \text{where} \quad \R_\pm=\{x\in \R: \pm x>0\}.$$
In each half-space, the permittivity is smooth (that is, $W^{3,\infty}$) for $\omega$ inside $D(\perm)$ with $\perm$ given by \eqref{E:perm}.

Throughout the paper the complex square root of $z=|z|e^{\ri \theta}, \theta \in (-\pi,\pi]$ is defined as $\sqrt{z}:=\sqrt{|z|}e^{\ri \theta/2}$. Hence $\Real(\sqrt{z})\geq 0$.

As mentioned above, we study the  spectral problem corresponding to \eqref{E:Maxw-2ndord} as an operator pencil.  We  first define
\begin{equation}\label{def:pencil}
\cL:=(L(\omega))_{\omega\in D(\perm)},
\end{equation}
\begin{equation}\label{E:pencil}
L(\omega):T(\nabla)-B(\omega), \quad \text{where} \quad  T(\nabla)u:=\nabla\times \nabla\times u, \ B(\omega)u:= \omega^2\perm(x_1,\omega)u.
\end{equation}
Note that in \cite{BDPW24} the pencil $\cL$ was defined using two parameters, namely $\omega$ and $\lambda\in \C$, where $\lambda$ was needed in order to define isolatedness of eigenvalues and hence the discrete spectrum. Here, however, we do not discuss these concepts and therefore omit the parameter $\lambda$.

Next, we describe our choice of function spaces. We work in the Hilbert space
$L^2(\R^N)^3 :=L^2(\R^N,\C^3)$
and choose the domain
\begin{equation}\label{E:domain}
\cD_\omega:=\{E \in L^2(\R^N)^3:  \nabla \times E, \ \nabla\times \nabla\times E\in L^2(\R^N)^3,
\nabla\cdot(\perm(\cdot,\omega) E)=0 \ (\text{distributionally})\}.
\end{equation}
Due to the divergence condition in the domain and because the divergence of the curl vanishes, $L(\omega)$ maps into the space of divergence free functions
$$L(\omega): \cD_\omega\to\cR :=\{f\in L^2(\R^N)^3: \nabla \cdot f =0\}.$$
We are in a non-standard situation with the domain of the operator not being a subspace of the codomain.

Note that $\cD_\omega$ can be written using $L^2$-conditions on each half-space and jump conditions across the interface, namely,
$$
\cD_\omega=\{E \in L^2(\R^N)^3:  \nabla \times E, \ \nabla\times \nabla\times E\in L^2(\R_\pm^N)^3,
\nabla\cdot(\perm(\cdot,\omega) E)=0 \ \text{on }\R^N_\pm \ \text{ and \eqref{E:IFC} holds} \}
$$
\beq\label{E:IFC}
\begin{aligned}
T_+^n(\perm(\cdot,\omega)E)-T_-^n(\perm(\cdot,\omega)E)=0, \\ 
 T_+^t(E)-T_-^t(E)=0, \\
T_+^t(\nabla \times E)-T_-^t(\nabla \times E)=0,
\end{aligned}
\eeq
where the normal trace operators $T_\pm^n$ and the tangential trace operators $T_\pm^t$ are defined in Appendix \ref{S:traces}.

Let us denote the variables for the directions parallel to the interface by $x_\parallel$, i.e.,
$$x_\parallel:=x_2 ~ {\rm if}~ N=2 \quad  {\rm and} \quad  x_\parallel:=(x_2,x_3)^\top ~ {\rm if}~ N=3.$$ To elements of $\cD_\omega$ we apply the Fourier transformation in the $x_\parallel-$variables, which in the special case of $E(x_1,\cdot)\in L^1(\R^{N-1})$ is given by the formula $$u(x_1,k):=
\begin{cases}
u(x_1,k_2):=\widehat{E}(x_1,k_2):=(2\pi)^{-1/2}\int_{\R}E(x_1,x_2)e^{-\ri k_2 x_2 }\dd x_2, & N=2,\\
\\
u(x_1,k_2,k_3):=\widehat{E}(x_1,k_2,k_3):=(2\pi)^{-1}\int_{\R^2}E(x_1,x_2,x_3)e^{-\ri(k_2 x_2 + k_3 x_3)}\dd (x_2,x_3), & N=3
\end{cases}
$$
with the variable $k\in \R^{N-1},$ and otherwise is defined by the standard extension procedure.

The functions $\nabla \times E$ and $\nabla \times \nabla \times E$ transform to
$$
\begin{aligned}
&\nabla_k\times u = \bspm
\ri k_2 u_3 \\ -\pa_{x_1}u_3\\ \pa_{x_1} u_2-\ri k_2 u_1
\espm \quad \text{and} \quad   T_k(\pa_{x_1})u:=\nabla_k\times \nabla_k\times u = \bspm k_2^2 & \ri k_2\pa_{x_1}& 0 \\ \ri k_2\pa_{x_1} & -\pa_{x_1}^2 & 0 \\ 0 & 0 & k_2^2-\pa_{x_1}^2\espm\qquad \text{if} \ N=2,\\
&\nabla_k\times u = \bspm
\ri k_2 u_3 -\ri k_3 u_2\\ \ri k_3 u_1-\pa_{x_1}u_3\\ \pa_{x_1} u_2-\ri k_2 u_1
\espm \quad \text{and} \quad   T_k(\pa_{x_1})u:=\nabla_k\times \nabla_k\times u = \bspm |k|^2 & \ri k_2\pa_{x_1}& \ri k_3 \pa_{x_1} \\ \ri k_2\pa_{x_1} & k_3^2-\pa_{x_1}^2 & -k_2k_3 \\ \ri k_3 \pa_{x_1} & -k_2k_3 & k_2^2-\pa_{x_1}^2\espm\qquad \text{if} \ N=3,
\end{aligned}
$$
with $|k|^2=k_2^2+k_3^2$ and
$$\nabla_k:= \begin{cases}
	(\pa_{x_1},\ri k_2, 0)^\top, & N=2,\\
	(\pa_{x_1},\ri k_2,\ri k_3)^\top, &N=3.
\end{cases}$$

This leads us to consider the operator  $\widehat{L}$, where
\beq\label{E:NL-eval}
 \widehat{L}(\omega)u:=T_k(\pa_{x_1})u-\omega^2\perm(x_1,\omega)u, \ x_1 \in \R\setminus \{0\}, k\in\R^{N-1}.
\eeq
Clearly, the operator $\widehat{L}(\omega)$ for $N=2$ can be obtained from $\widehat{L}(\omega)$ for $N=3$ by setting $k_3=0$. However, to avoid any confusion, we present many properties of $\widehat{L}$ in the two cases separately.

The domain of $\widehat{L}(\omega)$ can be rewritten via conditions on each component as
\beq\label{E:Dom-FT}
\begin{aligned}
	\widehat{\cD}_\omega=&\left\{u\in L^2(\R^2)^3: k_2u_3,u_3', \ri k_2u_1-u_2',k_2^2u_1+\ri k_2u_2', \ri k_2 u_1'-u_2'', -u_3'' +k_2^2u_3 \in L^2(\R_\pm^2),\right.\\
	&\left.\text{ and \eqref{E:divergence-u}, \eqref{E:IFC-u} hold}\right\}, \hbox{ for } N=2,\\
\widehat{\cD}_\omega=&\left\{u\in L^2(\R^3)^3: k_2u_3-k_3u_2,\ri k_3u_1-u_3', \ri k_2u_1-u_2',|k|^2u_1+\ri k_2u_2'+\ri k_3 u_3',\right.\\
&\left. \hspace{-1cm} \ri k_2 u_1'-u_2''+k_3^2u_2-k_2k_3 u_3, \ri k_3 u_1'-u_3'' -k_2k_3 u_2+k_2^2u_3 \in L^2(\R_\pm^3), \text{ and \eqref{E:divergence-u}, \eqref{E:IFC-u} hold}\right\}, \hbox{ for } N=3,
\end{aligned}
\eeq
where $f':=\pa_{x_1}f$, and
\beq\label{E:divergence-u}
\begin{aligned}
(\perm(\cdot,\omega) u_1)'+\ri \perm(\cdot,\omega)k_2 u_2&=0 \quad  \text{on} \quad \R^2_\pm \quad \text{if} \quad N=2,\\
(\perm(\cdot,\omega) u_1)'+\ri \perm(\cdot,\omega)(k_2 u_2 + k_3 u_3)&=0 \quad  \text{on} \quad \R^3_\pm \quad \text{if} \quad N=3,
\end{aligned}
\eeq
\beq\label{E:IFC-u}
\begin{aligned}
	\llbracket \perm(\cdot,\omega) u_1\rrbracket &:=\widehat{T_+^n}(\perm(\cdot,\omega)u)-\widehat{T_-^n}(\perm(\cdot,\omega)u)=0, \\ 
	(\llbracket u_2\rrbracket,\llbracket u_3\rrbracket)&:=\widehat{T_+^t}(u)-\widehat{T_-^t}(u)=0, \\
	(\llbracket (\nabla_k\times u)_2\rrbracket,\llbracket (\nabla_k\times u)_3\rrbracket)&=\widehat{T_+^t}(\nabla_k \times u)-\widehat{T_-^t}(\nabla_k \times u)=0,
\end{aligned}
\eeq
where the normal trace operators $\widehat{T_\pm^n}$ and the tangential trace operators $\widehat{T_\pm^t}$ in the Fourier space are defined in Appendix \ref{S:traces}. The definitions are such that $\widehat{T_\pm^n} \cF_t = \cF_t T_\pm^n$ in $\Hdiv(\R^N_\pm)$ and $\widehat{T_\pm^t} \cF_t = \cF_t T_\pm^t$ in $\Hcurl(\R^N_\pm)$, see  Appendix \ref{S:traces}.
In \eqref{E:IFC-u} and the rest of the paper we use the bracket notation both for jumps of normal and tangential traces.

The codomain $\cR$ in the Fourier variables reads
\beq\label{E:R-Four}
\widehat{\cR} = \begin{cases}
\{ r\in L^2(\R^2)^3:  r_1' + \ri k_2 r_2 =0 \}, & N=2,\\
\{ r\in L^2(\R^3)^3:  r_1' + \ri (k_2 r_2 + k_3 r_3)=0 \}, & N=3.
\end{cases}
\eeq

	On $\cD_\omega$ we introduce the norm
	\beq\label{E:graph-norm}
	\|u\|_{\cD_\omega}:= \sqrt{\langle u,u\rangle_{\cD_\omega}}, \ \langle u,v\rangle_{\cD_\omega}:=\langle u,v\rangle_{L^2(\R^N)} + \langle T(\nabla)u,T(\nabla)v\rangle_{L^2(\R^N)},
	\eeq
	i.e.,~the graph norm corresponding to $T(\nabla)$.
	\blem\label{L:HS}
$(\cD_\omega,\langle\cdot,\cdot\rangle_{\cD_\omega})$ is a Hilbert space.
\elem
\begin{proof}
Due to the integration by parts in Lemma \ref{L:PI-3D}
$$\|\nabla\times v\|_{L^2(\R^N)^3}^2=\langle v,\nabla\times \nabla \times v\rangle_{L^2(\R^N)^3}\leq \|v\|_{L^2(\R^N)^3} \|\nabla\times\nabla\times v\|_{L^2(\R^N)^3}$$
for each $v\in \cD_\omega$ and $N=2,3$. Hence the norm
$$\|v\|:=\left(\|v\|_{L^2(\R^N)^3}^2+\|\nabla\times v\|_{L^2(\R^N)^3}^2+\|\nabla\times \nabla\times v\|_{L^2(\R^N)^3}^2\right)^{1/2}$$
is equivalent to $\|\cdot\|_{\cD_\omega}$.

The rest of the proof is completely analogous to the proof of Lemma 4.2 in \cite{BDPW24}.
	\end{proof}

\bcor\label{C:closed-op}
The operator $T(\nabla):\cD_{\omega}\to\cR$ is closed.
\ecor
\begin{proof}
For a sequence $(u_k,T(\nabla)u_k)$ with $u_k\in \cD_\omega$ and $ u_k\to u$ and $T(\nabla)u_k\to v$ in $L^2(\R^N)^3$ we need to show that $v=T(\nabla)u.$ The convergence assumption implies that $(u_k)$ is Cauchy in $\cD_\omega$. 
Using Lemma \ref{L:HS} and the fact that the convergence in $\cD_\omega$ implies convergence in $L^2$, we get $u \in \cD_\omega$ and $T(\nabla) u_k \to T(\nabla)u$. We conclude that $v=T(\nabla)u.$
\end{proof}
Because of Corollary \ref{C:closed-op} for any fixed $\omega$ and $\lambda$ the operator $L(\omega;\lambda)$ is a sum of a bounded and a closed operator, and thus closed. Hence, it makes sense to study its spectrum.

\subsubsection{Definition of the Spectrum}\label{sec:spec}
Here we give the definitions of those spectral sets which are studied in the rest of the paper. For more details on the spectra of operator pencils, we refer to \cite{BDPW24}.

The \textbf{resolvent set}  of  the pencil $\cL$ is defined as
	$$\rho(\cL):=\{\omega\in D(\perm): L(\omega):\cD_\omega\to \cR \text{ is bijective with a bounded inverse}\},$$
and the \textbf{spectrum} of $\cL$ is given by
	\beq\label{E:Pspec}
	\sigma(\cL):=D(\perm)\setminus \rho(\cL)=\{\omega\in D(\perm): 0\in \sigma(L(\omega))\},
	\eeq
	where $\sigma(L(\omega))$ is the spectrum (defined in the standard sense) of the operator $L(\omega)$ at a
	fixed   $\omega$.

Next, the \textbf{point spectrum} is
	$$\sigma_p(\cL):=\{\omega\in D(\perm): \exists u \in \cD_\omega\setminus \{0\}: L(\omega)u=0 \}.$$
	Elements of $\sigma_p(\cL)$ are called \textbf{eigenvalues} of $\cL$. The \textbf{geometric multiplicity} of an eigenvalue $\omega$ is given by $\dim\ker(L(\omega))$.

	A sequence $(u^{(n)})\subset \cD_\omega$ is called a   \textbf{Weyl sequence at $\omega$} if
	$$\|u^{(n)}\|_{L^2(\R^N)^3}=1 \ \forall n\in\N, u^{(n)} \rightharpoonup 0 \text{ in } L^2(\R^N)^3, \text{ and } \|L(\omega)u^{(n)}\|_{L^2(\R^N)^3}\to 0 \ (n\to\infty).$$
	The \textbf{Weyl spectrum} is
	$$\sweyl(\cL):=\{\omega \in D(\perm): \text{ a Weyl sequence at $\omega$ exists}\}.$$



\subsection{Main Results}

Here we describe our main spectral results in a compact form. More details and all proofs will follow in subsequent sections.

Our results are subdivided into the case of general permittivity functions $\perm_\pm$ satisfying just \eqref{E:ass-perm}, the more special case where both $\perm_+(\cdot,\omega)$ and $\perm_-(\cdot,\omega)$ are periodic, and the even more special (but still physically very relevant) homogeneous case where $\perm_+$ and $\perm_-$ are independent of the spatial variable, i.e., $\perm_\pm(x_1,\omega)=\perm_\pm(\omega).$ The results for the homogeneous case with $N=3$ are a direct extension of \cite{BDPW24}, where the same setting in dimensions $N=1,2$, was studied.

The results on general $\perm_\pm$ contain various assumptions which are too lengthy to be mentioned here in this summary section. We refer to Theorem \ref{T:resolv-set-gen}, where a subset of the resolvent set is identified, to Theorems \ref{T:Weyl-sp-x1} and \ref{T:Weyl-sp} which provide subsets of the Weyl spectrum corresponding to radiation in the $x_1$-direction (Theorem \ref{T:Weyl-sp-x1}) and in a direction along the interface $\{x_1=0\}$ (Theorem \ref{T:Weyl-sp}), and to Theorem \ref{T:no-evals-fin-multip} containing a partial result on the non-existence of eigenvalues.

For the periodic and the homogeneous case, the required assumptions can be stated much more concisely and hence we give more details here.

\begin{customthm}{A}[periodic $\perm_\pm$]
\label{T:main-periodic}
Let $N=2,3$ and suppose that in addition to \eqref{E:ass-perm} the functions $\perm_\pm(\cdot,\omega)$ are $a_\pm$-periodic for each $\omega \in D(\perm)$.
Let $D_\pm^{(v)}(k)$ and $D_\pm^{(w)}(k)$ denote the discriminants of the differential equations \eqref{3} and \eqref{4}, respectively, on the intervals $[0,a_+]$ and $[-a_-,0]$ according to the index $\pm$,  respectively (see Sec.~\ref{S:per-estimates}).  Let $d(k)$ and $\widetilde{d}(k)$ be the quantities defined in Definition \ref{D:S} for \eqref{3} and \eqref{4}, respectively.

\begin{itemize}
\item[(a)] (\textbf{resolvent set}) Suppose that for some fixed $\omega\in D(\perm)\setminus\Omega$  and all $k \in \R^{N-1}$
$$D_+^{(v)}(k), D_-^{(v)}(k), D_+^{(w)}(k), D_-^{(w)}(k) \notin [-2,2],$$
\beq\label{ass-det}
d(k)\neq 0 \quad \hbox{ and } \quad\widetilde{d}(k)\neq 0.
\eeq

Furthermore, assume that
\beq\label{eps-cond} 
\perm_+(0,\omega)+\perm_-(0,\omega)\neq 0 \quad \text{or} \quad \perm_+'(0,\omega)-\perm_-'(0,\omega)\neq 0.
\eeq
Then $\omega\in \rho(\cL)$.

\item[(b)] (\textbf{radiation away from the interface}) For some fixed $\omega \in D(\perm)$ suppose that $k_0\in \R^{N-1}$ exists such that one of the following two options holds.
$$
\begin{aligned}
&\mathrm{I)} \qquad \left(\omega \notin \Omega^{+}_{|k_0|^2} \ \text{and} \ D_+^{(v)}(k_0) \in [-2,2] \right) \ \text{or}\ D_+^{(w)}(k_0)\in [-2,2],\\
&\mathrm{II)}  \qquad \left(\omega \notin \Omega^{-}_{|k_0|^2} \ \text{and} \ D_-^{(v)}(k_0)\in [-2,2] \right) \ \text{or} \ D_-^{(w)}(k_0)\in [-2,2].
\end{aligned}
$$
Then $\omega\in \sweyl(\cL)$.

\item[(c)] (\textbf{radiation along the interface})  For some fixed $\omega \in D(\perm)$ suppose that $k_0\in \R^{N-1}$ exists such that
$$\omega \notin \Omega^{+}_{|k_0|^2} \cup \Omega^{-}_{|k_0|^2}~, \quad D_+^{(v)}(k_0),D_-^{(v)}(k_0) \notin [-2,2], \quad \text{and} \quad d(k)=0$$
or
$$D_+^{(w)}(k_0),D_-^{(w)}(k_0) \notin [-2,2] \quad \text{and} \quad \widetilde{d}(k)=0.$$
Then $\omega\in \sweyl(\cL)$.

\item[(d)] (\textbf{eigenvalues}) There are no eigenvalues of $\cL$ of finite geometric multiplicity.
\end{itemize}
\end{customthm}
For the proof see Proposition \ref{P:Sp-resolv} and Theorem \ref{T:resolv-set-gen} (for (a)), Lemma \ref{lem:Weyl-x1-per} (for (b)), Lemma \ref{L:Lk-evals-periodic} and Theorem \ref{T:Weyl-sp} (for (c)), and Theorem \ref{T:no-evals-fin-multip} (for (d)).
We add that part (b) is proved by constructing Weyl sequences with their support moving to infinity in the $x_1$-direction (radiation in $x_1$), while for part (c) Weyl sequences localized at the interface $\{x_1=0\}$ with their support moving to infinity in an $x_\parallel$-direction (radiation along the interface) are chosen.

For the radiation in $x_1$ we study also the ``asymptotically periodic'' case where, for some $\omega\in D(\perm)$,
$$\perm_\pm(\cdot,\omega)=\perm_{p,\pm}(\cdot,\omega)+\perm_{as,\pm}(\cdot,\omega)$$
with $\perm_{p,\pm}(\cdot,\omega)$ periodic and $\perm_{as,\pm}(\cdot,\omega)\in L^1(\R_\pm)$ (for spectrum generated by equation \eqref{4}) or $\perm_{as,\pm}(\cdot,\omega)\in W^{1,1}(\R_\pm)$ (for spectrum generated by equation \eqref{3}). In Lemma \ref{L:Weyl-x1-as-per} we prove that under the additional condition of stability, i.e., both solutions in the fundamental system of the corresponding periodic problem (with $\perm_p$ instead of $\perm$) being bounded (on $\R_+$ or $\R_-$) we get $\omega \in \sweyl(\cL)$.

While Theorem \ref{T:main-periodic} uses conditions which, in concrete examples, can be difficult to check explicitly, like the conditions on the discriminants, the next theorem on the homogeneous case is fully explicit. Recall that for $\perm_\pm(\cdot,\omega)=\perm_\pm(\omega)$ the singular set $\Omega_0$ is defined in \eqref{Omega0}.

\begin{customthm}{B}[constant $\perm_\pm$]\label{T:main-homog}
Let $N=2,3$ and suppose $\perm_\pm(\cdot,\omega)=\perm_\pm(\omega)\in \C$ for each $\omega \in D(\perm)\setminus \Omega_0$. Define
$$
\begin{aligned}
\cN^\mathrm{red}&:=\left\{\omega\in D(\perm)\setminus \Omega_0: \perm_+(\omega)+\perm_-(\omega)\neq 0, \frac{\omega^2\perm_+^2(\omega)}{\perm_+(\omega)+\perm_-(\omega)}, \frac{\omega^2\perm_-^2(\omega)}{\perm_+(\omega)+\perm_-(\omega)}\notin [0,\infty),\right.\\
&\left. \qquad  \frac{\omega^2\perm_+(\omega)\perm_-(\omega)}{\perm_+(\omega)+\perm_-(\omega)}\in (0,\infty)\right\},\\
\cM_\pm^\mathrm{red}&:=\{\omega\in D(\perm)\setminus \Omega_0:\omega^2\perm_\pm(\omega)\in (0,\infty)\}.
\end{aligned}
$$
Then
\begin{itemize}
\item[(a)] (\textbf{resolvent set})  $$D(\perm) \setminus (\cN^\mathrm{red} \cup \cM_+^\mathrm{red}\cup \cM_-^\mathrm{red}\cup \Omega_0)\subset \rho(\cL),$$
\item[(b)] (\textbf{radiation orthogonal to the interface})
$$\cM_+^\mathrm{red}\cup \cM_-^\mathrm{red}\subset \sweyl(\cL),$$
\item[(c)] (\textbf{radiation along the interface})
$$\cN^\mathrm{red}\subset \sweyl(\cL),$$
\item[(d)] (\textbf{eigenvalues})
$$\sigma_p(\cL)\setminus \Omega_0=\emptyset.$$
\end{itemize}
\end{customthm}

The proof is given in Proposition \ref{P:resolv-hom} (for (a)), Lemma \ref{L:Weyl-hom-x1} (for (b)), Lemma \ref{L:Weyl-interf-hom} (for (c)), and Lemma \ref{L:pt-spec-hom} (for (d)).

\medskip

Note that the index ``red" in the notation of the above sets describes the fact that theses sets are ``reduced" by $\Omega_0$ being excluded.
It is interesting to note that for the case of homogeneous media in $\R^N_\pm$ our results provide a full description of the spectrum outside the exceptional set $\Omega_0$, like in the 1D and 2D cases in \cite{BDPW24}. 
Indeed, in the homogeneous case, Theorem \ref{T:main-homog} shows that the set
$\cN^\mathrm{red}\cup \cM_+^\mathrm{red}\cup \cM_-^\mathrm{red}$ comprises the Weyl spectrum (outside $\Omega_0$), while the complement of this set is the resolvent set (outside $\Omega_0$). In addition, we know that there are no eigenvalues outside $\Omega_0$. This is in contrast with the case of $x_1$-dependent $\perm_\pm$, where
there is a potentially larger set in $D(\perm)$, for which we cannot decide whether points lie in the spectrum or the resolvent set.

\brem\label{N-eq-Nred}
 Theorem \ref{T:main-homog} for the interface between two homogeneous media gives the same spectrum (outside $\Omega_0$) of the $N$-dimensional Maxwell operator as Theorem 3 in \cite{BDPW24} for the two dimensional setting. Indeed, the sets $\cM_\pm^\mathrm{red}$ coincide with the sets $M_\pm$ in  Theorem 3 of \cite{BDPW24} and, as we show next, the set
$$\cN:= \left\{ \omega \in D(\perm)\setminus \Omega_0: \exists a\geq 0 \text{ such that } \omega^2\perm_\pm(\omega)\notin[a,\infty), \ \frac{\omega^2 \perm_+(\omega)\perm_-(\omega)}{\perm_+(\omega)+\perm_-(\omega)}=a\right\}
$$
used in \cite{BDPW24} equals $\cN^\mathrm{red}$. 

For the proof of $\cN^\mathrm{red}=\cN$, let us label the three conditions
$\perm_++\perm_-\neq 0, \frac{\omega^2\perm_\pm^2}{\perm_++\perm_-} \notin [0,\infty),$ and $\frac{\omega^2\perm_+\perm_-}{\perm_++\perm_-}\in (0,\infty) $ in $\cN^\mathrm{red}$ by (i), (ii), and (iii) respectively.
The two conditions $\omega^2\perm_\pm\notin[a,\infty)$ and $\frac{\omega^2 \perm_+\perm_-}{\perm_++\perm_-}=a$ in $\cN$ will be labeled by (iv) and (v) respectively.

We first show $\cN^\mathrm{red}\subset \cN.$ If $\omega \in \cN^\mathrm{red}$, then (iii) implies $\omega^2\perm_+\perm_-=a(\perm_++\perm_-)$ with some $a>0$ and hence (v) holds. Also, dividing this equation by $\omega^2\perm_-^2$ and using (ii) yields
\beq\label{E:perm-frac}
\frac{\perm_+}{\perm_-}=a\frac{\perm_++\perm_-}{\omega^2\perm_-^2} \notin [0,\infty).
\eeq 
Equation \eqref{E:perm-frac} can also be reformulated as
$$\omega^2\perm_+=a(1+\tfrac{\perm_+}{\perm_-})$$
and
$$\omega^2\perm_-=a(1+\tfrac{\perm_-}{\perm_+}).$$
Using the fact that $\tfrac{\perm_+}{\perm_-}\notin [0,\infty)$, we get $\omega^2\perm_\pm \notin [a,\infty)$, i.e., (iv).

Now let $\omega \in \cN$. Property (v) and $\omega \notin \Omega_0$ imply $\perm_++\perm_-\neq 0$ and $a>0$ in (v) such that (i) and (iii) hold. To show (ii), first note that (v) can be rewritten as $\omega^2\perm_+=a(1+\tfrac{\perm_+}{\perm_-})$, which lies outside $[a,\infty)$ due to (iv). Equivalently
$\tfrac{\perm_+}{\perm_-}\notin [0,\infty)$.  Finally, multiplying (v) by $\tfrac{\perm_+}{\perm_-}$, we obtain
$$
\frac{\omega^2\perm_+^2}{\perm_++\perm_-} = a\frac{\perm_+}{\perm_-} \notin [0,\infty)$$
and similarly, $\frac{\omega^2\perm_-^2}{\perm_++\perm_-} \notin [0,\infty)$. In other words, (ii) holds and we conclude $\omega \in \cN^\mathrm{red}$.
\erem

\section{Preliminaries}

Before starting with the spectral considerations, we  transform the problem to one that is essentially diagonal (in the new variables $v,w$), introduce some standard notation and prove important estimates for the periodic problem. The periodic setting  provides an important application of our results.

\subsection{Transformation of the Problem}\label{S:trafo-vw}

The next lemma suggests that the problem
\be\label{1}
\widehat{L}(\omega)\bspm u_1 \\ u_2 \\ u_3\espm=\bspm r_1 \\ r_2 \\ r_3\espm \quad (x,k)\in \R\times \R^{N-1}
\ee
in $\widehat{\cD}_\omega$ for arbitrary $r\in \widehat{\cR}$ can be decoupled into two independent problems in the new variables $u_1,v,w$. In fact, in the case $N=2$ the problem is decoupled in the original variables due to the block diagonal structure of $\widehat{L}$. Nevertheless, in order to be able to work with the two cases $N=2$ and $N=3$ simultaneously, we use the new variables $u_1,v,w$ also for $N=2$. 
Note that from \eqref{E:Lktil}, it is easy to see that $u_1$ can be expressed in terms of $v$.

\blem\label{L:equiv-op}
Let $N=2,3$ and $\omega \in D(\perm)$. Consider the operator $\cU$ on $L^2 (\R^N)^3$ given by
$$\left(\cU \bspm u_1 \\ u_2 \\ u_3\espm\right)(x_1,k) = U(k) \bspm u_1 \\ u_2 \\ u_3\espm (x_1,k) \hbox{ for } k\in \R^{N-1}\setminus\{0\},$$
where
$$
U(k)=   \left( \begin{array}{ccc} 1&0&0 \\ 0 & \sign(k) & 0 \\ 0 & 0 &\sign(k) \end{array} \right)   ~ {\rm with }~ k=k_2 \in \R \setminus \{ 0 \} ~ {\rm for }~ N=2,
$$
$$
U(k)=   \left( \begin{array}{ccc} 1&0&0 \\ 0 & \frac{k_2}{|k|} & \frac{k_3}{|k|} \\ 0 & -\frac{k_3}{|k|} &\frac{k_2}{|k|} \end{array} \right)   ~ {\rm with }~ k=(k_2,k_3) \in \R^2 \setminus \{ 0 \} ~ {\rm for }~ N=3,
$$
and
where $|k|=|k_2|$ for $N=2$ and $|k|=\sqrt{k_2^2+k_3^2}$ for $N=3$.
Then $\cU$ is unitary and $\cU: \widehat{\cD}_\omega\to \widetilde{\cD}_\omega$ is bijective, where 
$$
\begin{aligned}
	\widetilde{\cD}_\omega :=& \left\{ (u_1, v, w) \in L^2 (\R^N)^3  : v' - \ri |k| u_1, v'' - \ri |k| u_1',  \ri |k| v' + |k|^2 u_1, w', |k| w, w'' - |k|^2 w \in L^2(\R_\pm^N),\right. \\
	&\left. \text{and} \ \eqref{E:divergence-uvw} \ \text{and}\ \eqref{E:IFC-uvw} \ \text{hold}\right\}.
\end{aligned}
$$
\beq\label{E:divergence-uvw}
(\perm(\cdot,\omega)  u_1)' + \ri |k| \perm(\cdot,\omega) v = 0 \quad  \text{on} \quad \R^N_\pm,
\eeq
\beq\label{E:IFC-uvw}
\begin{aligned}
	\llbracket \perm(\cdot,\omega)  u_1 \rrbracket = \llbracket v \rrbracket = \llbracket v' - \ri |k| u_1 \rrbracket = \llbracket w \rrbracket = \llbracket w' \rrbracket = 0.
\end{aligned}
\eeq
The operator 
\beq\label{E:Lktil}
\widetilde{L}(\omega)= \begin{pmatrix}
	|k|^2 -\omega^2\perm(x_1,\omega) & \ri |k|\pa_{x_1} & 0\\
	\ri |k| \pa_{x_1} & -\pa_{x_1}^2 -\omega^2\perm(x_1,\omega) & 0\\
	0 & 0 & -\pa_{x_1}^2+|k|^2-\omega^2\perm(x_1,\omega)
\end{pmatrix}
\eeq
with the domain  $\widetilde{\cD}_\omega$ satisfies
$ \widetilde{L}= \cU \widehat{L} \cU^*$.
\elem

\bpf
We prove the lemma in detail for $N=3$. The case $N=2$ follows analogously.

For fixed $k\in \R^2 \setminus \{ 0 \}$, it is clear that the matrix $U(k)$ is unitary. Therefore, the matrix operator $\cU$, having only multiplication operators as entries, is unitary on $L^2 (\R^3)^3$.

Let $(u_1, u_2, u_3) \in \widehat{\cD}_\omega$, and set 
\beq\label{eq:Uu}
\bspm u_1 \\ v \\ w\espm=\cU \bspm u_1 \\ u_2 \\ u_3\espm.
\eeq
 Then
\bea
 \widetilde{L}(\omega)\cU \bspm u_1 \\ u_2 \\ u_3\espm &=& \bspm 
	\left(|k|^2 -\omega^2\perm(x_1,\omega)\right) u_1 + \ri |k| v' \\
	\ri |k| u_1'- v''  -\omega^2\perm(x_1,\omega)v\\
-w''+\left(|k|^2-\omega^2\perm(x_1,\omega)\right)w
\espm  \\
&=& \bspm 
	\left(|k|^2 -\omega^2\perm(x_1,\omega)\right) u_1 + \ri k_2u_2'+\ri k_3u_3' \\
	\ri |k| u_1'-\frac{1}{|k|}\left[ k_2u_2''+k_3u_3''  +\omega^2\perm(x_1,\omega)(k_2u_2+k_3u_3)\right]\\
-\frac{1}{|k|}\left[ -k_3u_2''+k_2u_3'' -\left(|k|^2-\omega^2\perm(x_1,\omega)\right)(-k_3u_2+k_2u_3)\right]
\espm \\
&=& \left( \begin{array}{ccc} 1&0&0 \\ 0 & \frac{k_2}{|k|} & \frac{k_3}{|k|} \\ 0 & -\frac{k_3}{|k|} &\frac{k_2}{|k|} \end{array} \right) 
\bspm |k|^2-\omega^2\perm(x_1,\omega) & \ri k_2\pa_{x_1}& \ri k_3 \pa_{x_1} \\ \ri k_2\pa_{x_1} & k_3^2-\omega^2\perm(x_1,\omega)-\pa_{x_1}^2 & -k_2k_3 \\ \ri k_3 \pa_{x_1} & -k_2k_3 & k_2^2-\omega^2\perm(x_1,\omega)-\pa_{x_1}^2\espm
\bspm u_1 \\ u_2 \\ u_3\espm\\
& = & \cU\widehat{L}(\omega) \bspm u_1 \\ u_2 \\ u_3\espm,
\eea
so formally $\widetilde{L}\cU=\cU\widehat{L}$, as claimed. To prove the full operator identity, it remains to check that $\cU: \widehat{\cD}_\omega\to \widetilde{\cD}_\omega$ is bijective. 

Let $(u_1, u_2, u_3) \in \widehat{\cD}_\omega$. The property $(\perm  u_1)' + \ri \perm (k_2u_2+k_3u_3)=0$ contained in $\widehat{\cD}_\omega$ gives, using \eqref{eq:Uu},
\be\label{8a}
(\perm  u_1)' + \ri |k|\perm v = 0 ~ {\rm on }\ \R_{\pm}.
\ee
The conditions $u_2, u_3 \in L^2 (\R_{\pm}^3), \llbracket u_2 \rrbracket = \llbracket u_3 \rrbracket = 0$ imply, using \eqref{eq:Uu} and the boundedness of $\frac{k_2}{|k|}$ and $\frac{k_3}{|k|}$,
\be\label{9}
v,w \in L^2 (\R_{\pm}^3), \llbracket v \rrbracket = \llbracket w \rrbracket = 0.
\ee
In the same way, the conditions $\ri k_2 u_1 - u_2'$, $\ri k_3 u_1 - u_3' \in L^2 (\R_{\pm}^3), \llbracket \ri k_2 u_1 - u_2' \rrbracket = \llbracket \ri k_3 u_1 - u_3' \rrbracket = 0$
amount to
\be\label{10}
\ri |k| u_1 - v', w' \in L^2 (\R^3_{\pm}), \llbracket \ri |k| u_1 - v' \rrbracket = \llbracket w' \rrbracket = 0,
\ee
and $\ri k_2 u_1' - u_2'' + k_3^2 u_2 - k_2 k_3 u_3, \ri k_3 u_1' - u_3'' - k_2 k_3 u_2 + k_2^2 u_3 \in L^2 (\R^3_{\pm})$ to
\be\label{11}
\ri |k| u_1' - v'', - w'' + |k|^2 w \in L^2 (\R^3_{\pm}).
\ee
The remaining conditions $k_2u_3 - k_3 u_2, |k|^2 u_1 + \ri k_2 u_2' + \ri k_3 u_3' \in L^2 (\R^3_{\pm})$ give
\be\label{12}
|k|w, |k|^2 u_1 + \ri |k| v' \in L^2 (\R^3_{\pm}).
\ee
The conditions $u_1 \in L^2 (\R^3_{\pm}), \llbracket\perm u_1 \rrbracket = 0$ remain unchanged. \eqref{8a} - \eqref{12} imply $(u_1, v, w) \in \widetilde{\mathcal{D}}_\omega$.

Conversely, let $(u_1, v, w) \in \widetilde{\mathcal{D}}_\omega$. Hence \eqref{8a} - \eqref{12} are true. We define $u_2, u_3$ by inversion of \eqref{eq:Uu}, i.e.
\be\label{13}
{u_2 \choose u_3} := \frac{1}{|k|} \left( \begin{array}{cc} k_2 &- k_3\\k_3 &k_2 \end{array}\right) {v \choose w}.
\ee
Hence \eqref{9} implies
\bea
u_2, u_3 \in L^2 (\R^3_{\pm}), \llbracket u_2 \rrbracket = \llbracket u_3 \rrbracket = 0.
\eea
By \eqref{12},
\bea
k_2 u_3 - k_3 u_2 = |k| w \in L^2 (\R^3_{\pm}),
\eea
and  by \eqref{10}
\bea
\left( \begin{array}{c} \ri k_2 u_1 - u_2'\\\ri k_3 u_1 - u_3' \end{array}\right) = \frac{1}{|k|}  \left( \begin{array}{cc} k_2 &- k_3\\k_3 &k_2 \end{array}\right)  \left( \begin{array}{c} \ri |k| u_1 - v'\\-w' \end{array}\right) \in L^2 (\R^3_{\pm})^2,
\eea
and also $\llbracket \ri k_2 u_1 - u_2' \rrbracket = \llbracket \ri k_3 u_1 - u_3' \rrbracket = 0$.\\

Similarly,
\bea
\left( \begin{array}{c} \ri k_2 u_1' - u_2'' + k_3^2 u_2 - k_2 k_3 u_3\\\ri k_3 u_1' - u_3'' - k_2k_3 u_2 + k_2^2 u_3 \end{array}\right) = \frac{1}{|k|} \left( \begin{array}{cc} k_2 & -k_3\\k_3 & k_2 \end{array}\right) \left( \begin{array}{c} \ri |k| u_1' - v''\\-w'' + |k|^2 w \end{array}\right)\in L^2 (\R^3_{\pm})^2
\eea
by \eqref{11}. Furthermore,
\bea
|k|^2 u_1 + \ri k_2 u_2' + \ri k_3 u_3' =|k|^2 u_1 + \ri |k| v' \in L^2 (\R^3_{\pm})
\eea
by \eqref{12}. Finally, \eqref{8a} implies
\bea
(\perm u_1)' + \ri\perm  (k_2 u_2 + k_3 u_3) = (\perm  u_1)' + \ri |k| \perm v = 0 {\rm ~on~} \R^3_{\pm}.
\eea
Altogether, $(u_1, u_2, u_3) \in \widehat{\cD}_\omega$, which completes the proof.
\epf

 Next, we reformulate the transformed equations $\widetilde{L} (u_1,v,w)^\top= \cU r$
 so that both $v$ and $w$ satisfy a Sturm-Liouville type equation. This turns out to be useful in the later analysis.

\blem\label{lem:equiv} Suppose that $\omega \in D(\perm)\setminus \Omega$. Problem \eqref{1} for $u\in \widehat{\cD}_\omega$ is equivalent to the following ``almost'' decoupled system for 
$(u_1,v,w)^\top = \cU (u_1,u_2,u_3)^\top \in \widetilde{\cD}_\omega$ given by
\be\label{3}
\begin{array}{D{r} D{l} D{l}}
- \left(  \frac{\perm }{|k|^2 - \omega^2\perm } v' \right) ' + \perm v &=  \frac{\ri |k|\perm '}{(|k|^2-\omega^2\perm )^2} r_1 + \frac{\sign(k_2)\perm }{(|k|^2-\omega^2\perm )} r_2  ~ &{\rm if }~ N=2,\\
- \left(  \frac{\perm }{|k|^2 - \omega^2\perm } v' \right) ' + \perm v &=  \frac{\ri |k|\perm '}{(|k|^2-\omega^2\perm )^2} r_1 + \frac{\perm }{|k| (|k|^2-\omega^2\perm )} (k_2r_2 + k_3r_3) ~ &{\rm if }~ N=3,
\end{array}
\ee
\be\label{4}
\begin{array}{D{r} D{l} D{l}}
\\
	-w'' + (|k|^2 - \omega^2\perm ) w &= \sign(k_2)r_3   ~ &{\rm if }~ N=2,\\
	-w'' + (|k|^2 - \omega^2\perm ) w &= \frac{1}{|k|} (k_2r_3 - k_3r_2)  ~ &{\rm if }~ N=3,
\end{array}
\ee
\be\label{5}
u_1 = -\frac{ \ri |k|}{|k|^2 - \omega^2\perm } v' + \frac{1}{|k|^2-\omega^2\perm   }r_1.
\ee
\elem

\bpf The proof is carried out for the three dimensional setting. The case $N=2$ is obtained from this proof by simply setting $k_3=0$ and replacing $\R^3$ by $\R^2$. Let $(u_1, u_2, u_3) \in \widehat{\cD}_\omega$ denote a solution of \eqref{1}.
The first equation in \eqref{1} implies
$$u_1 = \frac{r_1 - \ri k_2 u_2' - \ri k_3 u_3'}{|k|^2- \omega^2\perm } = \frac{r_1 - \ri |k| v'}{|k|^2-\omega^2\perm },$$ i.e. \eqref{5}, and
\be\label{6}
u_1' =  \frac{r_1' - \ri |k| v''}{|k|^2- \omega^2\perm } +  \frac{\omega^2\perm '}{(|k|^2-\omega^2\perm )^2} (r_1 - \ri |k| v').
\ee
Using Lemma \ref{L:equiv-op}, the second and third equation in \eqref{1} give
\begin{align}\label{8}
	\begin{split}
		\ri |k| u_1' - (v'' + \omega^2\perm  v) = \frac{1}{|k|} (k_2r_2 + k_3 r_3),\\
		- w'' + (|k|^2 - \omega^2\perm ) w = \frac{1}{|k|} (k_2 r_3 - k_3 r_2).
	\end{split}
\end{align}
The second equation in \eqref{8} is just \eqref{4}. Inserting \eqref{6} into the first equation gives
\ben 
 &&\frac{\omega^2}{|k|^2 - \omega^2\perm } \left\{ \perm  v'' + \frac{|k|^2 \perm '}{|k|^2 - \omega^2\perm } v' - \perm (|k|^2 - \omega^2\perm ) v \right\} \nonumber \\
&&= - \frac{\ri |k|}{|k|^2 - \omega^2\perm } r_1' 
- \frac{\ri |k| \omega^2\perm '}{(|k|^2 - \omega^2\perm )^2} r_1 + \frac{1}{|k|} (k_2 r_2 + k_3 r_3). 
\een
Using \eqref{E:R-Four}, we find equation \eqref{3}.

Conversely, let $(u_1, v, w) \in \widetilde{\mathcal{D}}_\omega$ satisfy \eqref{3}, \eqref{4}, \eqref{5}. By Lemma \ref{L:equiv-op}, equation \eqref{1} is equivalent to 
\be\label{10tilde}
\widetilde{L}(\omega)\bspm u_1 \\ v \\ w\espm=\cU \bspm r_1 \\ r_2 \\ r_3\espm= \bspm r_1 \\ \frac{1}{|k|}(k_2r_2+k_3r_3) \\ \frac{1}{|k|}(-k_3r_2+k_2r_3)\espm.
\ee
The third equation in \eqref{10tilde} is simply \eqref{4}. From \eqref{5}, we deduce $(|k|^2 - \omega^2\perm) u_1 = - \ri |k| v' + r_1$, giving the first equation in \eqref{10tilde}.
Differentiating \eqref{5}, we can replace $u_1'$ in the second equation of \eqref{10tilde}. Again using \eqref{E:R-Four}, we recover the first equation in \eqref{8}. This is precisely the second equation in \eqref{10tilde}.
\epf

In a further transformation, we will now rewrite the expression involving $v$ in equation \eqref{3}. This will be useful for later calculations, as it will allow us to treat the $v$- and $w$-components simultaneously. However, we caution that this is only a transformation of the differential equation, and we will not consider how the domain of $\widetilde{L}(\omega)$ is transformed. 
\blem\label{lem:v-eqn}
Suppose that $\omega \in D(\perm)\setminus \Omega$.
On $\R_-\cup\R_+$, the equation
\beq\label{E:v}
- \left(  \frac{\perm }{|k|^2 - \omega^2\perm } v' \right) ' + \perm v=0
\eeq
can be rewritten as
\beq\label{E:z}
-z''+\left(|k|^2-\omega^2\perm+\frac{3\perm'^2}{4\perm^2}-\frac{\perm''}{2\perm}\right)z=0,
\eeq
where $z=\frac{\sqrt{\omega^2\perm}}{|k|^2-\omega^2\perm } v'$. Moreover, solutions $v$ of \eqref{E:v} can be recovered from solutions $z$ of \eqref{E:z} via
\beq\label{E:vz}
v=\frac{1}{\sqrt{\omega^2\perm}}z'+\frac{\perm'}{2\perm\sqrt{\omega^2\perm}}z. 
\eeq
\elem

\begin{remark}\label{rem:root}
We note that in general, for $z_1,z_2\in\C$, $\sqrt{z_1z_2}\neq \sqrt{z_1}\sqrt{z_2}$. However, for a complex-valued function $f$, we still have 
\beq
\sqrt{f}\ '=\frac{f'}{2\sqrt{f}}\quad\hbox{and}\quad  \left(\frac{1}{\sqrt{f}}\right)'=-\frac{f'}{2f\sqrt{f}}.
\eeq
This can be seen by differentiating $\left(\sqrt{f}\right)^2=f$ and $\frac{1}{\sqrt{f}}\sqrt{f}=1$, respectively.
\end{remark}

\bpf
To shorten the notation we define $$W_\pm:= \omega^2\perm_\pm.$$
For a solution $v$ of (\ref{E:v}),
let
$y = \frac{W_\pm}{|k|^2 - W_\pm}\, v'$.
Then $y' = W_\pm v$ and
\begin{equation}\label{eq:yeq}
	y'' = W_\pm'\,v + W_\pm\,v' = \frac{W_\pm'}{W_\pm}\,y' + (|k|^2 - W_\pm)\,y.
\end{equation}
Due to the definition of $z$, we have $y = \sqrt{W_\pm}\,z$, and (using Remark \ref{rem:root})
\begin{align*}
	y' &= \frac{W_\pm'}{2\sqrt{W_\pm}}\,z + \sqrt{W_\pm}\,z',
	\\
	y'' &= \left(\frac{W_\pm''}{2\sqrt{W_\pm}} - \frac{W_\pm'^2}{4W_\pm\sqrt{W_\pm}}\right) z + \frac{W_\pm'}{\sqrt{W_\pm}}\,z' + \sqrt{W_\pm}\,z''
\end{align*}
and from equation $(\ref{eq:yeq})$ also
\begin{equation*}
	y'' = \frac{W_\pm'}{W_\pm} \left(\frac{W_\pm'}{2\sqrt{W_\pm}}\,z + \sqrt{W_\pm}\,z'\right) + (|k|^2 - W_\pm)\,\sqrt{W_\pm}\,z.
\end{equation*}
Combining the last two equations gives
\begin{equation*}
	-z''+\left(|k|^2-W_\pm+\frac{3W_\pm'^2}{4W_\pm^2}-\frac{W_\pm''}{2W_\pm}\right)z=0.
\end{equation*}
This is \eqref{E:z}.

Conversely, for a solution $z$ of \eqref{E:z}, let $v$ be defined by \eqref{E:vz}. We calculate
\ben
 v' &=& \frac{1}{\sqrt{W_\pm}}z''+ \left(\frac{W_\pm'}{2W_\pm\sqrt{W_\pm}}\right)'\,z\\
	&=& \frac{1}{\sqrt{W_\pm}}\left[ |k|^2-W_\pm+\frac{3(W_\pm')^2}{4W_\pm^2}-\frac{W_\pm''}{2W_\pm}\right] z +\left[ \frac{W_\pm''}{2W_\pm\sqrt{W_\pm}}-\frac{3(W_\pm')^2}{4W_\pm^2\sqrt{W_\pm}} \right] z\\
	&=& \frac{|k|^2 - W_\pm}{\sqrt{W_\pm}}\,z,
\een
which gives 
\ben
 \left(  \frac{\perm_\pm }{|k|^2 - \omega^2\perm_\pm } v' \right) ' &=& \left(\frac{\perm_\pm}{\sqrt{W_\pm}} z\right)'\ =\ \frac{1}{\omega^2}\left(\sqrt{W_\pm} z\right)'\ =\ \frac{1}{\omega^2}\left(\frac{W_\pm'}{2\sqrt{W_\pm}}z + \sqrt{W_\pm} z'\right)\\
&=&  \frac{W_\pm}{\omega^2}\left(\frac{W_\pm'}{2W_\pm\sqrt{W_\pm}}z + \frac{1}{\sqrt{W_\pm}} z'\right)\ =\ \perm_\pm v,
\een
i.e., $v$ solves (\ref{E:v}).
\epf

\subsection{Notation and Fundamental Estimates for Periodic Media}\label{S:per-estimates}

For the case of periodic $\perm_\pm(\cdot,\omega)$  we consider a periodic $\perm_p:\R\to\C$, which is the periodic extension of a periodic $\perm_+(\cdot,\omega)$ or $\perm_-(\cdot,\omega)$ to $\R$. Then many of the spectral results and related estimates can be provided in a more explicit form. Hence, we study the periodic case as an important example.

As we have shown in Lemma \ref{lem:equiv} and Lemma \ref{lem:v-eqn}, the Maxwell spectral problem can be transformed and decoupled into spectral problems for Sturm-Liouville or Schr\"odinger operators. In this section we provide some well-known general results from the Floquet theory of
Sturm-Liouville equations with periodic coefficients (see also
\cite{MW2004,Eastham1973,BES2012}) and some fundamental estimates in the
case of a Schr\"odinger-type equation which may be of more general interest. 
These estimates do not seem to appear in the literature. 

Equations (\ref{3}) and (\ref{4}) with $r = 0$ are of the Sturm-Liouville type
\begin{equation}\label{eq:SL}
	-(P y')' + Q y = 0.
\end{equation}
In the case of (\ref{3}) we have
\beq\label{E:pv-qv}
P=P_v:=\frac{\perm_p}{|k|^2-\omega^2\perm_p}, \quad Q= Q_v:=\perm_p
\eeq
and in the case of (\ref{4}) we have
\beq\label{E:pw-qw}
P=P_w:=1, \quad Q= Q_w:=|k|^2-\omega^2\perm_p.
\eeq
Equation \eqref{eq:SL} can be rewritten as a first-order system for the vector-valued function
$(y, P y')^\top$.
The canonical fundamental system of this system is the $2\times 2$ matrix-valued solution of the initial value problem
\begin{equation*}
	\Phi' = \begin{pmatrix} 0 & \frac 1 P \\ Q & 0 \end{pmatrix} \Phi, \qquad
	\Phi(0) = \begin{pmatrix} 1 & 0 \\ 0 & 1 \end{pmatrix}.
\end{equation*}
When the coefficients of the equation have period $a > 0$, the growth
of solutions of the equation can be characterised in terms of
the monodromy matrix $\Phi(a)$. This matrix has (Wronskian) determinant 1 and
therefore its eigenvalues are determined by its trace, i.e., the
\textbf{discriminant} $D := \mathop{\rm Tr} \Phi(a)$.
If $D \notin \{-2, 2\}$, then $\Phi(a)$ has two distinct eigenvalues
$e^{\kappa a}$ and $e^{-\kappa a}$, where $\kappa \in \mathbb C \setminus \{0\}$ has
$\Real \kappa \ge 0$ and,
without loss of generality, $\Imag \kappa \in (-\pi/a, \pi/a]$.
If $\Psi^{(1)}, \Psi^{(2)}\in \mathbb C^2\setminus\{0\}$ are corresponding eigenvectors, then
equation (\ref{eq:SL}) has the linearly independent solutions \textbf{(Floquet solutions)}
$\psi_1 := (\Phi \Psi^{(1)})_1$ and $\psi_2 := (\Phi \Psi^{(2)})_1$,
which can also be written in the form
\begin{equation}\label{eq:6}
	\psi_1(x) = e^{\kappa x}\,p_1(x), \qquad \psi_2(x) = e^{-\kappa x}\,p_2(x) \qquad (x \in \mathbb R)
\end{equation}
with $a$-periodic functions $p_1$ and $p_2$.
Like the eigenvectors, the functions $p_1$ and $p_2$ are only fixed up to a multiplicative complex constant.

If $D \in (-2, 2)$, then $\kappa$ is purely imaginary and the Floquet solutions
(and consequently all solutions) are globally bounded; this is the case of
\textbf{stability}.
If $D \in \mathbb C \setminus [-2, 2]$, then $\Real \kappa > 0$ and $\psi_1$ is
exponentially small at $-\infty$ but unbounded at $\infty$, $\psi_2$ is
unbounded at $-\infty$ and exponentially small at $\infty$, so all non-trivial solutions
of (\ref{eq:SL}) on $\R$ are unbounded; this is the case of \textbf{instability}.
If $D \in \{-2, 2\}$, then $\Phi(a)$ only has one eigenvalue, either 1 or -1. If it has
geometric multiplicity 2,
then $\Phi(a)$ is the unit matrix or its negative and all solutions of
(\ref{eq:SL}) are $2a$-periodic and globally bounded (stability).
If the eigenvalue has geometric multiplicity 1, then (\ref{eq:SL}) has one
bounded solution (up to multiples) while all other solutions are unbounded;
this is the case of \textbf{conditional stability}.

We now give some fundamental estimates for the Schr\"odinger-type equation,
i.e., (\ref{eq:SL}) with $P\equiv 1$. Note that, with $r = 0$, equation (\ref{4})
is of this form; we have seen in Lemma \ref{lem:v-eqn} that equation (\ref{3}) can also be transformed
into this type.
In this spirit, consider the differential equation 
\beq\label{E:lV}
-u'' + (l^2 - V)\,u = 0,
\eeq
where $l\geq 0$ and 
$V : \mathbb R \rightarrow \mathbb C$ is bounded and periodic of period $a>0$.
Then $\kappa$ and the periodic functions $p_1, p_2$ in equation (\ref{eq:6})
depend on $l$. In the following, we give uniform bounds on $\kappa$, and on $p_1, p_2$ and their
derivatives for suitable $l$-dependent choice of the multiplicative constants
for these functions and for sufficiently large $l$.

\begin{lem}\label{lem:LemmaA} With $\kappa$ defined as above for equation \eqref{E:lV},
	there exist $C > 0$ and $l_0 > 0$
	such that
	\begin{equation}
		|\kappa(l) - l| \le \frac C l \qquad (l \ge l_0).
	\end{equation}
\end{lem}
\begin{proof}
	Consider the solution $u$ of equation \eqref{E:lV} with initial values
	$u(0) = 1$, $u'(0) = l$, and set
	$w(x) = e^{-l x}\,u(x) - 1$ $(x \in [0,a])$. Then
	$w(0) = w'(0) = 0$ and 
	$w$ satisfies the differential equation
	\begin{equation}
		w'' + 2 l w' = -V\,(1 + w).
	\end{equation}
	By the variation of constants formula, we obtain
	\begin{equation}\label{eq:12}
		\left(\begin{matrix}w(x) \\ w'(x)\end{matrix}\right) = \frac 1{2l} \int_0^x V(t)\,(1 + w(t)) \left(\begin{matrix} -1 + e^{-2l(x-t)} \\ -2l e^{-2l(x-t)}\end{matrix}\right) d t,
	\end{equation}
	which implies
	\begin{equation*}
		|w(x)| \le \frac 1{2l}\,\|V\|_\infty \, (1 + \|w\|_\infty) \int_0^x (1 - e^{-2l(x-t)})\, d t
		\qquad (x \in [0, a])
	\end{equation*}
	and hence, with $\|\cdot\|_\infty:=\|\cdot\|_{L^\infty([0,a])}$,
	\begin{equation*}
		\|w\|_\infty \le \frac 1{2l}\,\|V\|_\infty \, (1 + \|w\|_\infty)\,a.
	\end{equation*}
	Therefore we find for all $l > \|V\|_\infty a$ the estimate
	\begin{equation}\label{eq:13}
		\|w\|_\infty \leq \left(1-\frac{a}{2l}\|V\|_\infty \right)^{-1} \frac{a}{2l} \|V\|_\infty \leq \frac a l\,\|V\|_\infty  < 1.
	\end{equation}
	Furthermore, equation (\ref{eq:12}) also gives for all $l > \|V\|_\infty a$
	the estimate
	\begin{align}
		\|w'\|_\infty &\le \|V\|_\infty\,(1 + \|w\|_\infty) \sup_{x\in (0,a)}\int_0^x e^{-2l(x-t)}\,d t
		\nonumber
		\\
		&\le \|V\|_\infty\,(1 + 1)\, \frac 1{2l}\,(1 - e^{-2la})
		\le \frac 1 l\,\|V\|_\infty.
		\label{eq:14}
	\end{align}
	We now use d'Alembert's formula to find a second solution of equation \eqref{E:lV} that is linearly independent of $u$.
	Note that for $l > \|V\|_\infty a$, the solution $u(x) = e^{lx}(1 + w(x))$
	$(x \in [0, a])$ has no zeros due to the bound (\ref{eq:13}), so
	\begin{equation*}
		\phi_2(x) := u(x) \int_0^x \frac 1{u(t)^2}\, d t \qquad (x \in [0, a])
	\end{equation*}
	defines such a second solution. It has derivative
	\begin{equation}\label{eq:16}
		\phi_2'(x) = u'(x) \int_0^x \frac 1{u(t)^2}\,d t + \frac 1{u(x)}
		\qquad (x \in [0, a])
	\end{equation}
	and hence initial values
	$\phi_2(0) = 0,$ $\phi_2'(0) = 1.$
	Setting
	$\phi_1 := u - l\,\phi_2$,
	we then see that
	\begin{equation*}
		\Phi = \left(\begin{matrix} \phi_1 & \phi_2 \\ \phi_1' & \phi_2' \end{matrix}\right)
	\end{equation*}
	is the canonical fundamental system of equation \eqref{E:lV}, as
	$\Phi(0) = \mathbb I$.
	
	We now consider $\phi_1$ and $\phi_2'$ in order to obtain an estimate for
	the discriminant $D=\tr \Phi(a)$.
	In the following, let $l_0 > \|V\|_\infty\,\max\{1, a\}$.
	Note that $\phi_1$ can be written as
	\begin{equation*}
		\phi_1(x) = e^{lx}\,(1 + w(x))\,\left(1 - l \int_0^x \frac{e^{-2lt}}{(1 + w(t))^2}\, d t \right)
		\qquad (x \in [0, a]).
	\end{equation*}
	Since, for $l \ge l_0$,
	\begin{equation*}
		\left| 1 - \frac 1 {(1+w(t))^2} \right| \le \frac \c l
	\end{equation*}
	by the estimate (\ref{eq:13}) (where the constant is independent of $l$),
	we find
	\begin{equation*}
		1 - l \int_0^x \frac{e^{-2lt}}{(1+w(t))^2}\, d t
		= 1 - l \int_0^x e^{-2lt}\,d t + R_1
		= \frac 1 2\,(1 + e^{-2lx}) + R_1
	\end{equation*}
	with remainder term satisfying
	$|R_1| \le \frac \c l\,\frac 1 2\,(1 + e^{-2lx})$.
	Hence, using again the estimate (\ref{eq:13}),
	\begin{equation*}
		\phi_1(x) = (\cosh lx + e^{lx}\,R_1) (1 + w(x)) = \cosh lx + R_2
		\qquad (x \in [0,a])
	\end{equation*}
	with remainder term $|R_2| \le \frac \c l\,\cosh lx$.
	Furthermore, by equation (\ref{eq:16}),
	\begin{align*}
		\phi_2'(x) &= e^{lx}\left(l\,(1 + w(x)) + w'(x) \right) \int_0^x \frac{e^{-2lt}}{(1+w(t))^2}\,d t + \frac{e^{-lx}}{1 + w(x)}
		\\
		&= (1 + w(x) + \frac 1 l\,w'(x)) (\sinh lx + l\,e^{lx}\,R_3) + e^{-lx} + R_4
		\\
		&= \sinh lx + e^{-lx} + R_5
		= \cosh l x + R_5
		\qquad (x \in[0,a])
	\end{align*}
	for $l \ge l_0$, where
	$|R_3| \le \frac \c l\,\frac 1{2l}\,(1 - e^{-2lx})$,
	$|R_4| \le \frac \c l\,e^{-l x}$ by (\ref{eq:13}).
	Again using (\ref{eq:13}) and (\ref{eq:14}),
	$|R_5| \le \frac \c l\,\sinh lx + \frac \c l\,e^{-lx} \le \frac \c l\,\cosh lx$.
	Therefore we obtain for $l \ge l_0$
	\begin{equation*}
		D  = \phi_1(a) + \phi_2'(a) = 2 \cosh la\, (1 + R_6)
	\end{equation*}
	with $|R_6| \le \frac \c l$.
	
	Since, moreover, $\Phi(a)$ has the two eigenvalues $e^{\pm \kappa(l)a}$, we get that for $l$ large enough,
	\begin{align*}
		e^{\kappa(l) a} &= \frac 1 2 \,D + \sqrt{\frac 1 4\,D^2 - 1}
		= (1 + R_6)\,\cosh l a + \sqrt{(1+R_6)^2\,\cosh^2 l a - 1}
		\\
		&= (1 + R_6)\,\cosh l a + \sqrt{(1+R_6)^2\,\sinh^2 l a + 2 R_6 + R_6^2}
		\\
		&= (1 + R_6)\,\cosh l a + (1 + R_6)\,\sinh l a \,\sqrt{1 + R_7}
	\end{align*}
	with $|R_7| \le \frac \c l$. Then $R_8 := \sqrt{1 + R_7} - 1$ also satisfies
	$|R_8| \le \frac \c l$ and we find
	\begin{equation*}
		e^{\kappa(l) a} = e^{l a}\,(1 + R_6) + \sinh l a\,(1 + R_6)\,R_8
		= e^{l a}\,(1 + R_9)
	\end{equation*}
	with $|R_9| \le \frac \c l$.
	Consequently, denoting with $\log$ the principal branch of the logarithm
	function, we have
	\begin{equation*}
		\kappa(l) a = l a + \log(1 + R_9) + 2 \pi n \ri
	\end{equation*}
	with some $n \in \mathbb Z$. Since we assumed that $\Imag \kappa(l) a \in (-\pi,\pi]$ and have found that
	$|\Imag \log (1 + R_9)| \le |\log (1 + R_9)| \le \frac \c l$,
	we see that $n = 0$ and conclude that
	\begin{equation*}
		|\kappa(l) - l| = \frac 1 a\,|\log(1 + R_9)| \le \frac \c l
	\end{equation*}
	for $l \ge l_0$ with some $l_0>0$.
\end{proof}

\begin{lem}\label{lem:LemmaB}
	Consider the Floquet solutions (\ref{eq:6}) of equation \eqref{E:lV}. 
	There exists $l_0 > 0$ such that for $l \ge l_0$ the functions $p_1, p_2$  can be chosen so that
	\begin{equation}\label{eq:26}
		\frac 1 a \int_0^a p_1(x)\,d x = \frac 1 a \int_0^a p_2(x)\,d x = 1;
	\end{equation}
	furthermore, there is a constant $C > 0$ such that
	\begin{equation*}
		\|p_1 - 1\|_\infty, \|p_2 - 1\|_\infty, \|p_1'\|_\infty, \|p_2'\|_\infty \le \frac C l
	\end{equation*}
	for all $l \ge l_0$.
\end{lem}
\begin{proof}
	Set $m_1 := \kappa(l)$, $m_2 := -\kappa(l)$, and let $j \in \{1,2\}$.
	Differentiating the Floquet solution $\psi_j$ in equation (\ref{eq:6}) twice and
	using the differential equation \eqref{E:lV}, we find that $p_j$ satisfies
	the differential equation
	\begin{equation*}
		p_j'' + 2m_j p_j' + (m_j^2 - l^2 + V)\,p_j = 0
	\end{equation*}
	with periodic boundary conditions.
	Let
	\begin{equation*}
		M := \frac 1 a \int_0^a p_j(x)\,d x.
	\end{equation*}
	Then $p_j(x) = M + q(x)$ $(x \in [0, a])$, where
	\begin{equation}\label{eq:28}
		\int_0^a q(x)\,d x = 0.
	\end{equation}
	The function $q$ satisfies the differential equation
	\begin{equation}\label{eq:29}
		q'' + 2 m_j q' + (m_j^2 - l^2)\,q = -V\,q - (m_j^2 - l^2 + V)\,M,
	\end{equation}
	again with periodic boundary conditions.
	Integrating this equation over the interval $[0, a]$ and using the periodic
	boundary conditions for $q$, we find that
	$r := -V q - (m_j^2 - l^2 + V) M$
	has integral
	\begin{equation*}
		\int_0^a r(x)\,d x = 0.
	\end{equation*}
	The homogeneous equation where we set the right-hand side of equation (\ref{eq:29}) equal to 0 has the fundamental system
	$e^{(-m_j+l) x}$, $e^{(-m_j-l)x}$ $(x \in [0, a])$, and we can solve the
	inhomogeneous equation (\ref{eq:29}) by the variation of constants method.
	The general solution for $m_j^2-l^2 \neq 0$ takes the form
	\begin{equation*}
		q(x) = \frac 1{2l} \int_0^x \left(e^{(-m_j+l)(x-t)} - e^{(-m_j-l)(x-t)}\right)\,r(t)\,d t + c_1\,e^{(-m_j+l)x} + c_2\,e^{(-m_j-l)x}, \quad x \in [0, a],
	\end{equation*}
	and determining $c_1$ and $c_2$ from the conditions
	$q(a) = q(0)$ and equation (\ref{eq:28}), we find
	for all $x \in [0, a]$
	\begin{align}
		q(x) &= \frac 1{2l} \int_0^x\left(-\frac{e^{(l-m_j)(x-t)}-1}{e^{(l-m_j)a}-1} +\frac{e^{(-l-m_j)(x-t)}-1}{e^{(-l-m_j)a}-1}\right)r(t)\,d t
		\nonumber
		\\
		&\quad + \frac 1{2l} \int_x^a\left(\frac{e^{(-l+m_j)(t-x)}-1}{e^{(-l+m_j)a}-1} - \frac{e^{(l+m_j)(t-x)}-1}{e^{(l+m_j)a}-1}\right) r(t)\,d t;
		\label{eq:32}
	\end{align}
	this solution then also satisfies the boundary condition $q'(a) = q'(0)$.
	If $m_j = l$ or $m_j = -l$, then the formula for $q$ is still valid provided the corresponding quotient in the integrands is replaced with the respective limit $\frac{x-t}a$ or $\frac{t-x}a$.
	
	In order to estimate the four quotients appearing in the formula for $q$, we
	observe that they all are of the form
	\begin{equation*}
		\frac{e^{\lambda\zeta}-1}{e^\zeta - 1}
		= \frac{\int_0^\lambda e^{\zeta s}\,d s}{\int_0^1 e^{\zeta s}\,d s},
	\end{equation*}
	where $\zeta \in \{\pm(l-m_j)a, \pm(l+m_j)a\}$ and $\lambda = \frac{|x-t|}a \in [0, 1]$.
	Abbreviating $\alpha := \Real \zeta$, $\beta := \Imag \zeta$, we find
	\begin{equation}\label{eq:34}
		\left|\int_0^\lambda e^{\zeta s}\,d s \right|
		\le \int_0^\lambda e^{\alpha s}\,d s \le \int_0^1 e^{\alpha s}\,d s,
	\end{equation}
	and
	\begin{align*}
		\left|\int_0^1 e^{\zeta s}\,d s\right|
		&= \left|\int_0^1 e^{\alpha s}\,d s + \int_0^1 e^{\alpha s}\left(e^{\ri\beta s}-1\right)\,d s\right|
		\\
		&\ge \int_0^1 e^{\alpha s}\,d s - \int_0^1 e^{\alpha s}\,\left|2i\,e^{\ri \beta s/2}\,\sin \frac{\beta s}2 \right|\,d s
		\\
		&\ge \int_0^1 e^{\alpha s}\,d s - 2 \int_0^1 e^{\alpha s}\,\sin \frac{|\beta|\,s}2\,d s.
	\end{align*}
	Since
	\begin{equation*}
		|\beta| = a|\Imag(\pm l \pm m_j)| = a|\Imag m_j| \le \frac \c l
	\end{equation*}
	for sufficiently large $l$ by Lemma \ref{lem:LemmaA},
	\begin{equation*}
		2 \sin \frac{|\beta|\,s}2 \le \frac \c l \qquad (s \in [0, 1])
	\end{equation*}
	and therefore
	there is $l_0 > 0$ such that
	\begin{equation}\label{eq:35}
		\left|\int_0^1 e^{\zeta s}\,d s\right| \ge \left(1 - \frac \c l\right) \int_0^1 e^{\alpha s}\,d s
		\ge \frac 1 2 \int_0^1 e^{\alpha s}\,d s
	\end{equation}
	for all $l \ge l_0$.
	We thus obtain the following bound for $q$,
	\begin{align*}
		\|q\|_\infty &\le \sup_{x \in [0, a]} \frac 1{2l} \left(\int_0^x 4\,|r(t)|\,d t + \int_x^a 4\,|r(t)|\,d t\right)
		\le \frac {2a}l\,\|r\|_\infty
		\\
		&\le \frac {2a}l\,\left(\|V\|_\infty \|q\|_\infty + |m_j^2 - l^2|\,M + \|V\|_\infty\,M\right)
	\end{align*}
	and hence
	\begin{equation}\label{eq:38}
		\|q\|_\infty \le \frac{\frac {2a}l\,M}{1 - \frac {2a}l \|V\|_\infty}\,\left(|m_j^2 - l^2| + \|V\|_\infty \right)
	\end{equation}
	for $l \ge l_0$.
	Furthermore, Lemma \ref{lem:LemmaA} gives
	\begin{align}
		|m_j^2 - l^2| &= |\kappa(l)^2 - l^2| = |\kappa(l) - l|\,|\kappa(l) - l + 2l|
		\nonumber
		\\
		&\le \frac \c l \left(\frac \c l + 2l \right)
		\le \c,
		\label{eq:39}
	\end{align}
	so by the estimate (\ref{eq:38}) we have
	\beq\label{E:q-Linf-est}
	\|q\|_\infty \le M\,\frac \c l \quad \text{for} \quad l \ge l_0.
	\eeq
	
	We see that $M = 0$ would imply that $q \equiv 0$ and also $p_j \equiv 0$,
	which would make $\psi_j$ the trivial solution in contradiction to its definition.
	Therefore $M \neq 0$ and, after a suitable renormalization of $p_j$, we obtain $M=1$. Hence, \eqref{E:q-Linf-est} gives the desired bound for $p_j-1$.
	
	We conclude the proof by estimating $q' = p_j'$. By equation (\ref{eq:32}),
	\begin{align*}
		q'(x) &= \frac 1{2l} \int_0^x \left(-\frac{l-m}{e^{(l-m)a}-1}\,e^{(l-m)(x-t)}
		+ \frac{-l-m}{e^{(-l-m)a}-1}\,e^{(-l-m)(x-t)}\right) r(t)\,d t
		\\
		&\ + \frac 1{2l} \int_x^a \left(-\frac{-l+m}{e^{(-l+m)a}-1}\,e^{(-l+m)(t-x)}
		+ \frac{l+m}{e^{(l+m)a}-1}\,e^{(l+m)(t-x)}\right) r(t)\,d t
	\end{align*}
	and thus, for $l$ large enough,
	\begin{align*}
		&|q'(x)| \le \frac{\|r\|_\infty}{2l} \left(
		\left|\frac{l-m}{e^{(l-m)a}-1}\right|\,\frac{e^{\Real(l-m)x}-1}{\Real (l-m)}
		+ \left|\frac{-l-m}{e^{(-l-m)a}-1}\right|\,\frac{e^{\Real(-l-m)x}-1}{\Real (-l-m)}
		\right. \\
		&\quad\left.
		+ \left|\frac{-l+m}{e^{(-l+m)a}-1}\right|\,\frac{e^{\Real(-l+m)(a-x)}-1}{\Real(-l+m)}
		+ \left|\frac{l+m}{e^{(l+m)a}-1}\right|\,\frac{e^{\Real(l+m)(a-x)}-1}{\Real(l+m)} \right).
	\end{align*}
	Each of the four terms appearing in this sum has the form
	\begin{equation*}
		\left|\frac{\zeta}{e^\zeta-1}\right|\,\frac{e^{\lambda \alpha} - 1}\alpha
		= \frac{\int_0^\lambda e^{\alpha s}\,d s}{\left|\int_0^1 e^{\zeta s}\,d s\right|},
	\end{equation*}
	where $\zeta \in \{\pm(l-m)a, \pm(l+m)a\}$, $\alpha = \Real \zeta$ and either
	$\lambda = \frac x a \in [0, 1]$ or $\lambda = 1 - \frac x a \in [0, 1]$.
	Using the estimates (\ref{eq:34}) and (\ref{eq:35}), we find
	\begin{equation*}
		\left|\frac{\zeta}{e^\zeta-1}\right|\,\frac{e^{\lambda\alpha}-1}\alpha \le 2
	\end{equation*}
	for $l \ge l_0$.
	Therefore, bearing in mind the estimate (\ref{eq:39}) and the normalization \eqref{eq:26},
	\begin{align*}
		\|q'\|_\infty &\le \frac {\|r\|_\infty}{2l}\, 8
		\le \frac 4 l\,(\|V\|_\infty\,(1 + \|q\|_\infty) + |m_j^2 - l^2|\,1)
		\le \frac \c l
	\end{align*}
	for $l \ge l_0$.
\end{proof}

\section{The Resolvent Set}\label{S:resol-set}

To investigate the resolvent set, we need to study the unique solvability of \eqref{1} or, equivalently, \eqref{10tilde}. In the following, we work in the transformed variables $u_1,v,w$.
For ease of notation, in this section we denote the $x_1$-variable simply by $x$. The set $\Sc$ defined next will be shown to lie in the resolvent set. Recall that $k=k_2$ if $N=2$ and $k=(k_2,k_3)$ if $N=3$.

\begin{definition}\label{D:S}
	Let $N=2,3$. Define $\mathcal{S}$ to be the set of all  $\omega \in \C$ which satisfy i)-iv) below. Here we suppress the $\omega$-dependence of all occurring quantities and note that $\perm_\pm(0)$ denotes $\perm_\pm(x=0,\omega)$.
	\begin{enumerate}
		\item[i)] $\omega \in D(\perm)\setminus \Omega$.
		\item[ii)] For every $k\in\R^N$ there exist fundamental systems $\{v^{(1)}_+, v^{(2)}_+\}$ and $\{v^{(1)}_-, v^{(2)}_-\}$ of the homogeneous equation \eqref{3} on $[0,\infty)$ and on $(-\infty,  0]$, respectively, such that, for some (possibly $k$-dependent) $\nu_\pm > 0$ and $C_{v,1,\pm}, C'_{v,1,\pm}, C_{v,2,\pm},$ $ C'_{v,2,\pm}>0$,
		\bea
		\left. \begin{array}{ll} |v^{(1)}_\pm (x)| \le C_{v,1,\pm} e^{\nu_\pm x},& |v^{(1)'}_\pm (x)| \le C'_{v,1,\pm}  e^{\nu_\pm x}, \\
			|v^{(2)}_\pm (x)| \le C_{v,2,\pm} e^{-\nu_{\pm}x},& |v^{(2)'}_\pm (x)| \le C'_{v,2,\pm}  e^{-\nu_{\pm}x}, \end{array} \right\} ~{\rm on}~ \R_{\pm}
		\eea
		and moreover, $v^{(1)}_+ \notin L^2 (\R_+), v^{(2)}_- \notin L^2 (\R_-)$, and
		\beq\label{E:d}
		d(k):= - v^{(2)}_+ (0) v^{(1)'}_- (0) \frac{\omega^2\perm _- (0)}{|k|^2 - \omega^2\perm _- (0)} + v^{(1)}_- (0) v^{(2)'}_+ (0) \frac{\omega^2\perm _+ (0)}{|k|^2 - \omega^2\perm _+ (0)} \neq 0.
		\eeq
		\item[iii)] For every $k$ there exist fundamental systems $\{w^{(1)}_+, w^{(2)}_+\}$ and $\{w^{(1)}_-, w^{(2)}_-\}$ of the homogeneous equation \eqref{4} on $[0,\infty)$ and on $(-\infty,0]$, respectively, such that, for some (possibly $k$-dependent) constants $\widetilde{\nu}_{\pm} > 0$ and $C_{w,1,\pm}, C'_{w,1,\pm},$ $C_{w,2,\pm}, C'_{w,2,\pm}>0$,
		\bea
		\left.
		\begin{array}{ll} |w^{(1)}_\pm (x)| \le C_{w,1,\pm} e^{\widetilde{\nu}_{\pm}x}, & |w^{(1)'}_\pm (x)| \le C'_{w,1,\pm} e^{\widetilde{\nu}_{\pm}x} \\
			|w^{(2)}_\pm (x)| \le C_{w,2,\pm} e^{-\widetilde{\nu}_{\pm}x},& |w^{(2)'}_\pm (x)| \le C'_{w,2,\pm}  e^{-\widetilde{\nu}_{\pm}x} \end{array}
		\right\} ~{\rm on} ~ \R_{\pm}
		\eea
		and moreover, $w^{(1)}_+ \notin L^2 (\R_+), w^{(2)}_- \notin L^2 (\R_-)$ and
		\beq\label{E:dtil}
		\widetilde{d}(k) := - w^{(2)}_+ (0) w^{(1)'}_- (0) + w^{(1)}_- (0) w^{(2)'}_+ (0) \neq 0.
		\eeq
		\item[iv)] With $\tau_\pm$ and $\widetilde{\tau}_\pm$ denoting the Wronski determinants
		\beq\label{E:D_Dtil}
		\tau_\pm := \frac{\omega^2\perm _{\pm}}{|k|^2 - \omega^2\perm _{\pm}} (v^{(1)}_\pm v^{(2)'}_\pm - v^{(1)'}_\pm v^{(2)}_\pm), \quad \widetilde{\tau}_{\pm}:= w^{(1)}_\pm w^{(2)'}_\pm - w^{(1)'}_\pm w^{(2)}_\pm,
		\eeq
		which are constant with respect to $x$ (but usually depend on $k$), the following expressions are bounded by a constant independent of $k$:
		\bea
		\alpha_1^\pm:=& \frac{C_{v,2,\pm} C_{v,1,\pm}}{|\tau_\pm| (1 + |k|^2)\nu_\pm}  , \\
		\alpha_2^+:= &\frac{C_{v,2,+} }{|d| (1 + |k|^2) \sqrt{\nu_+}} \left\{ C_{v,1,-} (1+ \frac{1}{\sqrt{\nu_-}}) + \frac{1}{|\tau_+| (1+|k|^2) \sqrt{\nu_+}} (C_{v,1,-} C'_{v,1,+} + C'_{v,1,-} C_{v,1,+}) C_{v,2,+} \right\},\\
		\alpha_2^-:=&\frac{C_{v,1,-} }{|d| (1 + |k|^2) \sqrt{\nu_-}}  \left\{ C_{v,2,+} (1+\frac{1}{\sqrt{\nu_+}}) + \frac{1}{|\tau_-| (1+|k|^2) \sqrt{\nu_-}} (C_{v,2,-} C'_{v,2,+} + C'_{v,2,-} C_{v,2,+}) C_{v,1,-} \right\},  \\
		\alpha_3^+:=&\frac{1}{(1+|k|^3) \sqrt{\nu_+}}  \left[ \frac{C'_{v,2,+}}{|d|}  \left\{  C_{v,1,-}  (1+\frac{1}{\sqrt{\nu_-}})  + \frac{1}{ |\tau_+| (1+|k|^2)\sqrt{\nu_+}}  (C_{v,1,-} C'_{v,1,+} + C'_{v,1,-}         C_{v,1,+}) C_{v,2,+} \right\}\right. \\
		&\left. +  \frac{C'_{v,2,+} C_{v,1,+} + C'_{v,1,+} C_{v,2,+}}{|\tau_+| \sqrt{\nu_+}} \right] ,  \\
		\alpha_3^-:=&\frac{1}{(1+|k|^3) \sqrt{\nu_-}}  \left[ \frac{C'_{v,1,-}}{|d|}   \left\{  C_{v,2,+}  (1+\frac{1}{\sqrt{\nu_+}})  + \frac{1}{ |\tau_-| (1+|k|^2)\sqrt{\nu_-}}  (C'_{v,2,+} C_{v,2,-} + C'_{v,2,-}         C_{v,2,+}) C_{v,1,-} \right\} \right.\\
		&\left. +  \frac{C'_{v,2,-} C_{v,1,-} + C'_{v,1,-} C_{v,2,-}}{|\tau_-| \sqrt{\nu_-}} \right] , \\ 
		\alpha_4^\pm:=&\frac{C_{w,2,\pm} C_{w,1,\pm} }{|\widetilde{\tau}_{\pm}| \widetilde{\nu}_{\pm} } (1+|k|) ,   \\
		\alpha_5^+:=&\frac{C_{w,2,+} (1+|k|)}{|\widetilde{d}| \sqrt{\widetilde{\nu}_{+} }}   \left[   \frac{ C_{w,1,-}  C'_{w,1,+}  + C'_{w,1,-}   C_{w,1,+} }{|\widetilde{\tau}_+|                    \sqrt{\widetilde{\nu}_+}} C_{w,2,+} + \frac{C_{w,1,-}}{\sqrt{\widetilde{\nu}_-}}   \right] ,\\
		\alpha_5^-:=&\frac{C_{w,1,-} (1+|k|)}{|\widetilde{d}| \sqrt{\widetilde{\nu}_{-} }}   \left[   \frac{ C_{w,2,+}}{\sqrt{\widetilde{\nu}_+}}  + \frac{ C_{w,2,-} C'_{w,2,+} + C'_{w,2,-}  C_{w,2,+} }{|\widetilde{\tau}_-| \sqrt{\widetilde{\nu}_-}}  C_{w,1,-} \right] ,\\
		\alpha_6^+:=&\frac{C'_{w,2,+} }{|\widetilde{d}| \sqrt{\widetilde{\nu}_{+} }}   \left\{   \frac{ C_{w,1,-}  C'_{w,1,+} + C'_{w,1,-}   C_{w,1,+} }{|\widetilde{\tau}_+|                \sqrt{\widetilde{\nu}_+}}    C_{w,2,+} +  \frac{C_{w,1,-}}{\sqrt{\widetilde{\nu}_{-} }}   \right\} +  \frac{ C'_{w,2,+}  C_{w,1,+} + C'_{w,1,+}   C_{w,2,+} }{|\widetilde{\tau}_+| \widetilde{\nu}_+}  ,    \\
		\alpha_6^-:=&\frac{C'_{w,1,-}}{|\widetilde{d}| \sqrt{\widetilde{\nu}_{-} }}   \left\{   \frac{ C_{w,2,+}  }{\sqrt{\widetilde{\nu}_+}}  + \frac{C_{w,2,-} C'_{w,2,+}  + C'_{w,2,-} C_{w,2,+} }{|\widetilde{\tau}_-| \sqrt{\widetilde{\nu}_-}}    C_{w,1,-} \right\} +  \frac{ C'_{w,2,-}  C_{w,1,-} + C'_{w,1,-}   C_{w,2,-} }{|\widetilde{\tau}_-| \widetilde{\nu}_-}.
		\eea
	\end{enumerate}
\end{definition}
Note that in the special cases of homogeneous media, Sec. \ref{S:resolv-hom}, and periodic media, Sec \ref{S:resolv-per}, we can give a much simpler description of the set $\mathcal{S}$.

We next show that the set $\mathcal{S}$ lies in the resolvent set of $\mathcal{L}$.

\begin{thm} \label{T:resolv-set-gen} We have
	$\mathcal{S} \subset \rho (\mathcal{L})$.
\end{thm}

\bpf
Let $\omega \in \mathcal{S}$, i.e., i), ii), iii), iv) in Definition \ref{D:S} hold true. By Lemmas \ref{L:equiv-op} and \ref{lem:equiv}, we have to prove that the inhomogeneous system \eqref{3}, \eqref{4} has a unique solution $(v,w)$ such that $(u_1, v, w) \in \widetilde{\mathcal{D}}_\omega$ with $u_1$ defined by \eqref{5} and that $\|(u_1,v,w)\|_{L^2}\leq c \|r\|_{L^2}$ with $c$ independent of $r$. The case $k=0$ can be left out because only $L^2$-estimates with respect to $k$ are required and the set $\{(x_1,k) \text{ with } k=0\}$ has measure $0$ in $\R^N$. Hence, we assume $k\neq 0$.
Let $s$ denote the right-hand side of \eqref{3}. With $\{v^{(1)}_+, v^{(2)}_+\}$ and $\{v^{(1)}_-, v^{(2)}_-\}$ denoting fundamental systems for \eqref{3} given by ii), we first recall that the Wronski determinants $\tau_\pm$ defined in \eqref{E:D_Dtil} are constant  in $x$ and depend only on $k$.

By the variation of constants formula we get a particular solution $v^{{\rm part}}$ of \eqref{3} by
\be\label{14}
\begin{array}{l}v^{{\rm part}}_{+} (x) = - \frac{v^{(2)}_+ (x)}{\tau_+} \int\limits_0^x v^{(1)}_+ (t) s (t) dt - \frac{v^{(1)}_+ (x)}{\tau_+} \int\limits_x^{\infty} v^{(2)}_+ (t) s (t) dt \quad (x > 0), \\
	v^{{\rm part}}_{-} (x)=	- \frac{v^{(2)}_- (x)}{\tau_-} \int\limits_{-\infty}^x v^{(1)}_- (t) s (t) dt - \frac{v^{(1)}_- (x)}{\tau_-} \int\limits_x^0 v^{(2)}_- (t) s (t) dt \quad (x < 0). \end{array}  
\ee
By ii)
$$
\begin{array}{lcl}
	|v^{{\rm part}}_{+} (x)| &\le & \frac{C_{v,2,+} C_{v,1,+}}{|\tau_+|}  \left[ e^{-\nu_+ x}  \int\limits_0^x e^{\nu_+ t} |s (t)| dt + e^{\nu_+ x}  \int\limits_x^{\infty}  e^{-\nu_+ t} |s (t) | dt \right] \quad (x > 0), \\
	|v^{{\rm part}}_{-} (x)| &\le &	\frac{C_{v,2,-} C_{v,1,-}}{|\tau_-|}  \left[ e^{-\nu_- x}  \int\limits_{-\infty}^x e^{\nu_- t} |s (t)| dt + e^{\nu_- x}  \int\limits_x^0  e^{-\nu_- t} |s (t) | dt \right] \quad (x < 0),
\end{array}
$$
whence Lemma 4.7 of \cite{BDPW24} (and its natural extension to $\R_-$) gives
\be\label{15}
v^{{\rm part}}_{\pm} (\cdot,k) \in L^2 (\R_{\pm}), ~ \| v^{{\rm part}}_{\pm} (\cdot,k) \|_{L^2 (\R_{\pm})} \le \frac{2 C_{v,2,\pm} C_{v,1,\pm}}{|\tau_\pm| \nu_{\pm}}  \| s (\cdot,k)\|_{L^2 (\R_{\pm})}.
\ee

From now on $c>0$ denotes a generic $k$-independent constant.

By \eqref{3} we have
\be\label{16}
\| s (\cdot,k)\|_{L^2 (\R_{\pm})}  \le \frac{c}{1+|k|^2} \left\|  r (\cdot,k)  \right\|_{L^2 (\R_{\pm})^3} ~{\rm for~a.e.}~ k \in \R^{N-1}.
\ee
Hence, the boundedness of $\alpha_1^\pm$ in assumption iv) gives
\be\label{17}
v^{{\rm part}}_{\pm} \in L^2 (\R_{\pm}^N) , ~ \| v^{{\rm part}}_{\pm}\|_{L^2 (\R_{\pm}^N)} \le c \left\|  r \right\|_{L^2 (\R_{\pm}^N)^3} .
\ee
Since $v^{(1)}_+ \notin L^2 (\R_+), v^{(2)}_- \notin L^2 (\R_-)$ by ii), the general solution in $L^2 (\R_{\pm})$ of equation \eqref{3} is given by
\be\label{18}
\begin{array}{ll}  
	v_{+} (x) = A_+ v^{(2)}_+ (x) + v_+^{{\rm part}} (x) & (x >0),\\
	v_{-} (x)= A_- v^{(1)}_- (x) + v_-^{{\rm part}} (x) & (x <0), \end{array} 
\ee
with ($k$-dependent) constants $A_+, A_- \in \C$. The condition $\llbracket v \rrbracket = 0$ in $\widetilde{\mathcal{D}}_\omega$ requires $v_+ (0) = v_- (0)$, i.e.
\ben\label{19}
A_+ v^{(2)}_+ (0) - A_- v^{(1)}_- (0)& =& v^{{\rm part}}_- (0) - v^{{\rm part}}_+ (0) \nonumber\\
& = &\frac{v^{(1)}_+ (0)}{\tau_+}  \int\limits_0^{\infty} v^{(2)}_+ (t) s (t) dt - \frac{v^{(2)}_- (0)}{\tau_-}  \int\limits_{-\infty}^0 v^{(1)}_- (t) s (t) dt.
\een
Before analyzing this further, we first compute, using \eqref{18} and \eqref{14},
\be\label{20}
\begin{array}{l}  
	v'_{+} (x) = A_+ v^{(2)'}_+ (x) - \frac{v^{(2)'}_+ (x)}{\tau_+}  \int\limits_0^x v^{(1)}_+ (t) s (t) dt - \frac{v^{(1)'}_+ (x)}{\tau_+}  \int\limits_x^{\infty} v^{(2)}_+ (t) s (t) dt ~ (x>0), \\
	v'_{-} (x) = 	A_- v^{(1)'}_- (x) - \frac{v^{(2)'}_- (x)}{\tau_-}  \int\limits_{-\infty}^x v^{(1)}_- (t) s (t) dt - \frac{v^{(1)'}_- (x)}{\tau_-}  \int\limits_x^0 v^{(2)}_- (t) s (t) dt~ (x<0),
\end{array} 
\ee
which, by \eqref{5}, gives
\be\label{21}
\begin{array}{l}  
	u_{1,+} (x) = \frac{1}{|k|^2 - \omega^2\perm _+(x)}  \left\{ r_1 (x) - \ri |k|  \left[  A_+ v^{(2)'}_+ (x) - \frac{v^{(2)'}_+ (x)}{\tau_+}  \int\limits_0^x v^{(1)}_+ (t) s (t) dt
	- \frac{v^{(1)'}_+ (x)}{\tau_+}  \int\limits_x^{\infty} v^{(2)}_+ (t) s (t)  dt  \right] \right\} ~ (x>0),  \\
	u_{1,-} (x) = \frac{1}{|k|^2 - \omega^2\perm _-(x)}  \left\{ r_1 (x) - \ri |k|  \left[  A_- v^{(1)'}_- (x) - \frac{v^{(2)'}_- (x)}{\tau_-}  \int\limits_{-\infty}^x v^{(1)}_- (t) s (t) dt
	- \frac{v^{(1)'}_- (x)}{\tau_-}  \int\limits_x^0 v^{(2)}_- (t) s (t)  dt  \right] \right\} ~ (x<0).   \end{array} 
\ee
The condition $\llbracket\perm  u_1\rrbracket = 0$ in $\widetilde{\mathcal{D}}_\omega$ therefore requires
$$
\begin{aligned}
	&&\frac{\perm _+ (0)}{|k|^2 - \omega^2\perm _+ (0)} \left\{ r_1 (0) - \ri |k| \left[ A_+ v^{(2)'}_+ (0) - \frac{v^{(1)'}_+ (0)}{\tau_+}  \int\limits_0^{\infty} v^{(2)}_+ (t) s (t)  dt  \right] \right\} \\
	&&= \frac{\perm _- (0)}{|k|^2 - \omega^2\perm _- (0)} \left\{ r_1 (0) - \ri |k| \left[ A_- v^{(1)'}_- (0) - \frac{v^{(2)'}_- (0)}{\tau_-}   \int\limits_{-\infty}^0 v^{(1)}_- (t) s (t)  dt  \right] \right\},
\end{aligned}
$$
i.e.,
\beq\label{22}
\begin{aligned}
	&&\frac{\perm _+ (0)}{|k|^2 - \omega^2\perm _+ (0)}  v^{(2)'}_+ (0) A_+ - \frac{\perm _- (0)}{|k|^2 - \omega^2\perm _- (0)}   v^{(1)'}_- (0) A_-\ =\ \left[  \frac{\perm _+ (0)}{|k|^2 - \omega^2\perm _+ (0)} - \frac{\perm _- (0)}{|k|^2 - \omega^2\perm _- (0)} \right] \frac{r_1 (0)}{\ri |k|} \\
	&&+ \frac{\perm _+ (0)}{|k|^2 - \omega^2\perm _+ (0)} \frac{v^{(1)'}_+ (0)}{\tau_+}   \int\limits_0^{\infty}  v^{(2)}_+ (t) s (t)  dt   -  \frac{\perm _- (0)}{|k|^2 - \omega^2\perm _- (0)} \frac{v^{(2)'}_- (0)}{\tau_-}   \int\limits_{-\infty}^0  v^{(1)}_- (t) s (t)  dt  .
\end{aligned}
\eeq
Equations \eqref{19} and \eqref{22} together form a $2\times 2$-system for $(A_+, A_-)^\top$, the determinant of which is
\be\label{23}
- v^{(2)}_+ (0) v^{(1)'}_- (0) \frac{\perm _- (0)}{|k|^2 - \omega^2\perm _- (0)} +  v^{(1)}_- (0) v^{(2)'}_+ (0) \frac{\perm _+ (0)}{|k|^2 - \omega^2\perm _+ (0)} = \frac{d}{\omega^2}
\ee
with $d$ defined in \eqref{E:d}. The determinant is non-zero by assumption ii). Thus, $A_+$ and $A_-$ are uniquely given by
\ben\label{24}
A_+ &=&  - \frac{1}{d} \frac{\omega^2\perm _- (0)}{|k|^2 - \omega^2\perm _- (0)}  v^{(1)'}_- (0)  \left[ \frac{v^{(1)}_+ (0)}{\tau_+} \int\limits_0^{\infty} v^{(2)}_+ (t) s (t)  dt - \frac{v^{(2)}_-(0)}{\tau_-} \int\limits_{-\infty}^0 v^{(1)}_- (t) s (t)  dt \right]  \nonumber\\
&&+ \frac{1}{d}  v^{(1)}_- (0)  \Bigg\{  \left[ \frac{\omega^2\perm _+ (0)}{|k|^2 - \omega^2\perm _+ (0)} -  \frac{\omega^2\perm _- (0)}{|k|^2 - \omega^2\perm _- (0)}   \right] \frac{r_1 (0)}{\ri |k|}   \nonumber\\
&& + \frac{\omega^2\perm _+ (0)}{|k|^2 - \omega^2\perm _+ (0)}  \frac{v^{(1)'}_+ (0)}{\tau_+}  \int\limits_0^{\infty} v^{(2)}_+ (t) s (t)  dt -  \frac{\omega^2\perm _- (0)}{|k|^2 - \omega^2\perm _- (0)}   \frac{v^{(2)'}_- (0)}{\tau_-} \int\limits_{-\infty}^0 v^{(1)}_- (t) s (t)  dt \Bigg\} \nonumber\\
&=& \frac{1}{d}  \frac{|k| (\omega^2\perm _+ (0) - \omega^2\perm _- (0))}{(|k|^2 - \omega^2\perm _+ (0))(|k|^2 - \omega^2\perm _- (0))}  \cdot \frac{1}{\ri} v^{(1)}_- (0) r_1 (0)  \nonumber\\
&&+ \frac{1}{d\tau_+} \left[ \frac{\omega^2\perm _+ (0)}{|k|^2 - \omega^2\perm _+ (0)} v^{(1)}_- (0) v^{(1)'}_+(0) -  \frac{\omega^2\perm _- (0)}{|k|^2 - \omega^2\perm_-(0)} v^{(1)'}_- (0) v^{(1)}_+  (0)  \right] \int\limits_0^{\infty} v^{(2)}_+ (t) s (t)  dt \\
&&+ \frac{1}{d\tau_-} \underbrace{ \frac{\omega^2\perm _- (0)}{|k|^2 - \omega^2\perm _- (0)}  \left[ v^{(1)'}_- (0) v^{(2)}_- (0) - v^{(1)}_- (0) v^{(2)'}_- (0) \right] }_{= - \tau_-}   \int\limits_{-\infty}^0 v^{(1)}_- (t) s (t)  dt ,\nonumber
\een
\ben\label{25}
A_- &=&  - \frac{1}{d} \frac{\omega^2\perm _+ (0)}{|k|^2 - \omega^2\perm _+ (0)}  v^{(2)'}_+ (0)  \left[ \frac{v^{(1)}_+ (0)}{\tau_+} \int\limits_0^{\infty} v^{(2)}_+ (t) s (t)  dt - \frac{v^{(2)}_-(0)}{\tau_-} \int\limits_{-\infty}^0 v^{(1)}_- (t) s (t)  dt \right]  \nonumber\\
&&+ \frac{1}{d}  v^{(2)}_+ (0)  \Bigg\{  \left[ \frac{\omega^2\perm _+ (0)}{|k|^2 - \omega^2\perm _+ (0)} -  \frac{\omega^2\perm _- (0)}{|k|^2 - \omega^2\perm _- (0)}   \right] \frac{r_1 (0)}{\ri |k|}  \nonumber \\
&& + \frac{\omega^2\perm _+ (0)}{|k|^2 - \omega^2\perm _+ (0)}  \frac{v^{(1)'}_+ (0)}{\tau_+}  \int\limits_0^{\infty} v^{(2)}_+ (t) s (t)  dt -  \frac{\omega^2\perm _- (0)}{|k|^2 - \omega^2\perm _- (0)}   \frac{v^{(2)'}_- (0)}{\tau_-} \int\limits_{-\infty}^0 v^{(1)}_- (t) s (t)  dt \Bigg\} \nonumber\\
&=& \frac{1}{d}  \frac{|k| (\omega^2\perm _+ (0) - \omega^2\perm _- (0))}{(|k|^2 - \omega^2\perm _+ (0))(|k|^2 - \omega^2\perm _- (0))}  \cdot \frac{1}{\ri} v^{(2)}_+ (0) r_1 (0)  \nonumber\\
&&+ \frac{1}{d\tau_+} \underbrace{ \frac{\omega^2\perm _+ (0)}{|k|^2 - \omega^2\perm _+ (0)} \left[ - v^{(2)'}_+ (0) v^{(1)}_+(0) + v^{(2)}_+ (0) v^{(1)'}_+ (0) \right]}_{= - \tau_+}   \int\limits_0^{\infty} v^{(2)}_+ (t) s (t)  dt \\
&&+ \frac{1}{d\tau_-}   \left[ \frac{\omega^2\perm _+ (0)}{|k|^2 - \omega^2\perm _+ (0)}  v^{(2)'}_+ (0) v^{(2)}_- (0) - \frac{\omega^2\perm _- (0) }{|k|^2 - \omega^2\perm _- (0)}  v^{(2)}_+ (0) v^{(2)'}_- (0)  \right]   \int\limits_{-\infty}^0 v^{(1)}_- (t) s (t)  dt.\nonumber
\een
Next we bound $A_+, A_-$ with respect to their $k$-dependency. First, using the embedding $H^1 (\R) \hookrightarrow C_b (\R)$ (with embedding constant $C_0$) and that $r_1'=-\ri k_2r_2$ for $N=2$ and $r'_1= - \ri (k_2r_2 + k_3r_3)$ for $N=3$, we find
\ben\label{26}
|r_1 (0,k)| &\le& C_0 \| r_1(\cdot,k) \|_{H^1(\R)}  = C_0 \sqrt{\| r_1(\cdot,k) \|^2_{L^2(\R)} + \| r'_1(\cdot,k) \|^2_{L^2(\R)} }\nonumber\\
&\le& C_0 \sqrt{\| r_1(\cdot,k) \|^2_{L^2(\R)} + |k|^2 ( \| r_2(\cdot,k) \|^2_{L^2(\R)} + \| r_3(\cdot,k) \|^2_{L^2(\R)} ) }  \nonumber\\
&\le& C_0 \sqrt{1+|k|^2} \|r(\cdot,k)\|_{L^2(\R)^3}.
\een

Furthermore, using the assumptions ii), and \eqref{16}, we obtain
\ben\label{27}
\left| \int\limits_0^{\infty}  v^{(2)}_+ (t) s (t) dt \right| &\le& C_{v,2,+} \int\limits_0^{\infty} e^{- \nu_+t} | s (t)| dt \le C_{v,2,+} \frac{1}{\sqrt{2 \nu_+}} \| s \|_{L^2 (\R)}  \nonumber\\
&\le& C_{v,2,+} c  \frac{1}{\sqrt{2 \nu_+}} \frac{1}{1+|k|^2} \|r(\cdot,k)\|_{L^2(\R)^3}   ,
\een
and analogously,
\be\label{28}
\left| \int\limits_{-\infty}^0  v^{(1)}_- (t) s (t) dt \right| \le C_{v,1,-}c \frac{1}{\sqrt{2 \nu_-}}  \frac{1}{1+|k|^2} \|r(\cdot,k)\|_{L^2(\R)^3}  .
\ee
Using assumptions i), ii), and \eqref{24} - \eqref{28}, we find
\ben\label{29}
|A_+ (k) | &\leq & \left\{ \frac{1}{|d|}  \frac{c}{1+|k|^3} C_{v,1,-} C_0 \sqrt{1+|k|^2}  \right. \nonumber\\
&&+ \frac{1}{|d| |\tau_+|}   \left[\frac{c}{1+|k|^2}  \left( C_{v,1,-} C'_{v,1,+}  +  C'_{v,1,-}  C_{v,1,+}  \right) \right] \frac{C_{v,2,+}c}{\sqrt{2 \nu_+ } (1+|k|^2)} \nonumber\\
&&\left. + \frac{1}{d}  \frac{C_{v,1,-} c}{\sqrt{2 \nu_- } (1+|k|^2)} \right\}    \|r\|_{L^2(\R)^3} \nonumber \\
&\le& \frac{c}{|d| (1+|k|^2)}  \left\{ C_{v,1,-} +  \frac{C_{v,1,-} C'_{v,1,+}  +  C'_{v,1,-}  C_{v,1,+}}{|\tau_+| (1+|k|^2)}  \frac{ C_{v,2,+} }{\sqrt{\nu_+} }    +   \frac{C_{v,1,-}}{\sqrt{\nu_- }}   \right\}    \|r(\cdot,k)\|_{L^2(\R)^3}.
\een
Analogously
\ben\label{30}
|A_- (k) | &\leq & \frac{c}{|d| (1+|k |^2)} \left\{ C_{v,2,+} +   \frac{C_{v,2,+}}{\sqrt{\nu_+}} + \frac{C'_{v,2,+}  C_{v,2,-}  +  C_{v,2,+}  C'_{v,2,-}}{|\tau_-| (1+|k|^2)} \frac{C_{v,1,-}}{\sqrt{\nu_-}} \right\} \|r(\cdot,k)\|_{L^2(\R)^3}.
\een
Consequently,
\ben\label{31}
\| A_+(k) v^{(2)}_+\|_{L^2 (\R_+)} &\le&   | A_+(k)|C_{v,2,+}  \| e^{-\nu_+\cdot} \|_{L^2 (\R_+)} = |A_+(k)|    \frac{C_{v,2,+}}{ \sqrt{2\nu_+}} \nonumber\\
&\hspace{-155pt}\le& \hspace{-85pt}\frac{c C_{v,2,+}}{|d| (1+|k |^2) \sqrt{\nu_+}} \left\{ C_{v,1,-} \left(1+\frac{1}{\sqrt{\nu_-}} \right) + \frac{\left( C_{v,1,-}  C'_{v,1,+}  +  C'_{v,1,-}  C_{v,1,+}  \right) C_{v,2,+}}{|\tau_+| (1+|k|^2)  \sqrt{\nu_+}}   \right\}    \|r(\cdot,k)\|_{L^2(\R)^3}.
\een

Since, by assumption iv), the factor $c \alpha_2^+$ in \eqref{31} is a bounded function of $k$, we obtain some global constant $c$ such that
\be\label{32}
A_+ v^{(2)}_+ \in L^2 (\R_+^N), \| A_+ v^{(2)}_+ \|_{L^2 (\R_+^N)} \le  c \|r\|_{L^2(\R^N)^3}.
\ee
Analogously,
\be\label{33}
A_- v^{(1)}_- \in L^2 (\R_-^N), \| A_- v^{(1)}_- \|_{L^2 (\R_-^N)} \le  c \|r\|_{L^2(\R^N)^3}.
\ee
Together with \eqref{17}, \eqref{18}, we find
\be\label{34}
v \in L^2 (\R^N), \| v\|_{L^2 (\R^N)} \le  c \|r\|_{L^2(\R^N)^3}.
\ee

Next we show in an analogous way that also
\be\label{35}
\frac{v'}{1+|k|}  \in  L^2 (\R^N), \left\| \frac{v'}{1+|k|}   \right\|_{L^2 (\R^N)} \le  c \|r\|_{L^2(\R^N)^3}.
\ee
Indeed, using \eqref{20} and assumptions ii) we obtain
\bea
|v_+' (x)| &\le& | A_+ v^{(2)'}_+ (x) |+ \frac{C'_{v,2,+} C_{v,1,+}}{|\tau_+|} e^{-\nu_+ x} \int\limits_0^x   e^{\nu_+ t}  | s (t)| dt +             \frac{C'_{v,1,+} C_{v,2,+}}{|\tau_+|} e^{\nu_+ x} \int\limits_x^{\infty}   e^{-\nu_+ t}  | s (t)| dt \quad (x>0),
\eea
and hence, by Lemma 4.7 in \cite{BDPW24},
\bea
\|v'_+ \|_{L^2 (\R_+)}  \le | A_+ | C'_{v,2,+} \| e^{-\nu_+  \cdot} \|_{L^2 (\R_+)} +  \frac{C'_{v,2,+} C_{v,1,+} + C'_{v,1,+} C_{v,2,+} }{|\tau_+| \nu_+ } \| s \|_{L^2 (\R_+)}.
\eea
Using \eqref{29} and \eqref{16}, we conclude
\bea
\|v'_+ \|_{L^2 (\R_+)}  &\le& \left[ \frac{c~C'_{v,2,+}}{|d| (1+|k|^2) \sqrt{\nu_+ }} \left\{ C_{v,1,-}  \left(1+ \frac{1}{\sqrt{\nu_-}}\right)  +  \frac{( C_{v,1,-} C'_{v,1,+}  + C'_{v,1,-} C_{v,1,+} )    C_{v,2,+}} {|\tau_+| (1+|k|^2) \sqrt{\nu_+ }}\right\} \right.\\
&& \left.  + \frac{C'_{v,2,+} C_{v,1,+}  + C'_{v,1,+} C_{v,2,+}}{|\tau_+| \nu_+ }  \frac{c}{ 1+|k|^2} \right]   \|r(\cdot,k)\|_{L^2(\R_+)^3}  ,
\eea
which gives
\bea
\frac{1}{ 1+|k|}  \|v'_+ \|_{L^2 (\R_+)}  &\le&  \frac{c}{ (1+|k|^3) \sqrt{\nu_+ }} \left[ \frac{C'_{v,2,+}}{|d|}  \bigg\{  C_{v,1,-} \left(1+  \frac{1}{ \sqrt{\nu_-}}\right) +  \right. \ \\
&& \left. + \frac{ (C_{v,1,-} C'_{v,1,+}  + C'_{v,1,-} C_{v,1,+} ) C_{v,2,+}}{|\tau_+| (1+|k|^2) \sqrt{\nu_+ }}  \bigg\}  + \frac{C'_{v,2,+} C_{v,1,+}  + C'_{v,1,+} C_{v,2,+}}{|\tau_+| \sqrt{\nu_+ }} \right]      \|r(\cdot,k)\|_{L^2(\R_+)^3}.
\eea
Since the factor $c \alpha_3^+ $ on the right-hand side is bounded independently of $k$ by assumption iv), we conclude
\be\label{36}
\frac{1}{ 1+|k|}  v'_+  \in L^2 (\R_+^N), \qquad   \left\| \frac{1}{1+|k|}  v'_+  \right\|_{L^2(\R^N_+)}   \le  c   \|r\|_{L^2(\R_+^N)^3} .
\ee
Analogously, \eqref{20}, Lemma 4.7 in \cite{BDPW24}, \eqref{30} and \eqref{16} imply
\bea
|v'_- (x)| &\le& | A_- v^{(1)'}_- (x) |  + \frac{C'_{v,2,-} C_{v,1,-}}{|\tau_-|}   e^{-\nu_- x} \int\limits_{-\infty}^x   e^{\nu_- t}  | s (t)| dt \\ 
&& +   \frac{C'_{v,1,-} C_{v,2,-}}{|\tau_-|} e^{\nu_- x} \int\limits_x^0   e^{-\nu_- t}  | s (t)| dt \quad (x<0), \\
\|v'_- \|_{L^2 (\R_-)} &\le& | A_- | C'_{v,1,-} \| e^{\nu_- \cdot} \|_{L^2 (\R_-)}  + \frac{C'_{v,2,-} C_{v,1,-} + C'_{v,1,-} C_{v,2,-} }{|\tau_-| \nu_-}  \| s \|_{L^2 (\R_-)} \\
&\le&  \left[ \frac{c~C'_{v,1,-}}{|d| (1+|k|^2) \sqrt{ \nu_-}}   \left\{  C_{v,2,+}  \left( 1+\frac{1}{\sqrt{ \nu_+}} \right)  + \frac{( C'_{v,2,+} C_{v,2,-} + C'_{v,2,-} C_{v,2,+}) C_{v,1,-} }{|\tau_-| (1+|k|^2) \sqrt{ \nu_-}} \right\} \right. \\
&&+  \left. \frac{C'_{v,2,-} C_{v,1,-}  + C'_{v,1,-} C_{v,2,-} }{|\tau_-| \nu_-}  \frac{c}{1+|k|^2}  \right]  \|r(\cdot,k)\|_{L^2(\R_-)^3} , \\
\frac{1}{1+|k|}  \|v'_- \|_{L^2 (\R_-)}   &\le& \frac{c}{(1+|k|^3) \sqrt{\nu_-}} \left[  \frac{C'_{v,1,-}}{|d|}  \left\{ C_{v,2,+}  \left( 1 + \frac{1}{\sqrt{\nu_+}} \right) +  \frac{( C'_{v,2,+} C_{v,2,-} + C'_{v,2,-} C_{v,2,+}) C_{v,1,-} }{|\tau_-| (1+|k|^2) \sqrt{ \nu_-}}  \right\}  \right. \\
&&+ \left. \frac{C'_{v,2,-} C_{v,1,-} + C'_{v,1,-} C_{v^{(2)}_- } }{|\tau_-| \sqrt{\nu_-}}  \right]  \|r(\cdot,k)\|_{L^2(\R_-)^3} .
\eea

The factor $c \alpha_3^-$  on the right-hand side is bounded independently of $k$ by assumption iv), and hence
\bea
\frac{1}{ 1+|k|}  v'_-  \in L^2 (\R_-^N),    \left\| \frac{1}{1+|k|}  v'_-  \right\|_{L^2(\R^N_-)}   \le c    \|r\|_{L^2(\R_-^N)^3} ,
\eea
which together with \eqref{36} gives \eqref{35}.\\

Now we show the remaining properties of $u_1$ and $v$ required in the definition of $\mathcal{\widetilde{D}}_\omega$.\\

By \eqref{5} and \eqref{35} (and assumption i)), we obtain
\be\label{37}
u_1  \in L^2 (\R^N), ~ \| u_1\|_{L^2 (\R^N)} \le c  \|r\|_{L^2(\R^N)^3} .
\ee
Rearranging \eqref{5} gives
\be\label{38}
v' - \ri |k| u_1  = - \frac{\omega^2\perm }{|k|^2 - \omega^2\perm } v' - \frac{\ri |k|}{|k|^2 - \omega^2\perm } r_1 ,
\ee
and
\bea
\ri |k| v' + |k|^2 u_1  = - \frac{\ri |k| \omega^2\perm }{|k|^2 - \omega^2\perm } v' + \frac{|k|^2}{|k|^2 - \omega^2\perm } r_1
\eea
which are both in $L^2 (\R^N)$ by \eqref{35}. Furthermore, \eqref{38}, \eqref{3}, and $r'_1 + \ri k_2 r_2 + \ri k_3 r_3 = 0$ imply for $N=3$
\bea
v'' - \ri |k| u'_1 &=& - \left(  \frac{\omega^2\perm }{|k|^2 - \omega^2\perm } v'  \right)' - \frac{\ri |k|}{|k|^2 - \omega^2\perm } r'_1 -  \frac{\ri |k|\omega^2\perm '}{(|k|^2 - \omega^2\perm )^2} r_1 \\
&=& - \omega^2\perm v +   \frac{\omega^2\perm }{|k| (|k|^2 - \omega^2\perm )} (k_2r_2 + k_3 r_3) -  \frac{|k|}{|k|^2 - \omega^2\perm }  (k_2r_2 + k_3 r_3) \\
&=&   - \omega^2\perm v -   \frac{1}{|k|} (k_2r_2 + k_3 r_3)
\eea
which is in $L^2 (\R^3)$ by \eqref{34}. For $N=2$ the same calculation but with $k_3=0$ produces 
\bea
v'' - \ri |k| u'_1 &=& -\omega^2\perm v -   \sign(k_2)r_2  \in  L^2 (\R^2).
\eea
Moreover, by \eqref{5}, \eqref{3} and using $r_1'= - \ri (k_2 r_2 + k_3 r_3)$ for $N=3$,
\ben\label{39}
(\perm  u_1)' + \ri |k| \perm v &=& \left[ \frac{\perm }{|k|^2 - \omega^2\perm } (-\ri  |k| v' + r_1 ) \right]' + \ri |k| \perm v \nonumber\\
&=& \ri |k| \left[ - \left( \frac{\perm }{|k|^2 - \omega^2\perm } v' \right)' + \perm v \right] +  \frac{|k|^2 \perm '}{(|k|^2 - \omega^2\perm )^2} r_1 + \frac{\perm }{|k|^2-\omega^2\perm } r_1'  \nonumber\\
&=& \ri |k|   \left[ \frac{\ri |k| \perm '}{(|k|^2 - \omega^2\perm )^2} r_1 + \frac{\perm }{|k| (|k|^2 - \omega^2\perm )} (k_2 r_2 + k_3 r_3) \right]    \nonumber\\
&& +  \frac{|k|^2 \perm '}{(|k|^2 - \omega^2\perm )^2} r_1 -  \frac{\ri\perm }{|k|^2 - \omega^2\perm } (k_2 r_2 + k_3 r_3) \nonumber\\
&=& 0 \quad\rm{ ~ on  ~} \R_{\pm} .
\een
For $N=2$ the same calculation applies if we set $k_3=0$.

Moreover, $u_1 \in L^2 (\R^N)$ by \eqref{37} and thus $u_1(\cdot,k) \in L^2 (\R_{\pm})$ and $\perm  u_1(\cdot,k) \in L^2 (\R_{\pm})$ for almost all $k$. Therefore, since \eqref{39} gives $(\perm  u_1(\cdot,k))' \in L^2 (\R_{\pm})$, and $\llbracket\perm  u_1 \rrbracket = 0$ by \eqref{22}, we obtain $\perm  u_1(\cdot,k) \in H^1 (\R)$.\\

Furthermore, $(|k|^2 - \omega^2\perm ) u_1 = - \ri |k| v' + r_1$ by \eqref{5}, implying
\bea
v' - \ri  |k| u_1 = \frac{1}{\ri |k|} (\omega^2\perm  u_1 + r_1) \in H^1 (\R),
\eea
and, since $\llbracket\perm  u_1 \rrbracket = \llbracket r_1 \rrbracket = 0$, we have
$
\llbracket v' - \ri  |k| u_1 \rrbracket = 0.
$
Finally, $\llbracket v\rrbracket = 0$ by \eqref{19}, which completes the proof of all properties concerning $u_1$ and $v$ in $\mathcal{\widetilde{D}}_\omega$.\\

To prove also the properties concerning $w$ in $\mathcal{\widetilde{D}}_\omega$, let $\widetilde{s}$ denote the right-hand side of \eqref{4}. With $\{w^{(1)}_+, w^{(2)}_+\}$ and $\{w^{(1)}_-, w^{(2)}_-\}$ denoting the fundamental systems of problem \eqref{4} given by assumption iii), we proceed as before for $v$, and obtain \eqref{14} and \eqref{15}, now with $w_{\pm}^{\rm{(part)}}$ instead of $v_{\pm}^{\rm{(part)}}$ and with all expressions on the right-hand side correspondingly replaced, i.e., $v_\pm^{(1,2)}, s, \tau_\pm$ replaced by $w_\pm^{(1,2)}, \widetilde{s}, \widetilde{\tau}_\pm$ respectively. We also obtain \eqref{16} for $\widetilde{s}$, but without the denominator $1 +|k|^2$.\\

Hence, the boundedness of $\alpha_4^\pm$ in assumption iv) implies \eqref{17} for $(1+|k|) w_{\pm}^{\rm{(part)}}$, and we obtain the general solution in $L^2(\R_{\pm})$ of problem \eqref{4} by
\be\label{40}
\begin{array}{ll} 
	w_{+} (x) = \widetilde{A}_+ w^{(2)}_+ (x) + w_+^{\rm{(part)}} (x) &  (x \in \R_+), \\
	w_{-} (x) = \widetilde{A}_- w^{(1)}_- (x) + w_-^{\rm{(part)}} (x) &  (x \in \R_-), \end{array}
\ee
with ($k$-dependent) constants $\widetilde{A}_+,  \widetilde{A}_- \in \C$. The condition $\llbracket w\rrbracket = 0$ in $\widetilde{\mathcal{D}}_\omega$ gives  equation \eqref{19}, and $w'_{\pm}$ is given by \eqref{20} (both with the corresponding replacements).\\

The condition $\llbracket w'\rrbracket = 0$ in $\widetilde{\mathcal{D}}_\omega$ therefore requires
\bea
\widetilde{A}_+ w^{(2)'}_+ (0) - \widetilde{A}_- w^{(1)'}_- (0)  = \frac{w_+^{(1)'} (0)}{\widetilde{\tau}_+}  \int\limits_0^{\infty}   w^{(2)}_+ (t) \widetilde{s} (t) dt -   \frac{w^{(2)'}_- (0)}{\widetilde{\tau}_-}  \int\limits_{-\infty}^0   w^{(1)}_- (t) \widetilde{s} (t) dt.
\eea
Together with \eqref{19} (with the corresponding replacements), we obtain a $2\times 2$ system for $(\widetilde{A}_+,  \widetilde{A}_-)^\top$ with determinant $\widetilde{d} = - w^{(2)}_+ (0)  w^{(1)'}_- (0) +  w^{(1)}_- (0)  w^{(2)'}_+ (0) $ which is non-zero by assumption iii). Thus, $\widetilde{A}_+$ and $\widetilde{A}_-$ are uniquely calculated as
$$
\begin{aligned}
	\widetilde{A}_+  &= \frac{1}{\widetilde{d}}  \left[ \frac{w^{(1)}_- (0)  w_+^{(1)'} (0)  -  w^{(1)'}_- (0) w^{(1)}_+ (0)}{\widetilde{\tau}_+} \int\limits_0^{\infty}   w^{(2)}_+ (t) \widetilde{s} (t) dt -   \int\limits_{-\infty}^0   w^{(1)}_- (t) \widetilde{s} (t) dt \right] ,\\
	\widetilde{A}_-  &= \frac{1}{\widetilde{d}}  \left[ - \int\limits_0^{\infty}   w^{(2)}_+ (t) \widetilde{s} (t) dt + \frac{w^{(2)}_- (0)  w^{(2)'}_+  (0)  -  w^{(2)'}_- (0) w^{(2)}_+  (0)}{\widetilde{\tau}_-}   \int\limits_{-\infty}^0   w^{(1)}_- (t) \widetilde{s} (t) dt \right] .
\end{aligned}
$$
Using analogues of \eqref{27} and \eqref{28} (again, the denominator $1+|k|^2$ is not present now), we estimate
\be\label{41}
|\widetilde{A}_+| \le \frac{c}{|\widetilde{d}|}   \left[  \frac{C_{w,1,-}  C'_{w,1,+} + C'_{w,1,-}  C_{w,1,+}}{|\widetilde{\tau}_+|}  \cdot \frac{C_{w,2,+}}{\sqrt{\widetilde{ \nu}_+}} +  \frac{C_{w,1,-}}{\sqrt{\widetilde{ \nu}_-}}  \right]  \|r(\cdot,k)\|_{L^2(\R)^3} ,
\ee
\be\label{42}
|\widetilde{A}_-| \le \frac{c}{|\widetilde{d}|}   \left[  \frac{C_{w,2,+}}{\sqrt{\widetilde{ \nu}_+}}  + \frac{  C_{w,2,-} C'_{w,2,+} + C'_{w,2,-}   C_{w,2,+}}{|\widetilde{\tau}_-|}   \frac{C_{w,1,-}}{\sqrt{\widetilde{ \nu}_-}}  \right]  \|r(\cdot,k)\|_{L^2(\R)^3}
\ee
with some $k$-independent constant $c$. Consequently, as $\| w^{(2)}_+ \|_{L^2 (\R_+)} \le \frac{C_{w,2,+}}{\sqrt{2 \widetilde{ \nu}_+}}$ and $\| w^{(1)}_- \|_{L^2 (\R_-)} \le \frac{C_{w,1,-}}{\sqrt{2 \widetilde{ \nu}_-}}$, we get
\bea
(1+|k|) \| \widetilde{A}_+(k) w^{(2)}_+ \|_{L^2 (\R_+)}  \le \frac{c~C_{w,2,+} (1+|k|)}{|\widetilde{d}| \sqrt{\widetilde{ \nu}_+} }   \left[  \frac{C_{w,1,-}  C'_{w,1,+} + C'_{w,1,-}  C_{w,1,+}}{|\widetilde{\tau}_+| \sqrt{\widetilde{ \nu}_+} }   C_{w,2,+} +  \frac{C_{w,1,-}}{\sqrt{\widetilde{ \nu}_-}}  \right]  \|r(\cdot,k)\|_{L^2(\R)^3}
\eea
and
\bea
(1+|k|) \| \widetilde{A}_-(k) w^{(1)}_- \|_{L^2 (\R_-)}  \le \frac{c~C_{w,1,-} (1+|k|)}{|\widetilde{d}| \sqrt{\widetilde{ \nu}_-} }   \left[  \frac{C_{w,2,+}}{\sqrt{\widetilde{ \nu}_+}} +    \frac{C_{w,2,-}   C'_{w,2,+} + C'_{w,2,-}  C_{w,2,+}}{|\widetilde{\tau}_-| \sqrt{\widetilde{ \nu}_-}}  C_{w,1,-}    \right]  \|r(\cdot,k)\|_{L^2(\R)^3} .
\eea
Since, by assumption iv), the two factors $c \alpha_5^\pm$ on the right-hand sides are bounded independently of $k$, and since \eqref{17} holds for $(1+|k|) w_{\pm}^{\rm{(part)}}$, we obtain from \eqref{40} that
\be\label{43}
( 1+|k|) w  \in L^2 (\R^N), \| (1+|k|) w \|_{L^2(\R^N)}    \le  c   \|r\|_{L^2(\R^N)^3} ,
\ee
which shows that both $w$ and $|k| w$ are in $L^2 (\R^N)$. The next task is to show that $w'  \in L^2 (\R^N)$. Indeed, using \eqref{20} (with the corresponding replacements, see above) and assumption ii) we find
\bea
|w'_+ (x)| \le | \widetilde{A}_+ | C'_{w,2,+} e^{-\widetilde{\nu}_+ x}  + \frac{C'_{w,2,+} C_{w,1,+}}{|\widetilde{\tau}_+| }  e^{-\widetilde{\nu}_+ x}  \int\limits_0^x    e^{\widetilde{\nu}_+ t}  | \widetilde{s} (t)| dt  
+    \frac{C'_{w,1,+} C_{w,2,+}}{|\widetilde{\tau}_+| }  e^{\widetilde{\nu}_+ x}  \int\limits_x^{\infty}    e^{-\widetilde{\nu}_+ t}  | \widetilde{s} (t)| dt \ (x>0),
\eea
whence Lemma 4.7 in \cite{BDPW24} and \eqref{41} give
\bea
\|w'_+ \|_{L^2 (\R_+)} &\le& | \widetilde{A}_+ |   \frac{C'_{w,2,+}}{\sqrt{2 \widetilde{\nu}_+}}  + \frac{C'_{w,2,+} C_{w,1,+}  + C'_{w,1,+} C_{w,2,+}}{|\widetilde{\tau}_+| \widetilde{\nu}_+ } \| \widetilde{s} \|_{L^2 (\R_+)}   \\
&\le& \left[ \frac{c~C'_{w,2,+}}{|\widetilde{d}| \sqrt{\widetilde{\nu}_+}}   \left\{ \frac{C_{w,1,-} C'_{w,1,+} +  C'_{w,1,-}  C_{w,1,+}}{|\widetilde{\tau}_+| \sqrt{\widetilde{\nu}_+}} C_{w,2,+} +   \frac{C_{w,1,-}}{\sqrt{\widetilde{\nu}_-}} \right\}    \right.\\
&\quad & \left.+ c \frac{C'_{w,2,+}   C_{w,1,+} + C'_{w,1,+}  C_{w,2,+}}{|\widetilde{\tau}_+| \widetilde{ \nu}_+} \right]  \|r(\cdot,k)\|_{L^2(\R_+)^3} .
\eea
Since the factor $c\alpha_6^+$ on the right-hand side is bounded independently of $k$ by assumption iv), we conclude
\be\label{44}
w'_+  \in L^2 (\R^N_+), ~ \|  w'_+ \|_{L^2(\R_+^N)}    \le  c   \|r\|_{L^2(\R_+^N)^3} .
\ee
Analogously,
\bea
|w'_- (x)| \le | \widetilde{A}_- |C'_{w,1,-} e^{\widetilde{\nu}_- x}  + \frac{C'_{w,2,-} C_{w,1,-}}{|\widetilde{\tau}_-| }  e^{-\widetilde{\nu}_- x}  \int\limits_{-\infty}^x    e^{\widetilde{\nu}_- t}  | \widetilde{s} (t)| dt
+    \frac{C'_{w,1,-} C_{w,2,-}}{|\widetilde{\tau}_-| }  e^{\widetilde{\nu}_- x}  \int\limits_x^0    e^{-\widetilde{\nu}_- t}  | \widetilde{s} (t)| dt  ,
\eea
\bea
\|w'_- \|_{L^2 (\R_-)} &\le& | \widetilde{A}_- |  \frac{C'_{w,1,-}}{\sqrt{2 \widetilde{\nu}_-}}  + \frac{C'_{w,2,-} C_{w,1,-}  + C'_{w,1,-} C_{w,2,-}}{|\widetilde{\tau}_-| \widetilde{\nu}_- } \| \widetilde{s} \|_{L^2 (\R_-)}   \\
&\le& \left[ \frac{c~C'_{w,1,-}}{|\widetilde{d}| \sqrt{\widetilde{\nu}_-}}   \left\{ \frac{C_{w,2,+}}{ \sqrt{\widetilde{\nu}_+}} +  \frac{C_{w,2,-}   C'_{w,2,+} + C'_{w,2,-}  C_{w,2,+}}{|\widetilde{\tau}_-| \sqrt{\widetilde{ \nu}_-}} C_{w,1,-}       \right\}    \right.\\
&\quad & \left.+ c \frac{C'_{w,2,-}   C_{w,1,-} + C'_{w,1,-}  C_{w,2,-}}{|\widetilde{\tau}_-| \widetilde{ \nu}_-} \right]  \|r(\cdot,k)\|_{L^2(\R_-)^3} .
\eea
Again, the factor $c\alpha_6^-$ on the right-hand side is bounded independently of $k$ by iv), and hence
\bea
w'_-  \in L^2 (\R^N_-), ~  \|  w'_- \|_{L^2(\R_-^N)}    \le  c    \|r\|_{L^2(\R_-^N)^3} ,
\eea
which together with \eqref{44} implies
\be\label{45}
w'  \in L^2 (\R^N), ~  \|  w' \|_{L^2(\R^N)}    \le  c  \|r\|_{L^2(\R^N)^3} .
\ee
Finally, by \eqref{4},
\bea
w'' - |k|^2 w = \begin{cases}
- \omega^2\perm w - \sign(k_2)r_2 \in L^2 (\R^2) ~ &{\rm if } ~ N=2,\\
	- \omega^2\perm w - \frac{1}{|k|} (k_2r_2 + k_3 r_3) \in L^2 (\R^3) ~ &{\rm if } ~ N=3,
\end{cases}
\eea
which completes the proof of $(u_1, v, w)  \in \widetilde{\mathcal{D}}_\omega$, and \eqref{34}, \eqref{37}, \eqref{43} imply
\bea
\left\| (u_1,v,w) \right\|_{L^2 (\R^N)^3}   \le  c   \|r\|_{L^2(\R^N)^3} ,
\eea
and hence $\omega \in \rho (\mathcal{L})$.
\epf

\subsection{Media Homogeneous in $\R^N_\pm$}\label{S:resolv-hom}

While in \cite{BDPW24} we considered the one- and two-dimensional cases with homogeneous media on either side of the interface, in this section we analyze the 3D case in the same situation $\perm_\pm(\cdot,\omega)=\perm_\pm(\omega)$ for those $\omega \in \C$ for which
\beq\label{E:omperm-nz}
\omega \notin \Omega_0=\{\omega \in D(\perm): \omega^2\perm_+(\omega)=0 \quad \text{or} \quad \omega^2\perm_-(\omega)= 0\}.
\eeq

Of course, the case of $\perm_\pm$ independent of $x_1$ is a special case of periodic $\perm_\pm$ studied in Sec. \ref{S:resolv-per}. However, as the calculations and results for the homogeneous case are very explicit, we present them here independently.

We recall the definition of the sets
\ben\label{E:Nred}
\cN^\mathrm{red}&=&\left\{\omega\in D(\perm)\setminus \Omega_0: \perm_+(\omega)+\perm_-(\omega)\neq 0, \frac{\omega^2\perm_+^2(\omega)}{\perm_+(\omega)+\perm_-(\omega)}, \frac{\omega^2\perm_-^2(\omega)}{\perm_+(\omega)+\perm_-(\omega)}\notin [0,\infty),\right. \nonumber \\
&&\left. \hspace{80pt}\frac{\omega^2\perm_+(\omega)\perm_-(\omega)}{\perm_+(\omega)+\perm_-(\omega)}\in (0,\infty)\right\}
\een
and
\beq\label{E:Mpm-red}
\cM_\pm^\mathrm{red}= \{\omega \in D(\perm)\setminus \Omega_0: \omega^2\perm_\pm(\omega)\in (0,\infty)\}
\eeq
from Theorem \ref{T:main-homog}. We show that the complement of the union of theses sets and of $\Omega_0$ lies in $\mathcal{S}$ and therefore in the resolvent set.

	\bprop\label{P:resolv-hom}
Assume $\perm_\pm(\cdot,\omega) =\perm_\pm(\omega)$. Then
$$D(\perm)\setminus (\cN^\mathrm{red} \cup \cM_+^\mathrm{red}\cup \cM_-^\mathrm{red}\cup \Omega_0)\subset \mathcal{S}.$$
\eprop

\bpf
We first define for each $k \in \R^{N-1}$ the set
\beq\label{E:Nkred}
\begin{aligned}
N_k^\mathrm{red}:= &\left\{\omega \in D(\perm)\setminus \Omega_0, |k|^2 -\omega^2\perm_+(\omega), \ |k|^2 -\omega^2\perm_-(\omega)\notin (-\infty,0], \right.\\
&\left. \ \text{and} \ \omega^2\perm_+(\omega)\perm_-(\omega)=|k|^2(\perm_+(\omega)+\perm_-(\omega))\right\}
\end{aligned}
\eeq
and claim that
\beq\label{E:Nred-eq}
\bigcup_{k\in \R^{N-1}}N_k^\mathrm{red} = \cN^\mathrm{red}.
\eeq
Note that these sets  will be also used in Section \ref{S:Weyl},  where for $\omega\in N_k^\mathrm{red}$ we will
generate Weyl sequences localized at the interface, while for $\omega\in \cM_\pm^\mathrm{red}$ Weyl sequences with their support moving to infinity in the $x_1$-direction will be found.

To show \eqref{E:Nred-eq}, first note that for any fixed $k \in \R^{N-1}$ and any $\omega \in N_k^\mathrm{red}$ we have $\perm_+(\omega)+\perm_-(\omega)\neq 0$ because otherwise the last equation in \eqref{E:Nkred} implies $\omega^2\perm_+(\omega)\perm_-(\omega)=0$, i.e., $\omega \in \Omega_0$. Solving this equation for $|k|^2$, we get $|k|^2 = \frac{\omega^2\perm_+(\omega)\perm_-(\omega)}{\perm_+(\omega)+\perm_-(\omega)}$. Hence, if $\omega \in N_k^\mathrm{red}$ for some $k \in\R^{N-1}$, then
\beq\label{E:frac-perm}
\frac{\omega^2\perm_+(\omega)\perm_-(\omega)}{\perm_+(\omega)+\perm_-(\omega)} \in \R_+ ,
\eeq
and
$$ \frac{\omega^2\perm_+(\omega)\perm_-(\omega)}{\perm_+(\omega)+\perm_-(\omega)} - \omega^2 \perm_\pm(\omega) \notin (-\infty,0]$$
for both cases $+$ and $-$. The value $0$ is excluded in \eqref{E:frac-perm} due to $\omega \notin \Omega_0$. After simplification, we get the conditions in the definition of $\cN^\mathrm{red}$. This shows $\bigcup_{k\in \R^{N-1}}N_k^\mathrm{red} \subset \cN^\mathrm{red}$.

For the opposite inclusion, we take $\omega \in \cN^\mathrm{red}$ and choose some $k\in \R^{N-1}$ with $|k|^2 = \frac{\omega^2\perm_+(\omega)\perm_-(\omega)}{\perm_+(\omega)+\perm_-(\omega)}$. Then, clearly, $\omega \in N_k^\mathrm{red}$ and \eqref{E:Nred-eq} follows.

Next, we proceed to show the inclusion in the statement of the proposition. We show that i)-iv) in Definition \ref{D:S} are satisfied by $\omega \in D(\perm)\setminus (\cN^\mathrm{red} \cup \cM_+^\mathrm{red}\cup \cM_-^\mathrm{red}\cup \Omega_0)$.

Condition i) is satisfied as $\omega\notin \cM_+^\mathrm{red}\cup \cM_-^\mathrm{red}\cup \Omega_0$ and $\perm_\pm$ is independent of $x$. For ii) and iii) first note that equations \eqref{3} and \eqref{4} are identical if $r=0$ and $\perm_\pm$ are independent of $x$. For a given $k \in \R^{N-1}$ a fundamental system suitable for ii) is
$$v^{(1)}_+(x)=e^{\mu_+x}, v^{(2)}_+(x)=e^{-\mu_+x} \quad \text{and} \quad v^{(1)}_-(x)=e^{\mu_-x}, v^{(2)}_-(x)=e^{-\mu_-x},$$
where
$$\mu_\pm :=\sqrt{|k|^2-\omega^2\perm_\pm(\omega)}.$$
Because $\omega \notin \cM_\pm^\mathrm{red}$, we have $\Real(\mu_\pm)>0$. Recall the definition of the complex square root from Sec. \ref{S:math-setting}. As a result, $v^{(1)}_+\notin L^2(\R_+)$ and $v^{(2)}_-\notin L^2(\R_-)$.

The estimates in ii) hold with
$$\nu_\pm :=\Real(\mu_\pm), \ C_{v,1,\pm}=C_{v,2,\pm}=1 \  \text{and} \  C_{v,1,\pm}'=C_{v,2,\pm}'=|\mu_\pm|.$$
For $d$ we get
$$d=-\omega^2\left(\frac{\perm_-(\omega)}{\mu_-}+\frac{\perm_+(\omega)}{\mu_+}\right)\neq 0$$
because $\omega \notin \Omega_0$ and because the expression in the parentheses vanishes if and only if $\omega^2\perm_+(\omega)\perm_-(\omega)=|k|^2(\perm_+(\omega)+\perm_-(\omega))$, i.e. only if $\omega \in N_k^\mathrm{red}$, see Remark 4.6 in \cite{BDPW24}.

As explained above, for iii) we can choose
$$w^{(1)}_\pm= v^{(1)}_\pm,  \ w^{(2)}_\pm=v^{(2)}_\pm, \ C_{w,1,\pm}=C_{v,1,\pm}, \ C_{w,2,\pm}=C_{v,2,\pm}, \ C'_{w,1,\pm}=C'_{v,1,\pm}, \ C'_{w,2,\pm}=C'_{v,2,\pm}, \text{ and } \widetilde{ \nu}_\pm=\nu_\pm.$$
For $\widetilde{d}$ we obtain
$$\widetilde{d}=-(\mu_++\mu_-)\neq 0$$
because $\Real(\mu_\pm)>0.$

It remains to show that the constants $\alpha_j^\pm, j=1,\dots, 6$, are bounded independently of $k$. First note that
$$\tau_\pm=-2\omega^2\frac{\perm_\pm}{\mu_\pm} \quad \text{and} \quad \widetilde{\tau}_\pm = -2\mu_\pm.$$
The following straightforward estimates are used below
$$
\begin{aligned}
	&\nu_\pm \gtrsim 1+|k|, \ 1+|k|\lesssim |\mu_\pm|\lesssim 1+|k|, \ |\perm_-\mu_++\perm_+\mu_-|\gtrsim  1+|k|, \\
	&|d|, |\tau_\pm|\gtrsim  (1+|k|)^{-1}, \ \text{and} \ |\widetilde{d}|, |\widetilde{\tau}_\pm|\gtrsim  1+|k|,
\end{aligned}
$$
where $\lesssim$ and $\gtrsim$ denote inequalities up to $k$-independent multiplicative constants. Their straightforward application produces
$$
\begin{aligned}
	\alpha_1^\pm &= \frac{1}{2|\omega|^2|\perm_\pm|}\frac{|\mu_\pm|}{\nu_\pm(1+|k|^2)}\lesssim \frac{1}{\nu_\pm (1+|k|)}\lesssim \frac{1}{1+|k|^2},\\
	\alpha_2^+ &=\frac{1}{|d|(1+|k|^2) \sqrt{\nu_+}}
	\left(1+\frac{1}{\sqrt{\nu_-}} + \frac{|\mu_+|+|\mu_-|}{|\tau_+|(1+|k|^2)\sqrt{\nu_+}}\right)\lesssim \frac{1}{1+|k|^{3/2}},\\
	\alpha_3^+ & = \frac{1}{(1+|k|^3)\sqrt{\nu_+}}\left[\frac{|\mu_+|}{|d|} \left(  1+\frac{1}{\sqrt{\nu_-}}  + \frac{|\mu_+|+|\mu_-|}{ |\tau_+| (1+|k|^2)\sqrt{\nu_+}}  \right)
	+  \frac{|\mu_+|+|\mu_-|}{|\tau_+| \sqrt{\nu_+}}
	\right]\lesssim \frac{1}{1+|k|^{3/2}},\\
	\alpha_4^\pm & = \frac{1+|k|}{|\widetilde{\tau}_\pm|\widetilde{ \nu}_\pm} \lesssim \frac{1}{1+|k|},\\
	\alpha_5^+ &=\frac{1+|k|}{|\widetilde{d}|\sqrt{\widetilde{\nu}_+}}\left(\frac{|\mu_+|+|\mu_-|}{|\widetilde{\tau}_+|\sqrt{\widetilde{\nu}_+}} + \frac{1}{\sqrt{\widetilde{\nu}_-}}\right) \lesssim \frac{1}{1+|k|},\\
	\alpha_6^+&=\frac{|\mu_+|}{|\widetilde{d}| \sqrt{\widetilde{\nu}_{+} }}   \left(   \frac{ |\mu_+|+|\mu_-| }{|\widetilde{\tau}_+|\sqrt{\widetilde{\nu}_-}}   +  \frac{1}{\sqrt{\widetilde{\nu}_{-} }}   \right) +  \frac{ |\mu_+|+|\mu_-|}{|\widetilde{\tau}_+| \widetilde{\nu}_+}  \lesssim \frac{1}{1+|k|},
\end{aligned}
$$
and analogously
$$
\alpha_2^- \lesssim  \frac{1}{1+|k|^{3/2}}, \quad \alpha_3^- \lesssim \frac{1}{1+|k|^{3/2}}, \quad \alpha_5^-, \alpha_6^- \lesssim \frac{1}{1+|k|}.
$$
\epf

\subsection{Media Periodic in $\R^N_\pm$}\label{S:resolv-per}

Here we consider the case of $\perm_\pm(\cdot,\omega)$ being $a_\pm$-periodic with $a_\pm>0$. Clearly, equation \eqref{4} for the component $w$ has the form of the Schr\"odinger equation \eqref{E:lV} if we set  $V=\omega^2\perm_\pm$. But also equation \eqref{3} for the component $v$ transforms to \eqref{E:lV} if we set  $V=\omega^2\perm_\pm -\frac{3}{4}\left(\frac{\perm_\pm'}{\perm_\pm}\right)^2 + \frac{1}{2}\frac{\perm_\pm''}{\perm_\pm}$, as we showed in Lemma \ref{lem:v-eqn}.  Therefore, results of Sec. \ref{S:per-estimates} for the periodic Schr\"odinger  equation play a central role here. 

In the case when the discriminant of the differential equation \eqref{3} on $\R_+$ does not lie in $[-2,2]$, the fundamental system $\psi_1, \psi_2$ from \eqref{eq:6} can be chosen as the fundamental system $v^{(1)}_+$ and $v^{(2)}_+$ in Def.~\ref{D:S}. This is because in this case $\Re(\kappa)>0$ in \eqref{eq:6}.

\begin{definition}\label{D:Sp}
	For $\perm_\pm(\cdot,\omega)$ being $a_\pm$-periodic with $a_\pm>0$ we define the set
	$$
	\begin{aligned}
		\mathcal{S}_p:=&\left\{\omega\in D(\perm)\setminus \Omega: D_+^{(v)}(k), D_-^{(v)}(k), D_+^{(w)}(k), D_-^{(w)}(k)\notin [-2,2], \right. \\
		& \left.\quad d(k)\neq0 \quad \hbox{ and } \quad \widetilde{d}(k)\neq 0 \quad \hbox{ for all } k\in \R^{N-1},\right. \\
		& \left.\quad \hbox{ and} \quad \perm_+(0,\omega)+\perm_-(0,\omega)\neq 0 \quad \text{or} \quad \perm_+'(0,\omega)-\perm_-'(0,\omega)\neq 0 \right\},
	\end{aligned}$$
	where $D_\pm^{(v)}(k)$ and $D_\pm^{(w)}(k)$ denote the discriminants of the homogeneous differential equations \eqref{3} and \eqref{4}, respectively, on the intervals $[0,a_+]$ and $[-a_-,0]$ according to the index $\pm$,  respectively.  
	The quantities $d$ and $\widetilde{d}$ were introduced in Definition \ref{D:S}.
\end{definition}

In view of Lemma \ref{lem:v-eqn} one may wonder whether it makes a difference if, in Definition \ref{D:Sp}, we consider the discriminants $D_\pm^{(v)}(k)$  based on the homogeneous Sturm-Liouville equation \eqref{3} (i.e.~\eqref{E:v}) for $v$ or the equivalent homogeneous Schr\"{o}dinger equation  \eqref{E:z} for 
\be\label{E:zvp}
z=\frac{\sqrt{\omega^2\perm}}{|k|^2-\omega^2\perm } v';
\ee
see Lemma \ref{lem:v-eqn}. The next lemma shows that this is not the case.

\begin{lem}
The monodromy matrices $\Phi_\pm(a_\pm)$ and $\widehat{\Phi}_\pm(a_\pm)$  of the differential equations \eqref{E:v} and \eqref{E:z} on $\R_\pm$, respectively, are related via $\widehat{\Phi}_\pm(a_\pm)=T_\pm\Phi_\pm(a_\pm)T_\pm^{-1}$, where 
$$T_\pm:=\left(\begin{array}{cc} 0 & \frac{\omega^2}{\sqrt{\omega^2\perm_\pm(0)}} \\ \sqrt{\omega^2\perm_\pm(0)} & -\frac{\omega^2\perm_\pm'(0)}{2\perm_\pm(0)\sqrt{\omega^2\perm_\pm(0)}} \end{array}  \right).$$
As a consequence, the eigenvalues (and thus the discriminants) of $\Phi_\pm(a_\pm)$ and $\widehat{\Phi}_\pm(a_\pm)$ coincide.
\end{lem}

\bpf
Using the notation $W_\pm=\omega^2\perm_\pm(\cdot,\omega)$, the fundamental matrix $\Phi_\pm(x)$ of \eqref{E:v} on $\R_\pm$ satisfying $\Phi_\pm(0)=I$ is of the form 
$$\Phi_\pm=\left(\begin{array}{cc} v^{(1)}_\pm & v^{(2)}_\pm  \\ \frac{W_\pm}{\omega^2\left(|k|^2-W_\pm\right)}v^{(1)'}_\pm & \frac{W_\pm}{\omega^2\left(|k|^2-W_\pm\right)}v^{(2)'}_\pm \end{array}  \right)$$
with suitably normalised fundamental solutions $v^{(1)}_\pm$, $v^{(2)}_\pm$. Noting Remark \ref{rem:root}, by \eqref{E:zvp},  a fundamental matrix  of \eqref{E:z} on $\R_\pm$ is given by
\bea
\widetilde{\Phi}_\pm&=& \left(\begin{array}{cc} z^{(1)}_\pm & z^{(2)}_\pm  \\ z^{(1)'}_\pm & z^{(2)'}_\pm  \end{array}  \right)\\
&=&\left(\begin{array}{cc} \frac{\sqrt{W_\pm}}{|k|^2-W_\pm } v^{(1)'}_\pm  & \frac{\sqrt{W_\pm}}{|k|^2-W_\pm } v^{(2)'}_\pm \\
 \left(\frac{\sqrt{W_\pm}}{\perm_\pm} \frac{\perm_\pm}{|k|^2-W_\pm }v^{(1)'}_\pm\right)'  &\left(\frac{\sqrt{W_\pm}}{\perm_\pm} \frac{\perm_\pm}{|k|^2-W_\pm }v^{(2)'}_\pm\right)' \end{array}  \right)\\
&=& \left(\begin{array}{cc} \frac{\sqrt{W_\pm}}{|k|^2-W_\pm } v^{(1)'}_\pm  & \frac{\sqrt{W_\pm}}{|k|^2-W_\pm } v^{(2)'}_\pm  \\
 \frac{\sqrt{W_\pm}}{\perm_\pm} \perm_\pm v^{(1)}_\pm + \left(\frac{\omega^2}{\sqrt{W_\pm}}\right)'  \frac{\perm_\pm}{|k|^2-W_\pm } v^{(1)'}_\pm  & \frac{\sqrt{W_\pm}}{\perm_\pm} \perm_\pm v^{(2)}_\pm + \left(\frac{\omega^2}{\sqrt{W_\pm}}\right)'  \frac{\perm_\pm}{|k|^2-W_\pm } v^{(2)'}_\pm\end{array}  \right)\\
&=&\left(\begin{array}{cc} \frac{\sqrt{W_\pm}}{|k|^2-W_\pm } v^{(1)'}_\pm  & \frac{\sqrt{W_\pm}}{|k|^2-W_\pm } v^{(2)'}_\pm  \\
 \sqrt{W_\pm} v^{(1)}_\pm - \frac{W_\pm'}{2W_\pm\sqrt{W_\pm}}  \frac{W_\pm}{|k|^2-W_\pm } v^{(1)'}_\pm  &  \sqrt{W_\pm} v^{(2)}_\pm - \frac{W_\pm'}{2W_\pm\sqrt{W_\pm}}  \frac{W_\pm}{|k|^2-W_\pm } v^{(2)'}_\pm\end{array}  \right)\\
&=&\left(\begin{array}{cc} 0 & \frac{\omega^2}{\sqrt{W_\pm}} \\ \sqrt{W_\pm} & -\frac{\omega^2 W_\pm'}{2W_\pm\sqrt{W_\pm}} \end{array}  \right) 
\left(\begin{array}{cc} v^{(1)}_\pm & v^{(2)}_\pm  \\ \frac{W_\pm}{\omega^2\left(|k|^2-W_\pm\right)}v^{(1)'}_\pm & \frac{W_\pm}{\omega^2\left(|k|^2-W_\pm\right)}v^{(2)'}_\pm \end{array}  \right) \ =\ \cT_\pm \Phi_\pm,
\eea
where $\cT_\pm(x):=\left(\begin{array}{cc} 0 & \frac{\omega^2}{\sqrt{W_\pm(x)}} \\ \sqrt{W_\pm(x)} & -\frac{\omega^2 W_\pm'(x)}{2W_\pm(x)\sqrt{W_\pm(x)}} \end{array}  \right) $.
Thus the fundamental matrix $\widehat{\Phi}_\pm$ of \eqref{E:z} satisfying $\widehat{\Phi}_\pm(0)=I$ is given by 
$$\widehat{\Phi}_\pm=\widetilde{\Phi}_\pm\widetilde{\Phi}_\pm(0)^{-1}= \cT_\pm \Phi_\pm \Phi_\pm(0)^{-1}   \cT_\pm(0)^{-1} = \cT_\pm \Phi_\pm T_\pm^{-1}.$$
Since $\cT_\pm(a_\pm)=\cT_\pm(0)=T_\pm$ due to the $a_\pm$-periodicity of $\perm_\pm$, we have $\widehat{\Phi}_\pm(a_\pm)=T_\pm\Phi_\pm(a_\pm)T_\pm^{-1}$.
\epf

\begin{remark}\label{rem:Sp}
        The definition of $\Sc_p$ includes the conditions $d(k), \tilde{d}(k)\neq 0 $ for all $k$. Together
		 with the asymptotic results in Lemmas \ref{L:est-v-resolv} and \ref{L:est-w-resolv} for $|k|\to \infty$, these conditions imply that $d$ and $\widetilde{d}$ are bounded away from $0$ on the whole of $\R^{N-1}$.
\end{remark}

Recall that the set $\Sc$ is a subset of the resolvent set. Hence, the next result shows that $\Sc_p$ lies in the resolvent set.
\bprop\label{P:Sp-resolv}
We have $\mathcal{S}_p\subset \mathcal{S}$.
\eprop
\bpf
We need to check that for points $\omega \in \mathcal{S}_p$ there are fundamental systems of \eqref{3} and \eqref{4} which satisfy the properties in Definition \ref{D:S}. Note that \eqref{E:d} and \eqref{E:dtil} are satisfied by Remark \ref{rem:Sp}.
The proof is based on Lemmas \ref{L:est-v-resolv} and \ref{L:est-w-resolv} and subsequent estimates of the constants $\alpha^\pm_j, j=1, \dots, 6$.

\begin{lem}\label{L:est-v-resolv}
	Let $\omega\in D(\perm)\setminus \Omega$ and assume that $\perm_\pm(\cdot,\omega)$ are $a_\pm$-periodic with $a_\pm>0$ and satisfy 
	\beq\label{E:perm-sum}
	\perm_+(0,\omega)+\perm_-(0,\omega)\neq 0 \quad \text{or} \quad \perm_+'(0,\omega)-\perm_-'(0,\omega)\neq 0.
	\eeq
	If $D^{(v)}_+(k), D^{(v)}_-(k) \notin [-2,2]$ for all $ k \in \R^{N-1}$, then there exist
	fundamental systems $\{v^{(1)}_+, v^{(2)}_+\}$ and $\{v^{(1)}_-, v^{(2)}_-\}$ of the homogeneous version of equation
	\eqref{3} on $[0, \infty)$ and on $(-\infty, 0]$, respectively, such
	that, for some ($k$-dependent) constants $\nu_\pm > 0$,
	$C_{v,1,\pm}$, $C'_{v,1,\pm}$, $C_{v,2,\pm}$, $C'_{v,2,\pm} \ge 0$,
	\begin{align*}
		|v^{(1)}_\pm(x)| &\le C_{v,1,\pm}\,e^{\nu_\pm x},
		&|v^{(1)'}_\pm(x)| &\le C'_{v,1,\pm}\,e^{\nu_\pm x},
		\\
		|v^{(2)}_\pm(x)| &\le C_{v,2,\pm}\,e^{-\nu_\pm x},
		&|v^{(2)'}_\pm(x)| &\le C'_{v,2,\pm}\,e^{-\nu_\pm x}
	\end{align*}
	for all $x \in \mathbb R_\pm$
	and
	$v^{(1)}_+ \notin L^2(\R_+)$, $v^{(2)}_- \notin L^2(\R_-)$.
	Consider $d$ and $\tau_\pm$ defined in \eqref{E:d} and \eqref{E:D_Dtil} resp., and define 
\beq\label{E:delta}
\delta:=\dist(\Ran(\omega^2 \perm(\cdot,\omega), \R_+).
\eeq 
(Note that $\delta>0$ since $\omega \notin \Omega$.)

Then, as $|k|\to\infty$,
	\begin{align*}
		C_{v,1,\pm} &= C_{v,2,\pm} = \frac{|k|}{\sqrt{\delta}} + O(1),
		\\
		C'_{v,1,\pm} &= C'_{v,2,\pm} = \frac{|k|^2}{\sqrt{\delta}} + O(|k|),
		\\
		\tau_\pm &= 2|k| + O(1),\\
		\nu_\pm & = |k| + O\left(\frac{1}{|k|}\right),
	\end{align*}
	and
	\begin{align*}
		d = \frac{\perm_-(0) + \perm_+(0)}{\sqrt{\perm_+(0)\,\perm_-(0)}}\,|k| + O(1);
	\end{align*}
	if $\perm_-(0) + \perm_+(0) = 0$, then we have the more precise asymptotics
	\begin{equation*}
		d = \pm \ri\,\frac {\perm_+'(0) - \perm_-'(0)}{2\,\perm_+(0)} + O\left(\frac{1}{|k|}\right).
	\end{equation*}
\end{lem}
\begin{proof}
	To shorten the notation we define again $$W_\pm:= \omega^2\perm_\pm.$$
	We first focus on the interval $[0, \infty)$. Recall from Lemma \ref{lem:v-eqn} that equation (\ref{3}) with $r=0$ can be rewritten as \eqref{E:z}, i.e.,
	equation (\ref{E:lV}) for $z$ with potential $V=W_\pm -\frac{3}{4}\left(\frac{W_\pm'}{W_\pm}\right)^2 + \frac{1}{2}\frac{W_\pm''}{W_\pm}$ and $l=|k|$.
	We have Floquet solutions of
	equation (\ref{E:lV}) on $\mathbb R_+$  as in Lemma \ref{lem:LemmaA} and Lemma \ref{lem:LemmaB}. Due to the condition on the discriminant $D^{(v)}_+$, these solutions have the required exponential
	behaviour with $\nu_+ = \Real \kappa(|k|) > 0$, see equation \eqref{eq:6}.
	
	Now consider $|k| \ge k_0$, where $k_0$ is the larger of the numbers $l_0$ of
	Lemma \ref{lem:LemmaA} and Lemma \ref{lem:LemmaB} applied to equation (\ref{E:lV}) on
	$\mathbb R_+$. Lemma \ref{lem:LemmaA} immediately gives the asymptotics $\nu_+ = |k| + O\left(\frac{1}{|k|}\right).$

	The solution $z(x) = e^{\kappa(|k|)\,x}\,p_1(x)$ $(x \ge 0)$ gives the solution $v^{(1)}_+$ of equation (\ref{3}):
	\begin{equation*}
		v^{(1)}_+(x) = e^{\kappa(|k|)\,x}\,\frac 1 {\sqrt{W_+}}\,\left(\kappa(|k|)\,p_1(x) + p_1'(x) + \frac{W_+'}{2 W_+}\,p_1(x)\right)
	\end{equation*}
	with derivative
	\begin{equation*}
		v^{(1)'}_+(x) = \frac{|k|^2 - W_+}{\sqrt{W_+}}\,e^{\kappa(|k|)\,x}\,p_1(x).
	\end{equation*}
	Consequently we have by Lemma \ref{lem:LemmaA} and Lemma \ref{lem:LemmaB} for sufficiently large $|k|$, with $\nu_+ = \Real\kappa(|k|)$
	\begin{align*}
		|v^{(1)}_+(x)| &\le e^{\nu_+ x}\,\frac 1{|\sqrt{W_+}|}\,\left(\left(|\kappa(|k|)| + \frac 1 2\,\left|\frac{W_+'}{W_+}\right|\right)|p_1(x)| + |p_1'(x)|\right)
		\\
		&\le e^{\nu_+ x}\,\frac 1{\sqrt{\delta}}\left(\left(|k| + \frac{\c}{|k|} + \frac 1 2 \left\|\frac{W_+'}{W_+}\right\|_\infty\right) \left(1 + \frac{\c}{|k|}\right) + \frac{\c}{|k|}\right)
	\end{align*}
	which gives the asserted asymptotics of $C_{v,1,+}$. Similarly,
	\begin{align*}
		|v^{(1)'}_+(x)| &\le \frac{||k|^2 - W_+|}{|\sqrt{W_+}|}\,e^{\nu_+ x}\,|p_1(x)|
		\le e^{\nu_+ x}\,\frac{|k|^2 + \|W_+\|_\infty}{\sqrt{\delta}}\,\left(1 + \frac{\c}{|k|}\right)
	\end{align*}
	for all $x \ge 0$, proving the asymptotics of $C'_{v,1,+}$.
	
	Next, the solution
	$z(x) = e^{-\kappa(|k|)\,x}\,p_2(x)$
	gives another solution $v^{(2)}_+$ of equation \eqref{3} for $x\geq 0$:
	\begin{align*}
		v^{(2)}_+(x) &= e^{-\kappa(|k|)\,x}\,\frac 1 {\sqrt{W_+}}\left(-\kappa(|k|)\,p_2(x) + p_2'(x) + \frac{W_+'}{2W_+}\,p_2(x)\right),
		\\
		v^{(2)'}_+(x) &= \frac{|k|^2 - W_+}{\sqrt{W_+}}\,e^{-\kappa(|k|)\,x}\,p_2(x).
	\end{align*}
	Again we find for sufficiently large $|k|$ and $x\geq 0$ that
	\begin{align*}
		|v^{(2)}_+(x)| &\le e^{-\nu_+ x}\,\frac 1{|\sqrt{W_+}|}\left(\left(|\kappa(|k|)| + \frac 1 2 \left|\frac{W_+'}{W_+}\right|\right) |p_2(x)| + |p_2'(x)|\right)
		\\
		&\le e^{-\nu_+ x}\,\frac 1{\sqrt \delta}\,\left(\left(|k| + \frac{\c}{|k|} + \frac 1 2\,
		\left\|\frac{W_+'}{W_+}\right\|_\infty\right) \left(1 + \frac{\c}{|k|}\right) + \frac{\c}{|k|} \right)
	\end{align*}
	and
	\begin{equation*}
		|v^{(2)'}_+(x)| \le e^{-\nu_+ x}\,\frac{|k|^2 + \|W_+\|_\infty}{\sqrt{\delta}}\,\left(1 + \frac{\c}{|k|}\right).
	\end{equation*}
	These last two inequalities prove the asserted asymptotics of $C_{v,2,+}$ and $C'_{v,2,+}$.
	The Wronski determinant is
	\begin{align*}
		\tau_+ &= (v^{(1)}_+\,v^{(2)'}_+ - v_+^{(1)'}\,v^{(2)}_+)(0)\,\frac{W_+(0)}{|k|^2 - W_+(0)}
		\\
		&= \frac 1{\sqrt{W_+(0)}} \left(\kappa(|k|)\,p_1(0) + p_1'(0) + \frac{W_+'(0)}{2 W_+(0)}\,p_1(0)\right)\frac{|k|^2 - W_+(0)}{\sqrt{W_+(0)}}\,p_2(0)\,\frac{W_+(0)}{|k|^2 - W_+(0)}
		\\&\qquad
		- \frac{|k|^2 - W_+(0)}{\sqrt{W_+(0)}}\,p_1(0)\,\frac 1{\sqrt{W_+(0)}}\left(-\kappa(|k|)\,p_2(0) + p_2'(0) + \frac{W'_+(0)}{2 W_+(0)}\,p_2(0)\right)\frac{W_+(0)}{|k|^2 - W_+(0)}
		\\
		&= 2\kappa(|k|)\,p_1(0)\,p_2(0) + p_1'(0)\,p_2(0) - p_1(0)\,p_2'(0)
		\\
		&= 2|k|\,\left(1 + O\left(\frac 1{|k|^2}\right)\right)\,\left(1 + O\left(\frac{1}{|k|}\right)\right)^2 + 2\,\left(1 + O\left(\frac{1}{|k|}\right)\right)\,O\left(\frac{1}{|k|}\right)
		\\
		&= 2|k| + O(1).
	\end{align*}
	Analogous estimates apply to $v^{(1)}_-$, $v^{(2)}_-$, $v^{(1)'}_-$, $v^{(2)'}_-$, and $\tau_-$ but
	with different $\kappa(|k|)$ (which has real part $\nu_-$) and $W_-$ instead of
	$W_+$, thus proving the asserted asymptotics for $C_{v,1,-}$, $C'_{v,1,-}$, $C_{v,2,-}$ and $C'_{v,2,-}$.
	Finally, denoting by $\widehat p_1$, $\widehat p_2$ the periodic functions which arise
	on $(-\infty, 0]$ as counterparts of $p_1, p_2$ on $[0, \infty)$, we find
	\begin{align*}
		d &= -v^{(2)}_+(0)\,v^{(1)'}_-(0)\,\frac{W_-(0)}{|k|^2 - W_-(0)}
		+ v^{(1)}_-(0)\,v^{(2)'}_+(0)\,\frac{W_+(0)}{|k|^2 - W_+(0)}
		\\
		&= -\frac 1 {\sqrt{W_+(0)}}\left(-\kappa_+(|k|)\,p_2(0) + p_2'(0) + \frac 1 2\,
		\frac{W'_+(0)}{W_+(0)}\,p_2(0)\right)\frac{|k|^2 - W_-(0)}{\sqrt{W_-(0)}}\,\widehat p_1(0)\,\frac{W_-(0)}{|k|^2 - W_-(0)}
		\\&\qquad
		+ \frac 1 {\sqrt{W_-(0)}}\left(\kappa_-(|k|)\,\widehat p_1(0) + \widehat p_1'(0) + \frac 1 2\,\frac{W'_-(0)}{W_-(0)}\,\widehat p_1(0)\right) \frac{|k|^2 - W_+(0)}{\sqrt{W_+(0)}}\,p_2(0)\,\frac{W_+(0)}{|k|^2 - W_+(0)}
		\\
		&= \left(\frac{\sqrt{W_-(0)}}{\sqrt{W_+(0)}}\,\kappa_+(|k|) + \frac{\sqrt{W_+(0)}}{\sqrt{W_-(0)}}\,\kappa_-(|k|)\right) \widehat p_1(0)\,p_2(0)
		- \frac{\sqrt{W_-(0)}}{\sqrt{W_+(0)}}\,p_2'(0)\,\widehat p_1(0)
		+ \frac{\sqrt{W_+(0)}}{\sqrt{W_-(0)}}\,\widehat p_1'(0)\,p_2(0)
		\\&\qquad
		-\frac{\sqrt{W_-(0)}}{\sqrt{W_+(0)}}\,\frac 1 2\,\frac{W'_+(0)}{W_+(0)}\,p_2(0)\,\widehat p_1(0)
		+ \frac{\sqrt{W_+(0)}}{\sqrt{W_-(0)}}\,\frac 1 2\,\frac{W'_-(0)}{W_-(0)}\,\widehat p_1(0)\,p_2(0).
	\end{align*}
	If $W_+(0) + W_-(0) \neq 0$, we hence obtain
	\begin{align*}
		d&=\left(\frac{\sqrt{W_-(0)}}{\sqrt{W_+(0)}}\,|k|\,\left(1 + O\left(\frac 1{|k|^2}\right)\right)
		+ \frac{\sqrt{W_+(0)}}{\sqrt{W_-(0)}}\,|k|\,\left(1 + O\left(\frac 1{|k|^2}\right)\right) \right) \left(1 + O\left(\frac{1}{|k|}\right)\right)^2
		\\&\qquad
		- \frac{\sqrt{W_-(0)}}{\sqrt{W_+(0)}}\,\left(1 + O\left(\frac{1}{|k|}\right)\right)\,O\left(\frac{1}{|k|}\right)
		+ \frac{\sqrt{W_+(0)}}{\sqrt{W_-(0)}}\,\left(1 + O\left(\frac{1}{|k|}\right)\right)\,O\left(\frac{1}{|k|}\right)
		\\&\qquad
		- \frac{\sqrt{W_-(0)}}{\sqrt{W_+(0)}}\,\frac 1 2\,\frac{W'_+(0)}{W_+(0)}\,\left(1 + O\left(\frac{1}{|k|}\right)\right)^2
		+ \frac{\sqrt{W_+(0)}}{\sqrt{W_-(0)}}\,\frac 1 2\,\frac{W'_-(0)}{W_-(0)}\,\left(1 + O\left(\frac{1}{|k|}\right)\right)^2
		\\
		&= \frac{W_-(0) + W_+(0)}{\sqrt{W_+(0)\,W_-(0)}}\,|k| + O(1).
	\end{align*}
	In the case $W_+(0) + W_-(0) = 0$, we have
	\begin{equation*}
		\frac{\sqrt{W_-(0)}}{\sqrt{W_+(0)}} = -
		\frac{\sqrt{W_+(0)}}{\sqrt{W_-(0)}} \in \{\ri,-\ri\},
	\end{equation*}
	so
	\begin{align*}
		d&= \pm \ri \Bigg( (\kappa_+(|k|) - \kappa_-(|k|)) \,\widehat p_1(0)\,p_2(0)
		- p_2'(0)\,\widehat p_1(0) - \widehat p_1'(0)\,p_2(0)
		\\
		&\qquad\qquad\qquad
		- \frac 1 2 \left(\frac{W_+'(0)}{W_+(0)} + \frac{W_-'(0)}{W_-(0)}\right) p_2(0) \widehat p_1(0) \Bigg)
		\\
		&= \mp \ri\,\frac{W_+'(0) - W_-'(0)}{2\,W_+(0)}\,\left(1 + O\left(\frac{1}{|k|}\right)\right) + O\left(\frac{1}{|k|}\right).
	\end{align*}
\end{proof}

Lemma \ref{L:est-v-resolv} shows that $\omega \in \Sc_p$ satisfies (i) and the estimates of the fundamental system in (ii) of Definition \ref{D:S}. Next, we prove that also for equation \eqref{4} there is a fundamental system satisfying the corresponding properties of Definition \ref{D:S}. The proof is less technical than for the $v$-component treated by Lemma \ref{L:est-v-resolv}.
\begin{lem}\label{L:est-w-resolv}
	Let $\omega\in D(\perm)\setminus \Omega$ and assume \eqref{E:ass-perm}.
	If the discriminants $D^{(w)}_+(k), D^{(w)}_-(k) \notin [-2, 2]$ for all $k \in \R^{N-1}$, then
	there exist fundamental systems $\{w^{(1)}_+, w^{(2)}_+\}$ and
	$\{w^{(1)}_-, w^{(2)}_-\}$ of equation \eqref{4} on
	$[0, \infty)$ and on $(-\infty, 0]$, respectively such that, for some
	($k$-dependent) constants
	$\widetilde\nu_\pm > 0$, $C_{w,1,\pm}, C'_{w,1,\pm}, C_{w,2,\pm}, C'_{w,2,\pm} \ge 0$,
	\begin{align*}
		|w^{(1)}_\pm(x)| &\le C_{w,1,\pm}\,e^{\widetilde\nu_\pm x},
		\qquad
		|w^{(1)'}_\pm(x)| \le C'_{w,1,\pm}\,e^{\widetilde\nu_\pm x},
		\\
		|w^{(2)}_\pm(x)| &\le C_{w,2,\pm}\,e^{-\widetilde\nu_\pm x},
		\qquad
		|w^{(2)'}_\pm(x)| \le C'_{w,2,\pm}\,e^{-\widetilde\nu_\pm x}
	\end{align*}
	for all $x \in \mathbb R_\pm$
	and
	$w^{(1)}_+ \notin L^2(\R_+)$, $w^{(2)}_- \notin L^2(\R_-)$.
	Consider $\widetilde{d}$ and $\widetilde{\tau}_\pm$ defined in \eqref{E:dtil} and \eqref{E:D_Dtil} resp. Then, as $k \to\infty$,
	\begin{align*}
		C_{w,1,\pm} &= C_{w,2,\pm} = 1 + O\left(\frac{1}{|k|}\right),
		\\
		C'_{w,1,\pm} &= C'_{w,2,\pm} = |k| + O(1),
		\\
		\widetilde \tau_\pm &= -2|k| + O(1),\\
		\widetilde{\nu}_\pm &= |k| + O\left(\frac{1}{|k|}\right),
	\end{align*}
	and
	\begin{equation*}
		\widetilde d = -2|k| + O(1).
	\end{equation*}
\end{lem}

\begin{proof}
	Just like in Lemma \ref{L:est-v-resolv}  we set $\widetilde\nu_\pm = \Real \kappa_\pm(|k|) > 0$, see equation \eqref{eq:6}. Note that equation \eqref{4}  has the form of (\ref{E:lV}) with $V = \omega^2\perm_\pm(\omega)$ and $l = |k|^2$. The asymptotics of $\widetilde{\nu}_\pm$ follow immediately from Lemma \ref{lem:LemmaA}.
	
	Next, we focus on the interval $\mathbb R_+$. We have, with periodic
	functions $p_1, p_2$ as in Lemma \ref{lem:LemmaB},
	\begin{equation*}
		w^{(1)}_+(x) = e^{\kappa_+(|k|) x}\,p_1(x),
		\quad
		w^{(1)'}_+(x) = (\kappa_+(|k|)\,p_1(x) + p_1'(x))\,e^{\kappa_+(|k|)\,x}
		\quad
		(x \ge 0),
	\end{equation*}
	and hence
	\begin{align*}
		|w^{(1)}_+(x)| &= e^{\widetilde\nu_+ x}\,|p_1(x)| \le e^{\widetilde\nu_+ x}\,(1 + |p_1(x) - 1|)
		\\
		&\le e^{\widetilde\nu_+ x}\,\left(1 + \frac{\c}{|k|}\right),
		\\
		|w^{(1)'}_+(x)| &\le e^{\widetilde\nu_+ x}\,(|\kappa_+(|k|)|\,|p_1(x)| + |p_1'(x)|)
		\\
		&\le e^{\widetilde\nu_+ x}\,\left(\left(|k| + \frac{\c}{|k|}\right)\,\left(1 + \frac{\c}{|k|}\right) + \frac{\c}{|k|}\right)
		\le e^{\widetilde\nu_+ x}\,(|k| + \c)
	\end{align*}
	for all $x \le 0$ by Lemma \ref {lem:LemmaA} and \ref{lem:LemmaB}.
	Similarly,
	\begin{equation*}
		w^{(2)}_+(x) = e^{-\kappa_+(|k|)\,x}\,p_2(x), \quad
		w^{(2)'}_+(x) = (-\kappa_+(|k|)\,p_2(x) + p_2'(x))\,e^{-\kappa_+(|k|)\,x}
		\quad (x \ge 0),
	\end{equation*}
	so
	\begin{align*}
		|w^{(2)}_+(x)| &= e^{-\widetilde\nu_+ x}\,|p_2(x)| \le e^{-\widetilde\nu_+ x}\,\left(1 + \frac{\c}{|k|}\right),
		\\
		|w^{(2)'}_+(x)| &\le e^{-\widetilde\nu_+ x}\,(|\kappa_+(|k|)|\,|p_2(x)| + |p_2'(x)|)
		\\
		&\le e^{-\widetilde\nu_+ x}\,\left(\left(|k| + \frac{\c}{|k|}\right)\,\left(1 + \frac{\c}{|k|}\right) + \frac{\c}{|k|}\right)
		\le e^{-\widetilde\nu_+ x}\,(|k| + \c)
	\end{align*}
	for all $x \ge 0$. These estimates prove the asymptotics of $C_{w,1,+}$, $C'_{w,1,+}$, $C_{w,2,+}$ and $C'_{w,2,+}$.
	For the Wronski determinant we obtain
	\begin{align*}
		\widetilde \tau_+ &= w^{(1)}_+(0) w^{(2)'}_+(0) - w^{(1)'}_+(0)  w^{(2)}_+(0)
		\\
		&= -2\kappa_+(|k|)\,p_1(0)\,p_2(0) + p_1(0)\,p_2'(0) - p_1'(0)\,p_2(0)
		\\&= -2|k|\,\left(1 + O\left(\frac{1}{|k|^2}\right)\right)\,\left(1 + O\left(\frac{1}{|k|}\right)\right)^2
		+ 2\,\left(1 + O\left(\frac{1}{|k|}\right)\right)\,O\left(\frac{1}{|k|}\right)
		\\
		&= -2|k| + O(1).
	\end{align*}
	Analogous estimates apply to $w^{(1)}_-$, $w^{(2)}_-$, $w^{(1)'}_-$, $w^{(2)'}_-$, and $\widetilde{\tau}_-$ but with $\widetilde\nu_- = \Real \kappa_-(|k|)$.
	Finally,
	\begin{align*}
		\widetilde d &= - w^{(2)}_+(0)\,w^{(1)'}_-(0) + w^{(1)}_-(0)\,w^{(2)'}_+(0)
		\\
		&= -\left(1 + O\left(\frac{1}{|k|}\right)\right)\,\left(|k|\,\left(1 + O\left(\frac 1{|k|^2}\right)\right)\,\left(1 + O\left(\frac{1}{|k|}\right)\right) + O\left(\frac{1}{|k|}\right)\right)
		\\
		&\qquad
		+ \left(-|k|\,\left(1 + O\left(\frac 1{|k|^2}\right)\right)\,\left(1 + O\left(\frac{1}{|k|}\right)\right) + O\left(\frac{1}{|k|}\right)\right)\,\left(1 + O\left(\frac{1}{|k|}\right)\right)
		\\
		&= -2|k| + O(1).
	\end{align*}
\end{proof}

To finish the proof of Proposition \ref{P:Sp-resolv}, it remains to provide (in 1-6 below) $k$-independent bounds of $\alpha^\pm_j, j=1,\dots,6$  from Definition \ref{D:S}.   All the asymptotic expansions below are for $|k|\to\infty$. We use $\delta$ as defined in \eqref{E:delta}.

\noindent 1)
\begin{equation*}
	\alpha_1^\pm = \frac{(\frac{|k|}{\sqrt \delta} + O(1))^2}{(2|k| + O(1))\,(1 + |k|^2)\,(|k| + O\left(\frac{1}{|k|}\right))}
	= \frac 1{2 \delta\,|k|^2}\,(1 + o(1)).
\end{equation*}

\medskip\noindent
2)
If $\perm_+(0) + \perm_-(0) \neq 0$, then
\begin{align*}
	\alpha_2^\pm&= \frac{\frac{|k|}{\sqrt{\delta}} + O(1)}{\left(\left|\frac{\perm_+(0) + \perm_-(0)}{\sqrt{\perm_+(0) \perm_-(0)}}\right| \,|k| + O(1)\right)\,(1 + |k|^2)\,\sqrt{|k|}\,\left(1 + O\left(\frac{1}{|k|}\right)\right)}
	\times
	\\
	&\quad\times
	\left( \left(\frac{|k|}{\sqrt{\delta}} + O(1)\right)\,\left(1 + \frac 1{\sqrt{|k|}\,\left(1 + O\left(\frac{1}{|k|}\right)\right)}\right)\right.
	+ \left. \frac{2 (\frac{|k|}{\sqrt{\delta}} + O(1))^2\,(\frac{|k|^2}{\sqrt{\delta}} + O (|k|))}{(2 |k| + O(1))\,(1 + |k|^2)\,\sqrt{|k|}\,\left(1 + O\left(\frac{1}{|k|}\right)\right)} \right)
	\\
	&= \frac{|\sqrt{\perm_+(0) \perm_-(0)}|}{\sqrt{\delta}\,|\perm_+(0) + \perm_-(0)|}\,|k|^{-5/2}\,(1 + o(1))
	\left(\frac{|k|}{\sqrt{\delta}} + o(|k|) + \frac{\sqrt{|k|}}{\sqrt{\delta}^3}\,(1 + o(1))\right)
	\\
	&= \frac{|\sqrt{\perm_+(0) \perm_-(0)}|}{\delta\,|\perm_+(0) + \perm_-(0)|}\,|k|^{-3/2}\,(1 + o(1)).
\end{align*}
If $\perm_+(0) + \perm_-(0) = 0$ and $\perm_+'(0) - \perm_-'(0) \neq 0$, then
\begin{align*}
	&\alpha_2^+= \frac{2\,|\perm_+(0)|}{\delta\,|\perm_+'(0) - \perm_-'(0)|}\,|k|^{-1/2}\,(1 + o(1)).
\end{align*}


\medskip\noindent
3)
If $\perm_+(0) + \perm_-(0) \neq 0$, then
\begin{align*}
	\alpha_3^\pm&= \frac 1{(1+|k|^3)\,\sqrt{|k|\,(1 + O(\frac 1 {|k|^2}))}}
	\left(\frac{\frac{|k|^2}{\sqrt{\delta}} + O(|k|)}{\left|\frac{\perm_-(0)+\perm_+(0)}{\sqrt{\perm_-(0)\perm_+(0)}}\right|\,|k| + O(1)}
	\left(\left(\frac{|k|}{\sqrt{\delta}} + O(1)\right)\,\left(1 + \frac 1{\sqrt{|k|\,(1 + O(\frac 1{|k|^2}))}}\right) \right. \right.
	\\&\qquad
	+ \left.  \frac{2\,(\frac{|k|}{\sqrt{\delta}} + O(1))^2\,(\frac{|k|^2}{\sqrt{\delta}} + O(|k|))}
	{(2|k| + O(1))\,(1 + |k|^2)\,\sqrt{|k|\,(1 + O(\frac 1{|k|^2}))}} \right)
	+ \left. \frac{2 (\frac{|k|}{\sqrt{\delta}} + O(1))\,(\frac {|k|^2}{\sqrt{\delta}} + O(|k|))}
	{(2 |k| + O(1))\,\sqrt{|k|\,(1 + O(\frac 1{|k|^2}))}} \right)
	\\
	&= \frac{1 + o(1)}{|k|^{7/2}}\left(\frac{|k| +O(1)}{\sqrt{\delta}\left|\frac{\perm_+(0)+\perm_-(0)}{\sqrt{\perm_+(0)\perm_-(0)}}\right|}
	\left(\frac{|k|}{\sqrt{\delta}}\,(1 + o(1)) +\frac{\sqrt{|k|}\,(1 + o(1))}{\sqrt{\delta}^3} \right)
	+ \frac{|k|^{3/2}}{\delta}\,(1 + o(1)) \right)
	\\
	&= \frac{|\sqrt{\perm_+(0) \perm_-(0)}|}{\delta\,|\perm_+(0) + \perm_-(0)|}\,|k|^{-3/2}\,(1 + o(1)).
\end{align*}
If $\perm_+(0) + \perm_-(0) = 0$ and $\perm_+'(0) - \perm_-'(0) \neq 0$, then
\begin{align*}
	\alpha_3^+&= \frac{2\,|\perm_+(0)|}{\delta\,|\perm_+'(0) - \perm_-'(0)|}\,|k|^{-1/2}\,(1 + o(1)).
\end{align*}


\medskip\noindent
4)
\begin{equation*}
	\alpha_4^\pm=
	\frac{\left(1 + O\left(\frac{1}{|k|}\right)\right)^2}{(2|k| + O(1))\,\left(|k| + O\left(\frac{1}{|k|}\right)\right)}\,(1 + |k|)
	= \frac 1{2|k|}\,(1 + o(1)).
\end{equation*}

\medskip\noindent
5)
\begin{align*}
	\alpha_5^\pm&= \frac{\left(1 + O\left(\frac{1}{|k|}\right)\right)\,(1+|k|)}{(2|k| + O(1)) \sqrt{|k| + O\left(\frac{1}{|k|}\right)}}
	\left(2\,\frac{\left(1 + O\left(\frac{1}{|k|}\right)\right)\,(|k| + O(1))}{(2|k| + O(1)) \sqrt{|k| + O\left(\frac{1}{|k|}\right)}}\,\left(1 + O\left(\frac{1}{|k|}\right)\right) + \frac{1 + O\left(\frac{1}{|k|}\right)}{\sqrt{|k| + O\left(\frac{1}{|k|}\right)}} \right)
	\\
	&= \frac{1 + o(1)}{2\sqrt{|k|}} \left(\frac{1 +o(1)}{\sqrt{|k|}} + \frac{1 + o(1)}{\sqrt{|k|}}\right)
	= \frac{1}{|k|}\,(1 + o(1)).
\end{align*}


\medskip\noindent
6)
\begin{align*}
	\alpha_6^\pm
	&= \frac{|k| + O(1)}{(2|k| + O(1)) \sqrt{|k| + O (\frac{1}{|k|})}}
	\left(\frac{2\,\left(1 + O\left(\frac{1}{|k|}\right)\right)\,(|k| + O(1))}{(2|k| + O(1))\,\sqrt{|k| + O\left(\frac{1}{|k|}\right)}}\,\left(1 + O\left(\frac{1}{|k|}\right)\right) + \frac{1 + O\left(\frac{1}{|k|}\right)}{\sqrt{|k| + O\left(\frac{1}{|k|}\right)}} \right)
	\\&\qquad
	+ \frac{2\,\left(1 + O\left(\frac{1}{|k|}\right)\right)\,(|k| + O(1))}{(2|k| + O(1))\,\left(|k| + O\left(\frac{1}{|k|}\right)\right)}
	\\
	&= \frac{1 + o(1)}{2\sqrt{|k|}}\left(\frac{1 + o(1)}{\sqrt{|k|}} + \frac{1 + o(1)}{\sqrt{|k|}}\right) + \frac{1 + o(1)}{|k|}
	= \frac{2}{|k|}\,(1 + o(1)).
\end{align*}



This concludes the proof of Proposition \ref{P:Sp-resolv}.
\epf


\section{Weyl Spectrum} \label{S:Weyl}

We study two sources of the Weyl spectrum, namely radiation orthogonal to and along the interface.
In Sec.~\ref{S:weyl-x} we construct  Weyl sequences with support moving to infinity in the $x_1$-direction. These correspond physically to radiation in Maxwell's equations in the $x_1$-direction, i.e., orthogonal to the interface. In Sec.~\ref{S:Weyl-interf} we construct  Weyl sequences the support of which is localized near the interface $x_1=0$ and moves to infinity in the $x_\parallel$-variables. Recall that $x_\parallel = x_2$ if $N=2$ and $x_\parallel=(x_2,x_3)^\top$ if $N=3$. Physically, these sequences describe radiation along the interface (guided modes).

\subsection{Radiation orthogonal to the interface}\label{S:weyl-x}

We first consider the case of homogeneous media in $\R^N_\pm$ to compare the results with those in \cite{BDPW24} and because the analysis is considerably simpler than in the case of a general $x_1$ dependence of $\perm_\pm$. Later, in Sec. \ref{S:weyl-x1-gen}, we study the case of general $\perm_\pm(\cdot,\omega)$. Finally, in Sec. \ref{S:Wey-x1-per} we apply the general results to the case when $\perm_\pm(\cdot,\omega)$ are periodic or asymptotically periodic.

\subsubsection{Media Homogeneous in $\R^N_\pm$}

Let us consider the case $\perm_\pm(\cdot,\omega) =\perm_\pm(\omega)$ separately. Just as in \cite{BDPW24} for  the one and two dimensional cases, we show that for any $\omega\in \cM_\pm^\mathrm{red}$, defined in \eqref{E:Mpm-red}, a Weyl sequence traveling to $x_1\to\pm\infty$ can be constructed.
\blem\label{L:Weyl-hom-x1}
Assume $\perm_\pm(\cdot,\omega) =\perm_\pm(\omega)$. Then
$\cM_+^\mathrm{red}\cup \cM_-^\mathrm{red}\subset \sigma_{\mathrm{Weyl}}(\cL)$.
\elem
\bpf We study $\omega\in \cM_+^\mathrm{red}$. The case $\omega \in \cM_-^\mathrm{red}$ can be treated in a completely analogous manner.

Recall that a Weyl sequence at $\omega \in D(\perm)$ is a sequence $(E^{(n)})_n\subset \cD_\omega$ such that $\|E^{(n)}\|_{L^2(\R^N)^3}=1 \ \forall n$, $E^{(n)} \rightharpoonup 0$ in $L^2(\R^N)$ and  $\|L(\omega)E^{(n)}\|_{L^2(\R^N)^3}\to 0$ as $n\to\infty$. It is easy to see that one can equivalently check these properties in the Fourier variables. Let $u^{(n)}:=\widehat{E}^{(n)}$, where $~\widehat{\ }~$ is the Fourier transform with respect to the variables $x_\parallel$. Due to the Plancherel identity one has $\|u^{(n)}\|_{L^2(\R^N)^3}=\|E^{(n)}\|_{L^2(\R^N)^3}$ and $\|L(\omega)E^{(n)}\|_{L^2(\R^N)^3}=\|\widehat{L}(\omega)u^{(n)}\|_{L^2(\R^N)^3}$. Moreover, using the Parseval identity, we have $(E^{(n)},\psi)_{L^2(\R^N)^3}=(u^{(n)},\widehat{\psi})_{L^2(\R^N)^3}$ for all $\psi \in L^2(\R^N)^3$. Because $\ \widehat{\ }:L^2(\R^N)\to L^2(\R^N)$ is an isometric isomorphism, we conclude that $E^{(n)} \rightharpoonup 0$ if and only if $u^{(n)} \rightharpoonup 0$.

With the help of Lemma \ref{lem:equiv} we can also work in the transformed variables $\xi:=(u_1,v,w)^\top$ instead of $u$ and with the transformed operator  $\widetilde{L}$,  defined in \eqref{E:Lktil}, instead of $\widehat{L}$.
Applying the above argument again, we can also study the inverse Fourier transform of $\xi$, denoted by $\widecheck{\xi}$. Hence, below we find a sequence $(\xi^{(n)})\subset \widetilde{\cD}_\omega$ with $\xi^{(n)} =(u_1^{(n)},  v^{(n)}, w^{(n)})^\top $ such that $\|\xi^{(n)}\|_{L^2(\R^N)^3}=1$, $\widecheck{\xi}^{(n)} \rightharpoonup 0$ in $L^2(\R^N)^3$ and  $\|(\widetilde{L}(\omega)\xi^{(n)})\widecheck{\phantom{a}}\|_{L^2(\R^N)^3}\to 0$ as $n\to\infty$.

Again, we use the notation $k=k_2$ if $N=2$ and $k=(k_2,k_3)^\top$ if $N=3$. Let $k_0=k_{0,2}$ if $N=2$ and $k_0=(k_{0,2}, k_{0,3})^\top$ if $N=3$ be such that $|k_0|^2<\omega^2\perm_+(\omega)$. Recall that $\omega^2\perm_+(\omega)>0$ due to $\omega \in \cM_+^\mathrm{red}$.	We set $\xi^{(n)} =(u_1^{(n)},  v^{(n)}, w^{(n)})^\top $ with
\beq\label{E:vnwn}
w^{(n)}(x_1,k) := n^{\frac{N}{2}-1} \widehat{\varphi}\left(\frac{x_1-n^2}{n},n(k-k_{0})\right)e^{\ri \mu x_1}\quad\hbox{and}\quad \ u_1^{(n)}=v^{(n)}=0,
\eeq
where
$$\mu =\sqrt{\omega^2\perm_+(\omega)-|k_0|^2} \in \R_+$$
and $\widehat{\varphi}$ is the Fourier transform (with respect to $x_\parallel$) of a scalar-valued function $\varphi\in C_c^\infty(\R^N)$ with $\|\varphi\|_{L^2(\R^N)}=1$.
Note that $e^{\ri \mu x_1}$ solves the homogeneous $w-$equation \eqref{4} (with $k=k_0$) on $\R_+$ and that $w^{(n)}$ is the Fourier-transform in the $x_\parallel$-variables of the truncated plane wave
\beq\label{E:cvncwn}
\widecheck{w}^{(n)}(x) = n^{-\frac N2} e^{\ri\left(k_0 \cdot x_\parallel +\mu x_1\right)} \varphi\left(\frac{x_1-n^2}{n},\frac{x_\parallel}{n}\right).
\eeq

Since $\varphi$ is  compactly supported, for sufficiently large $n$ the function $\xi^{(n)}$ is supported away from the interface $x_1=0$. Hence, it trivially satisfies the interface conditions in $\widetilde{\cD}_\omega$. Due to the normalization of $\varphi$ we have $\|\xi^{(n)}\|_{L^2(\R^N)^3}=1$.
The divergence condition $(u_1^{(n)})'=-\ri|k_0|v^{(n)}$ holds also trivially as $u_1^{(n)}=v^{(n)}=0$.
Finally, the $L^2$-conditions in $\widetilde{\cD}_\omega$ are satisfied due to $\varphi \in C^\infty_c(\R^N)$. We conclude that $\xi^{(n)}\in \widetilde{\cD}_\omega$ for all $n$.

It remains to check that $(\widetilde{L}(\omega)\xi^{(n)})\widecheck{\phantom{a}}\to 0$ and that $\widecheck{\xi}^{(n)}\rightharpoonup 0$. As
$$
\begin{aligned}
	(\widetilde{L}(\omega)\xi^{(n)})\widecheck{\phantom{a}}  =& ~(-\Delta \widecheck{w}^{(n)} -\omega^2\perm_+(\omega)\widecheck{w}^{(n)})\bspm 0\\0\\1 \espm \\
	=&  ~-e^{\ri\left(k_0 \cdot x_\parallel +\mu x_1\right)}\left[n^{-\frac N2-2}\Delta \varphi\left(\frac{x_1-n^2}{n},\frac{x_\parallel}{n}\right)+2\ri n^{-\frac N2-1}\bspm \mu \\k_0\espm \cdot \nabla\varphi\left(\frac{x_1-n^2}{n},\frac{x_\parallel}{n}\right) \right] \bspm 0\\0\\1 \espm,
\end{aligned}
$$
we have
$$\|\widetilde{L}(\omega)\xi^{(n)}\|_{L^2(\R^N)^3}\leq c\left(n^{-2}\|\Delta \varphi\|_{L^2(\R^N)} + n^{-1}\|\nabla \varphi\|_{L^2(\R^N)}\right)\to 0 \ (n\to\infty).$$

Finally, to show $\widecheck{\xi}^{(n)}\rightharpoonup 0$, it suffices to check that
$\int_{\R^N} \eta(x)\widecheck{w}^{(n)}(x)\dd x \to 0 \ \forall \eta \in L^2(\R^N,\R)$. For $n$ large enough we have
$$
\begin{aligned}
	\left|\int_{\R^N} \eta(x)\widecheck{w}^{(n)}(x)\dd x\right| &\leq n^{-\frac N2}\int_{\R^N}\left| \varphi\left(\frac{x_1-n^2}{n},\frac{x_\parallel}{n}\right)\eta(x_1,x_\parallel)\right|\dd x \\
	& =n^{-\frac N2}\int_{[n,\infty)\times \R^{N-1}}\left| \varphi\left(\frac{x_1-n^2}{n},\frac{x_\parallel}{n}\right)\eta(x_1,x_\parallel)\right|\dd x \\
	& \leq  \|\varphi\|_{L^2(\R^N)}\|\eta\|_{L^2((n,\infty)\times \R^{N-1})}\to 0.
\end{aligned}
$$
\epf
\begin{remark}
	In the proof of Lemma \ref{L:Weyl-hom-x1} we construct Weyl sequences only of the simple form $(0,0,w_n)^\top$. 
	Also Weyl sequences of the form $(u_{1,n},v_n,0)^\top$ are conceivable but in view of the fact that Lemma \ref{L:Weyl-hom-x1}, Lemma \ref{L:Weyl-interf-hom} and Lemma \ref{L:pt-spec-hom}  describe the whole spectrum, as explained below Theorem \ref{T:main-homog}, such sequences will generate no additional spectrum outside $\Omega_0$.
\end{remark}

\subsubsection{General $\perm$}\label{S:weyl-x1-gen}

With $\Omega^\pm_a$ for $a>0$ introduced in \eqref{Omegaa}, we define
$$
\begin{aligned}
	\Sigma^\pm_{e,v}&:=\left\{ \omega\in D(\perm): \text{ there is } \ k_0\in \R^{N-1} \ \text{such that}\ \omega \notin \Omega^\pm_{|k_0|^2} \text{ and there exists}\right.\\
	& \qquad \left. \text{a solution} \ v\in C^1_b(\R_\pm) \ \text{of} \ \eqref{3} \ \text{on} \ \R_\pm \ \text{with} \ k=k_0 \ \text{and} \ r=0 \ \text{such that} \ \eqref{E:seq-supp} \ \text{holds}\right\}
\end{aligned}
$$
and $$
\begin{aligned}
	\Sigma^\pm_{e,w}&:=\left\{ \omega\in D(\perm): \text{ there is } \ k_0\in \R^{N-1} \ \text{such that} \   \text{there exists}\ \text{a solution} \ w\in C^1_b(\R_\pm) \ \text{of} \ \eqref{4} \ \text{on} \ \R_\pm\right.\\
	& \qquad \left.  \ \text{with} \ k=k_0 \ \text{and} \ r=0 \ \text{such that} \ \eqref{E:seq-supp} \ \text{holds with} \ w \ \text{instead of} \ v\right\},
\end{aligned}
$$
where condition \eqref{E:seq-supp} reads
\beq\label{E:seq-supp}
\begin{aligned}
	&\exists~\eps,c>0, (R_n)_{n\in \N}\subset \R_\pm, (l_n)_{n\in \N}\subset \R_+: |R_n|, l_n \to\infty, \ l_n\leq \frac{|R_n|}{2}  \   \text{and} \\ &\frac{1}{2l_n}\text{meas}\{x\in [R_n-l_n,R_n+l_n]: |v(x)|\geq \eps\}\geq c \ \forall n \in \N.
\end{aligned}
\eeq
\brem
 The condition $\omega \notin \Omega^\pm_{|k_0|^2}$, i.e., $\dist(\omega^2 \perm_\pm(\R_\pm,\omega), {|k_0|^2})>0$, is not needed in $\Sigma_{e,w}^\pm$ because equation \eqref{4} does not contain $|k_0|^2-\omega^2\perm (x_1,\omega)$ as a denominator.
\erem
\begin{remark}
	Condition \eqref{E:seq-supp} means that one can find a sequence of intervals moving to infinity with their lengths also diverging to infinity and such that the solution $v$ is bounded away from zero on a substantial part of these intervals.
	
	Note that  \eqref{E:seq-supp} (in the ``$+$ case'') is equivalent to the existence of $\eps>0$ such that
	$$\limsup_{\stackrel{R,l\to\infty}{l\leq R/2}}\frac{1}{2l}{\rm meas}\{x\in [R-l,R+l]:|v(x)|\geq \eps\}>0.$$
\end{remark}
\brem\label{rem:seq-supp}
Condition \eqref{E:seq-supp} is generally difficult to check. However, as we show in Section \ref{S:Wey-x1-per}, see Lemma \ref{lem:Weyl-x1-per}, it can be easily satisfied in the case of periodic media.
\erem

\bthm\label{T:Weyl-sp-x1}
$$\Sigma_{e,v}^+\cup \Sigma_{e,v}^-\cup \Sigma_{e,w}^+\cup \Sigma_{e,w}^- \subset \sigma_{\mathrm{Weyl}}(\cL)$$
\ethm
\bpf
We prove $\Sigma_{e,v}^+ \subset \sigma_{\mathrm{Weyl}}(\cL)$ in detail and comment on the simpler case $\Sigma_{e,w}^+ \subset \sigma_{\mathrm{Weyl}}(\cL)$ at the end. The remaining two statements are proved analogously.

Let $\omega \in \Sigma_{e,v}^+$ and $k_0$ and $v$ be like in the definition of $\Sigma_{e,v}^+$. To construct a Weyl sequence, we first define
\beq\label{E:vn}
v_n(x_1,k_2,k_3):=\frac{n^{\frac{N-1}{2}}}{\sqrt{l_n}}\zeta(n(k-k_{0}))\varphi\left(\frac{x_1-R_n}{l_n}\right)v(x_1),
\eeq
where $\zeta \in C^\infty_c(\R^{N-1},\R)$ is arbitrary and $\varphi \in C^\infty_c(\R,\R)$ satisfies
\beq\label{E:supp-phi}
\supp~\varphi \subset [-1,\infty),  \quad \varphi >0 \quad \text{on}\ (-1,1].
\eeq
Next, we calculate
\beq\label{E:vn-norm}
\begin{aligned}
	\|v_n\|_{L^2(\R^N)}^2 &= \frac{n^{N-1}}{l_n}\int_{\R^N} \zeta(n(k-k_{0}))^2 \varphi\left(\frac{x_1-R_n}{l_n}\right)^2|v(x_1)|^2\dd (x_1,k) \\
	&= \|\zeta\|_{L^2(\R^{N-1})}^2 \int_{\R} \varphi(y_1)^2|v(R_n+l_ny_1)|^2 \dd y_1.
\end{aligned}
\eeq
Since $v$ is bounded, \eqref{E:vn-norm} shows that $\|v_n\|_{L^2(\R^N)}$ is bounded from above independently of $n$. Furthermore, denoting
$$M_n:=\{y_1\in [-1,1]: |v(R_n+l_ny_1)|\geq \eps\}$$
and using \eqref{E:seq-supp}, we get
$$\text{meas}~M_n = \frac{1}{l_n}\text{meas}~\{x_1\in [R_n-l_n,R_n+l_n]: |v(x_1)|\geq \eps\} \geq \widetilde{c},$$
where $\widetilde{c}:=\min\{2c,3\}$ is introduced for a technical reason explained next. Because $\text{meas}~(M_n\cap [-1,-1+\tfrac{\widetilde{c}}{2}])\leq \tfrac{\widetilde{c}}{2}$, we therefore get
$$
\text{meas}~\left(M_n\cap \left[-1+\frac{\widetilde{c}}{2},1\right]\right) \geq \frac{\widetilde{c}}{2}.
$$
Note that the interval $\left[-1+\frac{\widetilde{c}}{2},1\right]$ is non-empty because $\widetilde{c}\leq 3.$
Using this estimate, we obtain
$$\int_\R\varphi(y_1)^2|v(R_n+l_ny_1)|^2 \dd y_1 \geq \eps^2 \int_{M_n} \varphi(y_1)^2\dd y_1 \geq \eps^2 \int_{M_n\cap [-1+\frac{\widetilde{c}}{2},1]} \varphi(y_1)^2\dd y_1 \geq \eps^2\frac{\widetilde{c}}{2} \min_{[-1+\frac{\widetilde{c}}{2},1]}\varphi^2 >0.$$
Hence, \eqref{E:vn-norm} implies that $\|v_n\|_{L^2(\R^N)}$ is bounded also from below independently of $n$.

For the rest of the proof, once again, we use the notation
$$W:=W(x_1):=\omega^2 \perm(x_1,\omega).$$
Motivated by \eqref{5} (with $r=0$), we define the corresponding component $u_1$ by
$$u_{1,n}(x_1,k):=\frac{-\ri |k_0|v_n'(x_1,k)}{|k_0|^2-W(x_1)}.$$
Since the resulting vectors $(u_{1,n},v_n,0)$ in general do not satisfy the divergence condition, we modify the second component to obtain our proposed Weyl sequence
\beq\label{E:Un-rn}
U_n:=\begin{pmatrix}
	u_{1,n}\\ v_n+r_n\\ 0
\end{pmatrix}, \quad \text{where} \quad  r_n(x_1,k):=\frac{1}{W(x_1)}\frac{|k_0|}{|k|} \left(\frac{W(x_1) v_n'(x_1,k)}{|k_0|^2-W(x_1)}\right)'-v_n(x_1,k).
\eeq

To show that $U_n$ is a Weyl sequence, we start by proving $U_n \in \widetilde{\cD}_\omega$ for all $n\in \N$ and $\widetilde{L}U_n\to 0$ in $L^2(\R^N)^3$, where $\widetilde{L}$ was defined in \eqref{E:Lktil}. First note that besides $v_n\in L^2(\R^N)$ (shown above), one also has $v_n'\in L^2(\R^N)$ and by \eqref{3} also $v_n''\in L^2(\R^N)$. It follows that $u_{1,n}\in L^2(\R^N)$. Next, we show that $r_n, r_n',r_n''\to 0$ in $L^2$. Abbreviating
$$\rho_{n,k}:=\frac{n^\frac{N-1}{2}}{\sqrt{l_n}}\zeta(n(k-k_{0})),$$
we obtain
$$
\begin{aligned}
	r_n&=\frac{\rho_{n,k}}{W} \frac{|k_0|}{|k|}\left[\frac{W}{|k_0|^2-W}\left(\varphi\left(\frac{\cdot-R_n}{l_n}\right)v'+\frac{1}{l_n}\varphi'\left(\frac{\cdot-R_n}{l_n}\right)v\right)\right]'-v_n\\
	&= \left(\frac{|k_0|}{|k|}-1\right)v_n +\frac{\rho_{n,k}}{W}\frac{|k_0|}{|k|}\left\{\frac{2}{l_n}\varphi'\left(\frac{\cdot-R_n}{l_n}\right)\frac{W}{|k_0|^2-W}v'+\right.\\
	&\qquad \left.\frac{1}{l_n}\left[\left(\frac{W}{|k_0|^2-W}\right)'\varphi'\left(\frac{\cdot-R_n}{l_n}\right)+ \frac{1}{l_n}\frac{W}{|k_0|^2-W}\varphi''\left(\frac{\cdot-R_n}{l_n}\right)\right]v \right\}\\
	&=:\left(\frac{|k_0|}{|k|}-1\right)v_n +\frac{\rho_{n,k}}{W}\frac{|k_0|}{|k|} A_n(v,\varphi),
\end{aligned}
$$
where we have used the differential equation for $v$, i.e., $\left(\frac{W}{|k_0|^2-W}v'\right)'=W v$. Next, differentiating $r_n$ yields
\beq\label{E:rn-prime}
\begin{aligned}
	r_n'&=\left(\frac{|k_0|}{|k|}-1\right)v_n' - \frac{\rho_{n,k}W'}{W^2}\frac{|k_0|}{|k|}A_n(v,\varphi) \\
	&\quad + \frac{\rho_{n,k}}{W}\frac{|k_0|}{|k|l_n}\left\{2 \varphi'\left(\frac{\cdot-R_n}{l_n}\right)W v +\left[ \left(\frac{W}{|k_0|^2-W}\right)'\varphi'\left(\frac{\cdot-R_n}{l_n}\right)+\frac{3}{l_n}\varphi''\left(\frac{\cdot-R_n}{l_n}\right)\frac{W}{|k_0|^2-W}\right]v'\right.\\
	&\quad \left.  +\left[ \left(\frac{W}{|k_0|^2-W}\right)''\varphi'\left(\frac{\cdot-R_n}{l_n}\right)+ \frac{2}{l_n}\left(\frac{W}{|k_0|^2-W}\right)'\varphi''\left(\frac{\cdot-R_n}{l_n}\right)+\frac{1}{l_n^2}
	\frac{W}{|k_0|^2-W}\varphi'''\left(\frac{\cdot-R_n}{l_n}\right)\right]v\right\}.
\end{aligned}
\eeq
Differentiating once more, we get $r_n''$ in terms of $v_n'', v, v',$ and $v''$. The formula for $r_n''$ includes $W'''$. This is where the assumption $\perm_\pm\in W^{3,\infty}(\R_\pm)$  from \eqref{E:ass-perm} is fully used. Because $v\in C^1_b$ and with the use of the differential equation \eqref{3}, we get that $v,v',v''$ are bounded on $[0,\infty)$. Moreover,
\beq\label{E:rho-est1}
\begin{aligned}
	\int_{\R^{N-1}} \rho_{n,k}^2\left(\frac{|k_0|}{|k|}-1\right)^2\dd k =  \frac{1}{l_n}\int_{\R^{N-1}} \zeta(\kappa)^2\left(\frac{|k_0|}{|k_0+\tfrac{1}{n}\kappa|}-1\right)^2\dd \kappa \to 0 \ (n\to \infty)
\end{aligned}
\eeq
due to the compact support of $\zeta$. Similarly, one shows that
\beq\label{E:rho-est2}
\int_{\R^{N-1}} \rho_{n,k}^2 \frac{|k_0|^2}{|k|^2} \dd k\leq \frac{c}{l_n}
\eeq
for all $n$ large enough with $c$ independent of $n$. Thanks to \eqref{E:rho-est1} and \eqref{E:rho-est2} and using the formulas for $r_n, r_n'$ and $r_n''$, we have $r_n,r_n',r_n''\in L^2(\R^N)$ and
\beq\label{E:rn-conv}
r_n,r_n',r_n'' \to 0 \quad \text{in} \ L^2(\R^N).
\eeq
In particular, we have now shown that $U_n \in L^2(\R^N)^3$.
From \eqref{E:rn-prime} we can derive in a similar way that $|k|r_n'\to 0$ in $L^2$, and therefore
\beq\label{eq:Lrn}
\widetilde{L}\bspm
0\\r_n\\0
\espm\to 0 \quad \text{in} \ L^2(\R^N).
\eeq
Straightforward estimates using the compact support of $\varphi$ and $\zeta$ and the boundednness of $v$, $v'$, and $v''$ produce the remaining
$L^2$-properties in $\widetilde{\mathcal{D}}$, i.e.,
$$(v_n+r_n)'-\ri |k|u_{1,n}, \ (v_n+r_n)''-\ri |k|u_{1,n}', \ \ri |k| (v_n+r_n)'+|k|^2u_{1,n} \in L^2(\R^N).$$
The interface conditions in $\widetilde{\mathcal{D}}$, i.e.,
$$\llbracket \perm u_{1,n}\rrbracket = \llbracket v_n+r_n\rrbracket=\llbracket \ri |k|u_{1,n}-(v_n+r_n)'\rrbracket=0$$
are trivially satisfied because due to \eqref{E:supp-phi} we have $v_n(x_1)=0$ for $x_1\leq R_n/2$, and all terms in $u_{1,n}$ and $r_n$ are proportional either to $v_n$ or $v_n'$.

The last property needed to conclude $U_n\in \widetilde{\cD}_\omega$ is the divergence condition. This holds automatically since
$$(Wu_{1,n})'+\ri |k| W(v_n+r_n)=  (Wu_{1,n})' + \ri|k_0| \left(\frac{Wv_n'}{|k_0|^2-W}\right)'=0$$
using the definition of $u_{1,n}$ and $r_n$.

Next, we need to show that $\widetilde{L}\bspm u_{1,n}\\ v_n\\ 0\espm \to 0$ in $L^2$. We calculate
\begin{align}
	\widetilde{L} \bspm u_{1,n}\\ v_n\\ 0\espm & = \begin{pmatrix}
		(|k|^2-W)u_{1,n} +\ri |k| v_n'\\ \ri |k| u_{1,n}' - v_n'' - Wv_n \\0
	\end{pmatrix} = \begin{pmatrix}
		(|k_0|^2-W)u_{1,n} +\ri |k_0| v_n'\\ \ri |k_0| u_{1,n}' - v_n'' - Wv_n \\0
	\end{pmatrix}+\begin{pmatrix}
		(|k|^2-|k_0|^2)u_{1,n} + \ri (|k|-|k_0|)v_n'\\ \ri (|k|-|k_0|)u_{1,n}'\\0
	\end{pmatrix} \notag\\
	&=\begin{pmatrix}
		0\\ \left(\frac{W}{|k_0|^2-W}v_n'\right)'-Wv_n\\0
	\end{pmatrix}+\begin{pmatrix}
		(|k|^2-|k_0|^2)u_{1,n} + \ri (|k|-|k_0|)v_n'\\ \ri (|k|-|k_0|)u_{1,n}'\\0
	\end{pmatrix},\label{E:LUn-part}
\end{align}
where we have used the definition of $u_{1,n}$ and the identity $\left(\frac{|k_0|^2 v_n'}{|k_0|^2-W}\right)'-v_n''= \left(\frac{W}{|k_0|^2-W}v_n'\right)'$.

We denote the second component of the first term in \eqref{E:LUn-part} by $B_n$, i.e., $B_n:=\left(\frac{W}{|k_0|^2-W}v_n'\right)'-Wv_n$ and obtain
$$
\begin{aligned}
	B_n(x_1,k) &= \frac{n^\frac{N-1}{2}}{\sqrt{l_n}}\zeta(n(k-k_{0}))\left[\frac{W}{|k_0|^2-W}\left(\varphi\left(\frac{x_1-R_n}{l_n}\right)v\right)'' \right.\\
	&\left. \qquad +\left(\frac{W}{|k_0|^2-W}\right)'\left(\varphi\left(\frac{x_1-R_n}{l_n}\right)v\right)' - W \varphi\left(\frac{x_1-R_n}{l_n}\right) v\right]\\
	&= \frac{n^\frac{N-1}{2}}{\sqrt{l_n}}\zeta(n(k-k_{0}))\left\{\frac{1}{l_n^2}\frac{W}{|k_0|^2-W}\varphi''\left(\frac{x_1-R_n}{l_n}\right)v\right.\\
	&\left. \qquad +\frac{1}{l_n}\left[2\frac{W}{|k_0|^2-W} \varphi'\left(\frac{x_1-R_n}{l_n}\right)v'+\frac{|k_0|^2W'}{(|k_0|^2-W)^2}\varphi'\left(\frac{x_1-R_n}{l_n}\right)v\right] \right\}.
\end{aligned}
$$
As $v\in C^1_b([0,\infty))$ and $\dist(W(\R),\{|k_0|^2\})>0$, we get
$$|B_n(x_1,k)|\leq C\frac{n^\frac{N-1}{2}}{\sqrt{l_n}}|\zeta(n(k-k_{0}))|\left(\frac{1}{l_n^2}\left|\varphi''\left(\frac{x_1-R_n}{l_n}\right)\right|+\frac{1}{l_n}\left|\varphi'\left(\frac{x_1-R_n}{l_n}\right)\right|\right)$$
leading to
$$
\|B_n\|_{L^2(\R^N)}^2\leq 2C^2\|\zeta\|_{L^2(\R^{N-1})}^2 \left(\frac{1}{l_n^4}\|\varphi''\|_{L^2(\R)}^2+\frac{1}{l_n^2}\|\varphi'\|_{L^2(\R)}^2\right)
\leq \frac{c}{l_n^2}\to 0 \quad (n\to \infty).$$
In order to estimate the second vector in \eqref{E:LUn-part}, i.e.,
$$C_n(x_1,k):=\bspm
(|k|^2-|k_0|^2)u_{1,n} + \ri (|k|-|k_0|)v_n'\\ \ri (|k|-|k_0|)u_{1,n}'\\0
\espm, $$
first note that $||k|-|k_0||\leq |k-k_0|$ and $||k|^2-|k_0|^2|\leq |k-k_0|(|k|+|k_0|)$.  Hence,
$$
\begin{aligned}
	|C_n(x_1,k)|^2&=\left|\begin{pmatrix}-\ri (|k|^2-|k_0|^2)\frac{|k_0|v_n'}{|k_0|^2-W} + \ri (|k|-|k_0|)v_n'\\  (|k|-|k_0|) \left(\frac{|k_0|v_n'}{|k_0|^2-W}\right)'\\ 0\end{pmatrix}\right|^2\leq C|k-k_0|^2\left[(|k|+|k_0|)^2|v_n'|^2+|v_n'|^2+|v_n''|^2\right]\\
	&\leq C|k-k_0|^2\frac{n^{N-1}}{l_n}\zeta(n(k-k_{0}))^2(|k|^2+1)
	\left[\varphi\left(\frac{x_1-R_n}{l_n}\right)^2 + \varphi'\left(\frac{x_1-R_n}{l_n}\right)^2 + \varphi''\left(\frac{x_1-R_n}{l_n}\right)^2\right],
\end{aligned}
$$
where we have used the boundedness of $v, v',$ and $v''$ on $\R$. The fact that the second derivative $v''$ is bounded follows from the boundedness of $v$ and $v'$ and from the differential equation \eqref{3} (with $r=0$).

For the $L^2-$norm we obtain
$$
\begin{aligned}
	\|C_n\|_{L^2(\R^N)}^2 &\leq C \int_{\R^{N-1}} \zeta(\kappa)^2\frac{|\kappa|^2}{n^2}\left[
	\left|k_{0}+\frac{\kappa}{n}\right|^2+1\right]\dd \kappa\left(\|\varphi\|_{L^2(\R)}^{2} +\|\varphi'\|_{L^2(\R)}^2+\|\varphi''\|_{L^2(\R)}^2\right)\\
	&\leq Cn^{-2} \to 0 \quad (n\to \infty),
\end{aligned}
$$
leading to the conclusion $\widetilde{L} \bspm u_{1,n}\\ v_n\\ 0\espm \to 0$ in $L^2(\R^N)^3$ and together with \eqref{eq:Lrn} we get $\widetilde{L} U_n \to 0$ in $L^2$.

It remains to be proved that $U_n \rightharpoonup 0$ in $L^2(\R^N)^3$. As shown in \eqref{E:vn-norm}, $\|v_n\|_{L^2}$ is bounded above independently of $n$. Differentiating $v_n$ and using that $v'$ is bounded by assumption, we find similarly to \eqref{E:vn-norm} that also $\|v_n'\|_{L^2}$ is bounded above independently of $n$. By the definition of $u_{1,n}$ and the convergence in \eqref{E:rn-conv}, we find that
\beq\label{E:Un-bded}
\|U_n\|_{L^2(\R^N)^3}\leq C \quad \forall n \in \N
\eeq
with $C>0$ independent of $n$. For any $\eta\in L^2(\R^N)^3$ we find
$$
\begin{aligned}
\left|\int_{\R^N}U_n\cdot\eta \dd x\right|&=\left|\int_{R_n-l_n}^\infty \int_{\R^{N-1}} U_n\cdot \eta \dd k \dd x_1\right|\\
&\leq \|U_n\|_{L^2((R_n-l_n,\infty)\times \R^{N-1})^3}\|\eta\|_{L^2((R_n-l_n,\infty)\times \R^{N-1})^3} \to 0  \quad (n\to \infty),
\end{aligned}
$$
where the first equality follows from  $v_n(x_1,k)=0$ for $x_1\leq R_n-l_n$, which is ensured by \eqref{E:supp-phi}. The convergence step is concluded using \eqref{E:Un-bded} and the fact $R_n - l_n \to \infty$.

This finishes the proof of $\Sigma_{e,v}^+ \subset \sigma_{\mathrm{Weyl}}(L)$.

For $\omega\in \Sigma_{e,w}^+$ we choose $k_0$ and $w$ like in the definition of $\Sigma_{e,w}^+$. After defining $w_n$ by \eqref{E:vn} with $v$ replaced by $w$ and the Weyl sequence $U_n:=(0,0,w_n)^\top$, the proof is a simplified version of the one above. Note, in particular, that the divergence condition is trivially satisfied by $U_n$.
\epf

\subsubsection{Periodic and Asymptotically Periodic Media}\label{S:Wey-x1-per}

Let us first consider the case of $\perm_+(\cdot,\omega)$ or $\perm_-(\cdot,\omega)$ being periodic. As we show, a sufficient condition for \eqref{E:seq-supp} with respect to equation \eqref{3} or \eqref{4} is that the corresponding system is in the case of stability or conditional stability. See Sec. \ref{S:per-estimates} for the stability concept.

If, in the case of periodic $\perm_+(\cdot,\omega)$, we have for the corresponding discriminants $D^{(v)}_+\in [-2,2]$ or $D^{(w)}_+\in [-2,2]$, then there are bounded solutions of \eqref{3} or \eqref{4}, resp., on $\R_+$. If, on the other hand $\perm_-(\cdot,\omega)$ is periodic and $D^{(v)}_-\in [-2,2]$ or $D^{(w)}_-\in [-2,2]$, then there are bounded solutions on $\R_-$. 

\blem\label{lem:Weyl-x1-per}
Let $\omega \in D(\perm)$ and let $\perm_+(\cdot,\omega)$ be $a$-periodic (with $a>0$).  Assume that for some $k_0\in \R^{N-1}$ one of the following two options holds.
\begin{itemize}
	\item[a)] $\omega \notin \Omega^+_{|k_0|^2}$ and equation \eqref{3}, with $\perm$ replaced by $\perm_+$ and $k$ by $k_0$, is in the
	case of stability or conditional stability, i.e., $D_+^{(v)}(k_0)\in [-2,2]$, 
	\item[b)]  \eqref{4}, with $\perm$ replaced by $\perm_+$ and $k$ by $k_0$, is in the
	case of stability or conditional stability, i.e., $D_+^{(w)}(k_0)\in [-2,2]$.
\end{itemize}
Then $\omega \in \sweyl(\cL)$.

An analogous statement holds if $\perm_-(\cdot,\omega)$ is periodic using the conditions $D_-^{(v)}(k_0)\in [-2,2]$ and $D_-^{(w)}(k_0)\in [-2,2]$, respectively.
\elem
\brem
Note that in Lemma \ref{lem:Weyl-x1-per} the periodicity of $\perm(\cdot,\omega)$ needs to be satisfied on $\R_+$ or on $\R_-$, not necessarily on both half lines. In fact, periodicity on $(x_0,\infty)$ or on $(-\infty,-x_0)$ with $|x_0|$ large enough is sufficient as one easily sees in the construction of the Weyl sequence.
\erem

\bpf
We assume $a$-periodicity of $\perm_+$ and first show that in the case a) we get $\omega \in \Sigma_{e,v}^+$. The statement then follows by Theorem \ref{T:Weyl-sp-x1}. First note that due to $D_+^{(v)}(k_0)\in [-2,2]$ a solution $v$ of the homogeneous version of  \eqref{3} with the form
$v(x_1)=e^{\ri mx_1}p(x_1)$ exists, where $m\in \R$ and $p$ is $a-$periodic (cf. equation (\ref{eq:6})). The solution $v$ satisfies $v\in C_b^1(\overline{\R_+})$.
Condition  \eqref{E:seq-supp} is satisfied by choosing $R_n:=2na, l_n:=na$ and $\eps:=\tfrac{1}{2}\|p\|_\infty$. Indeed, we obtain
$$
\begin{aligned}
	\frac{1}{2l_n}\text{meas}\{x\in[R_n-l_n,R_n+l_n]:|v(x)|\geq \eps\} &= \frac{1}{2na}\text{meas}\{x\in[0,2na]:|p(x)|\geq \tfrac{1}{2}\|p\|_\infty\} \\
	&= \frac{1}{a}\text{meas}\{x\in[0,a]:|p(x)|\geq \tfrac{1}{2}\|p\|_\infty\},
\end{aligned}
$$
which is positive since $p$ is continuous.

The same argument applies in the case b) and leads to $\omega \in \Sigma_{e,w}^+$. The condition $\omega\notin \Omega_{|k_0|^2}^+$ is not needed here because it does not appear in the definition of $ \Sigma_{e,w}^+$.

Analogous arguments apply when $\perm_-(\cdot,\omega)$ is periodic, leading to $\omega \in \Sigma_{e,v}^-$ or $\omega \in \Sigma_{e,w}^-$, respectively.
\epf

For the rest of Section \ref{S:weyl-x}  we consider  $\perm$  given by an $L^1$-perturbation of a periodic function. This can be understood as the case of an asymptotically periodic permittivity.

\begin{lem}\label{L:Weyl-x1-as-per}
	Let $\omega \in D(\perm)$.  Suppose that $\perm_+(\cdot,\omega)=\perm_{p,+}(\cdot,\omega)+\perm_{as,+}(\cdot,\omega)$, where $\perm_{p,+}(\cdot,\omega)$ is $a$-periodic and $\perm_{as,+}(\cdot,\omega)\in L^1(\R_+)$.
	\begin{enumerate}
		\item[a)] Let $k_0\in\R^{N-1}$ be such that equation \eqref{4}, with $\perm$ replaced by $\perm_{p,+}$ and $k$ by $k_0$, is in the
		case of stability, see Section \ref{S:per-estimates}. Then $\omega\in\sigma_{\mathrm{Weyl}}(\cL)$.
		\item[b)] Let $k_0\in\R^{N-1}$ be such that equation \eqref{3}, with $\perm$ replaced by $\perm_{p,+}$ and $k$ by $k_0$, is in the case of stability. Moreover, suppose that $\left(\perm_{as,+}\right)'\in L^1(\R_+)$ and $\dist(\omega^2 \perm_{p,+}(\R_+,\omega),\{0,|k_0|^2\})>0$.
		Then $\omega\in\sigma_{{\mathrm Weyl}}(\cL)$.
	\end{enumerate}
	Analogous statements hold if the corresponding assumptions are satisfied by $\perm_-$.
\end{lem}
\brem
We note here that the medium does not need to be asymptotically periodic on both half lines $\R_\pm$. Lemma \ref{L:Weyl-x1-as-per} assumes this structure on either side of the interface. Moreover, for clarity, we remark that the conditions on on $\perm_{p,+}$ in b) are equivalent to $\omega \notin \Omega^+_{|k_0|^2} \cup \Omega^+_0$ with the $\Omega$-set defined based on the permittivity $\perm_{p,+}$.
\erem
\bpf
We restrict ourselves to the ``+'' case. For ease of notation we drop the subscript + and denote $x_1$ simply by $x$ in this proof.

a) Due to the fact that equation \eqref{4} is in the case of stability, there exists a fundamental system of the form $\psi_1(x)=e^{\ri mx}p_1(x)$, $\psi_2(x)=e^{-\ri mx}p_2(x)$ with $m\in\R$ and periodic $p_1,p_2$ (cf. equation (\ref{eq:6})).

Let $w$ denote the solution of the initial value problem
\beq\label{eq:ap1}
-w''+(|k_0|^2-\omega^2\perm(\cdot,\omega))w = 0, \ w(x_0)=\psi_1(x_0), \ w'(x_0)=\psi'_1(x_0),
\eeq
where $x_0>0$ will be specified later. Then $u=w-\psi_1$ satisfies
\beqq
-u''+(|k_0|^2-\omega^2\perm_p(\cdot,\omega))u = \omega^2 \perm_{as}(\cdot,\omega)w, \ u(x_0)=u'(x_0)=0,
\eeqq
whence the variation of constants formula gives
\beq\label{eq:ap2}
u(x)=\int_{x_0}^x \frac{\psi_1(x)\psi_2(t)-\psi_2(x)\psi_1(t)}{\psi_1(t)\psi'_2(t)-\psi_2(t)\psi'_1(t)}\omega^2\perm_{as}(t,\omega)w(t)\ dt
\eeq
and
\beq\label{eq:ap3}
u'(x)=\int_{x_0}^x \frac{\psi'_1(x)\psi_2(t)-\psi'_2(x)\psi_1(t)}{\psi_1(t)\psi'_2(t)-\psi_2(t)\psi'_1(t)}\omega^2\perm_{as}(t,\omega)w(t)\ dt.
\eeq
Here, the Wronski determinant $\psi_1\psi'_2-\psi_2\psi'_1=p_1p_2'-p_2p_1'-2\ri mp_1p_2$ is periodic and hence, since it is non-zero, bounded away from $0$ on $[0,\infty)$. Furthermore, the numerators in \eqref{eq:ap2} and \eqref{eq:ap3},
\beqq
e^{\ri m(x-t)} p_1(x)p_2(t)- e^{\ri m(t-x)} p_1(t)p_2(x)
\eeqq
and
\beqq
e^{\ri m(x-t)} (p_1'(x)+\ri mp_1(x))p_2(t)- e^{\ri m(t-x)} p_1(t)(p'_2(x)-\ri mp_2(x)),
\eeqq
respectively, are bounded on $[0,\infty)\times[0,\infty)$. Consequently, for some ($\omega$-dependent) $C>0$, by \eqref{eq:ap2} and \eqref{eq:ap3},
\beq\label{eq:ap4}
|u(x)|,|u'(x)|\leq C \int_{x_0}^x |\perm_{as}(t,\omega)|\ dt \ \|w\|_{L^\infty(x_0,x)} \hbox{ for } x\in[x_0,\infty).
\eeq
In particular, since $|w(x)|\leq|u(x)|+|\psi_1(x)|=|u(x)|+|p_1(x)|$, \eqref{eq:ap4} implies
\beqq
\|w\|_{L^\infty(x_0,x)} \leq C\|\perm_{as}(\cdot,\omega)\|_{L^1(x_0,x)}\|w\|_{L^\infty(x_0,x)} +  \|p_1\|_{L^\infty(0,a)}
\eeqq
and hence $w$ is bounded if
\beq\label{eq:ap5}
\|\perm_{as}(\cdot,\omega)\|_{L^1(x_0,\infty)}<\frac{1}{C}.
\eeq
As $\perm_{as}(\cdot,\omega)\in L^1(\R_+)$, \eqref{eq:ap5} will be satisfied if $x_0$ is chosen sufficiently large.
Now, \eqref{eq:ap4} together with $|w'(x)|\leq|u'(x)|+|p'_1(x)+\ri mp_1(x)|$ implies boundedness also of $w'$. Thus, $w$ is a solution of \eqref{4} with $k=k_0$ in $C^1_b(\overline{\R}_+)$.

Moreover, again choosing $x_0$ sufficiently large, we can arrange that
\beq\label{eq:ap6}
\|\perm_{as}(\cdot,\omega)\|_{L^1(x_0,\infty)} \leq \frac{\|p_1\|_{L^\infty(0,a)}}{4C\|w\|_{L^\infty(\R_+)}}.
\eeq
Then \eqref{eq:ap4} together with  $|w(x)|\geq |\psi_1(x)|-|u(x)|= |p_1(x)|-|u(x)|$
gives
\beqq
|w(x)|\geq|p_1(x)|-\frac14 \|p_1\|_{L^\infty(0,a)}, \quad x\in[x_0,\infty).
\eeqq
With $R_n,l_n$ chosen as in the proof of Lemma \ref{lem:Weyl-x1-per}, we have
\beqq
\left\{ x\in[R_n-l_n,R_n+l_n] : |w(x)|\geq \frac14 \|p_1\|_{L^\infty(0,a)} \right\} \supseteq \left\{ x\in[R_n-l_n,R_n+l_n] : |p_1(x)|\geq \frac12 \|p_1\|_{L^\infty(0,a)} \right\}
\eeqq
provided that $R_n-l_n\geq x_0$, i.e., $n\geq\frac{x_0}{a}$. Hence, using the calculation in the proof of Lemma \ref{lem:Weyl-x1-per},
$$
\begin{aligned}
	\frac{1}{2l_n}\mathrm{meas\ }&\left\{ x\in[R_n-l_n,R_n+l_n] : |w(x)|\geq \frac14 \|p_1\|_{\infty} \right\}\\
	&\geq  \frac{1}{2l_n}\mathrm{meas\ }\left\{ x\in[R_n-l_n,R_n+l_n] : |p_1(x)|\geq \frac12 \|p_1\|_{\infty} \right\} \\
	&=\frac{1}{a}\mathrm{meas\ }\left\{ x\in[0,a] : |p_1(x)|\geq \frac12 \|p_1\|_{\infty} \right\} \ >\ 0
\end{aligned}
$$
shows that condition \eqref{E:seq-supp} is satisfied. Theorem \ref{T:Weyl-sp-x1} implies $\omega\in\Sigma^+_{e,w}\subseteq\sigma_{\mathrm{Weyl}}(\cL)$.

b) We use the fundamental system of \eqref{3} with $\perm$ replaced by $\perm_p$ of the form $\psi_1(x)=e^{\ri mx}p_1(x)$, $\psi_2(x)=e^{-\ri mx}p_2(x)$ with $m\in\R$ and periodic $p_1,p_2$. Let $v$ denote the solution of the initial value problem
\beq\label{eq:ap7}
-v''-\frac{|k_0|^2\perm'(\cdot,\omega)}{\perm(\cdot,\omega)\left(|k_0|^2-\omega^2 \perm(\cdot,\omega)\right)}v' +\left(|k_0|^2-\omega^2 \perm(\cdot,\omega)\right) v = 0, \ v(x_0)=\psi_1(x_0), \ v'(x_0)=\psi'_1(x_0),
\eeq
where $x_0\in[0,\infty)$ will again be chosen later. Note that the differential equation in \eqref{eq:ap7} is \eqref{3} with $k=k_0$ and $r\equiv 0$. In what follows, we suppress the $\omega$-dependence of $\perm, \perm_p$ and $\perm_{as}$. Now,  $u=v-\psi_1$ satisfies
\beq\label{eq:ap8}
\begin{aligned}
	-u''-\frac{|k_0|^2\perm'_p}{\perm_p\left(|k_0|^2-\omega^2 \perm_p\right)}u' +\left(|k_0|^2-\omega^2 \perm_p\right) u &= \omega^2\perm_{as}v+ \left[\frac{|k_0|^2\perm'}{\perm\left(|k_0|^2-\omega^2 \perm\right)}-\frac{|k_0|^2\perm_p'}{\perm_p\left(|k_0|^2-\omega^2 \perm_p\right)}\right] v', \\
	u(x_0)=u'(x_0)&=0.
\end{aligned}
\eeq
We set
\beq\label{eq:ap8b}
r:=\omega^2\perm_{as}v+ \left[\frac{|k_0|^2\perm'}{\perm\left(|k_0|^2-\omega^2 \perm\right)}-\frac{|k_0|^2\perm_p'}{\perm_p\left(|k_0|^2-\omega^2 \perm_p\right)}\right] v'.
\eeq
Then, as in a),  the variation of constants formula provides
\beq\label{eq:ap9}
|u(x)|, |u'(x)|\leq C\int_{x_0}^x |r(t)|\ dt \hbox{ for } x\in[x_0,\infty).
\eeq
Since, by assumption, $\perm_{as}\in W^{1,1}(\R_+)$, the embedding $W^{1,1}(\R_+)\hookrightarrow C_b[0,\infty)$ implies boundedness of $\perm_{as}$ and that $\perm_{as}(x)\to 0$ as $x\to\infty$. Hence, our assumptions that $|\perm_p|\geq\delta$ and $ ||k_0|^2-\omega^2 \perm_p|\geq \delta$ with some $\delta>0$ imply the same inequalities on $[x_0,\infty)$ with $\perm_p$ replaced by $\perm$ and $\delta$ by $\delta/2$ when $x_0$ is chosen sufficiently large. Thus, $r$ given in \eqref{eq:ap8b} satisfies
\beq\label{eq:ap10}
|r|\leq C_1(|\perm_{as}|+|\perm'_{as}|)(|v|+|v'|) \hbox{ on } [x_0,\infty),
\eeq
which together with \eqref{eq:ap9} and $u=v-\psi_1$ gives, for $x\in(x_0,\infty)$,
\beqq
\|v\|_{L^\infty(x_0,x)}\leq CC_1 \left(\int_{x_0}^x (|\perm_{as}|+|\perm'_{as}|)\ dt \right) \left( \|v\|_{L^\infty(x_0,x)}+\|v'\|_{L^\infty(x_0,x)}  \right) +\|p_1\|_{L^\infty(0,a)}
\eeqq
and
\beqq
\|v'\|_{L^\infty(x_0,x)}\leq CC_1 \left(\int_{x_0}^x (|\perm_{as}|+|\perm'_{as}|)\ dt \right) \left( \|v\|_{L^\infty(x_0,x)}+\|v'\|_{L^\infty(x_0,x)}  \right) +\|p'_1+\ri mp_1\|_{L^\infty(0,a)}.
\eeqq
Thus, choosing $x_0$ such that
\beqq
CC_1 \left( \|\perm_{as}\|_{L^1(x_0,\infty)}+\|\perm'_{as}\|_{L^1(x_0,\infty)} \right)  \leq \frac13,
\eeqq
we obtain   $v\in C^1_b[x_0,\infty)$. Moreover, \eqref{eq:ap9} and \eqref{eq:ap10} show that
\beqq
\|u\|_{L^\infty(x_0,x)}\leq CC_1\left( \|\perm_{as}\|_{L^1(x_0,\infty)}+\|\perm'_{as}\|_{L^1(x_0,\infty)} \right) \left( \|v\|_{L^\infty(\R_+)}+\|v'\|_{L^\infty(\R_+)}  \right)
\eeqq
and hence for $x_0$ large enough
\beqq
\|u\|_{L^\infty(x_0,x)}\leq \frac14 \|p_1\|_{L^\infty(0,a)}.
\eeqq
Therefore,
\beqq
|v(x)|\geq |p_1(x)|-\frac14 \|p_1\|_{L^\infty(0,a)}, \ x\in[x_0,\infty).
\eeqq
The rest of the proof is as in a) and we conclude $\omega \in \Sigma^+_{e,v} \subset \sigma_{\mathrm{Weyl}}(\cL)$.
\epf


\subsection{Radiation along the Interface}\label{S:Weyl-interf}

The main building block of a Weyl sequence localized at the interface $x_1=0$ is an eigenfunction of the operator 
$$\widetilde{L}_{k_0}(\omega):=
\begin{pmatrix}
	|k_0|^2 -\omega^2\perm(x_1,\omega) & \ri |k_0|\pa_{x_1} & 0\\
	\ri |k_0| \pa_{x_1} & -\pa_{x_1}^2 -\omega^2\perm(x_1,\omega) & 0\\
	0 & 0 & -\pa_{x_1}^2+|k_0|^2-\omega^2\perm(x_1,\omega)
\end{pmatrix}
$$
with a fixed $k_0\in \R^{N-1}$ and the domain
$$
\begin{aligned}
	D(\widetilde{L}_{k_0}(\omega)) =& \left\{ (u_1, v, w) \in L^2 (\R)^3  : v' - \ri |k_0| u_1, v'' - \ri |k_0| u_1',  \ri |k_0| v' + |k_0|^2 u_1, w', |k_0| w, w'' - |k_0|^2 w \in L^2(\R_\pm),\right. \\
	&\left. \text{and} \ \eqref{E:divergence-uvw-1d} \ \text{and}\ \eqref{E:IFC-uvw-1d} \ \text{hold}\right\}.
\end{aligned}
$$
\beq\label{E:divergence-uvw-1d}
(\perm(\cdot,\omega)  u_1)' + \ri |k_0| \perm(\cdot,\omega) v = 0 \quad  \text{on} \quad \R_\pm,
\eeq
\beq\label{E:IFC-uvw-1d}
\begin{aligned}
	\llbracket \perm(\cdot,\omega)  u_1 \rrbracket = \llbracket v \rrbracket = \llbracket v' - \ri |k_0| u_1 \rrbracket = \llbracket w \rrbracket = \llbracket w' \rrbracket = 0.
\end{aligned}
\eeq

Note that, unlike in \eqref{E:IFC-u}, the notation $\llbracket f \rrbracket$ in \eqref{E:IFC-uvw-1d} denotes the difference of the classical limits $f(0+)-f(0-)$, which exist since all the arguments in \eqref{E:IFC-uvw-1d} are in $H^1(\R_\pm)$. Moreover, $\widetilde{L}_{k_0}$ is an operator pencil with its spectrum defined analogously to that of $\mathcal{L}$.

%


Due to the decoupling of the system \eqref{3}, \eqref{4} we find all eigenvalues $\omega$ of $\widetilde{L}_{k_0}$ by studying solutions of $\widetilde{L}_{k_0}(\omega)\psi=0$ of the form $\psi=(u_1,v,0)^\top$ and those of the form $\psi=(0,0,w)^\top$. 

For the former case we assume that there are fundamental systems of \eqref{3} on $\R_+$ denoted by $\{v^{(1)}_+, v^{(2)}_+\}$ and on $\R_-$ by $\{v^{(1)}_-, v^{(2)}_-\}$ with
\beq\label{E:L2-cond-v}
v^{(1)}_-\in L^2(\R_-), \quad v^{(2)}_- \notin L^2(\R_-), \quad v^{(1)}_+ \notin L^2(\R_+), \quad \text{and}  \quad v_+^{(2)} \in L^2(\R_+)
\eeq
and set
$$v:=\begin{cases}
	\alpha v^{(1)}_- & \text{on} \ \R_-,\\
	\beta v^{(2)}_+	 & \text{on} \ \R_+,
\end{cases}
$$
and $u_1:=-\frac{\ri |k_0|}{|k_0|^2-\omega^2\perm}v'$ with $\alpha, \beta \in \C$. We consider $\omega \in \C \setminus (\Omega^+_{|k_0|^2} \cup \Omega^-_{|k_0|^2})$. Then we have $u_1\in H^1(\R_\pm)$ and $(u_1,v)$ satisfies all $L^2$-conditions in $D(\widetilde{L}_{k_0}(\omega))$. For $k_0\neq 0$ the interface conditions $\llbracket v\rrbracket=\llbracket\perm u_1\rrbracket=0$ read
\beq\label{E:syst-alpha-beta-v}
\begin{pmatrix}
	-v^{(1)}_-(0-) & v^{(2)}_+(0+)\\
	-\frac{\perm(0-)}{|k_0|^2-\omega^2\perm(0-)}v^{(1)'}_-(0-) & \frac{\perm(0+)}{|k_0|^2-\omega^2\perm(0+)}v^{(2)'}_+(0+)
\end{pmatrix}\begin{pmatrix}
	\alpha \\ \beta
\end{pmatrix} =0.
\eeq
The last interface condition $\llbracket \ri |k_0| u_1-v'\rrbracket=0$ follows from the second equation in \eqref{E:syst-alpha-beta-v}. For $k_0=0$ the condition $\llbracket\perm u_1\rrbracket=0$ holds trivially and $\llbracket v\rrbracket=\llbracket \ri |k_0| u_1-v'\rrbracket=0$ reduces to
\beq\label{E:syst-alpha-beta-v-k0}
\begin{pmatrix}
	-v^{(1)}_-(0-) & v^{(2)}_+(0+)\\
	-v^{(1)'}_-(0-) & v^{(2)'}_+(0+)
\end{pmatrix}\begin{pmatrix}
	\alpha \\ \beta
\end{pmatrix} =0.
\eeq
In summary, the interface conditions \eqref{E:IFC-uvw} can be satisfied by $v\neq 0$ if and only if
\beq\label{E:det-v-eq}
\det  \begin{pmatrix}
	v^{(1)}_-(0-) & v^{(2)}_+(0+)\\
	\frac{\perm(0-)}{|k_0|^2-\omega^2\perm(0-)}v^{(1)'}_-(0-) & \frac{\perm(0+)}{|k_0|^2-\omega^2\perm(0+)}v^{(2)'}_+(0+)
\end{pmatrix}=0.
\eeq
Note that condition \eqref{E:det-v-eq} is the same as $d(k_0)=0$ for $d$ defined in \eqref{E:d} - however with the fundamental systems satisfying slightly weaker conditions than in Section \ref{S:resol-set}.

Finally, because $(\perm u_1)'+\ri |k_0|\perm v=\ri |k_0|\left(-\left(\frac{\perm}{|k_0|^2-\omega^2\perm}v'\right)'+\perm v \right)$, the divergence condition \eqref{E:divergence-uvw-1d} follows from \eqref{3}.

Similarly, for eigenfunctions of the form  $\psi=(0,0,w)^\top$, we assume that there are fundamental systems of \eqref{4} on $\R_+$ denoted by $\{w^{(1)}_+, w^{(2)}_+\}$ and on $\R_-$ by $\{w^{(1)}_-, w^{(2)}_-\}$ with 
\beq\label{E:L2-cond-w}
w^{(1)}_-\in L^2(\R_-), \ w^{(2)}_-\notin L^2(\R_-), \ \ w^{(1)}_+ \notin L^2(\R_+), \quad \text{and} \quad w^{(2)}_+ \in L^2(\R_+)
\eeq
and set
$$w:=\begin{cases}
	\alpha w^{(1)}_- & \text{on} \ \R_-,\\
	\beta w^{(2)}_+	& \text{on} \ \R_+
\end{cases} $$
with $\alpha, \beta \in \C$. The interface conditions $\llbracket w\rrbracket =\llbracket w'\rrbracket=0$ are
\beq\label{E:syst-alpha-beta-w}
\begin{pmatrix}
	-w^{(1)}_-(0-) & w^{(2)}_+(0+)\\
	-w^{(1)'}_-(0-) & w^{(2)'}_+(0+)
\end{pmatrix}\begin{pmatrix}
	\alpha \\ \beta
\end{pmatrix} =0.
\eeq
They can be satisfied with $w\neq 0$ if and only if
\beq\label{E:det-w-eq}
\det  \begin{pmatrix}
	w^{(1)}_-(0-) & w^{(2)}_+(0+)\\
	w^{(1)'}_-(0-) &w^{(2)'}_+(0+)
\end{pmatrix}=0.
\eeq
The divergence condition \eqref{E:divergence-uvw-1d} is trivially satisfied since $u_1=v=0$.

Note again that condition \eqref{E:det-w-eq} is the same as $\widetilde{d}(k_0)=0$ for $\widetilde{d}$ defined in \eqref{E:dtil} - up to the slightly different conditions on the fundamental system.

In summary, we have the following lemma.
\blem\label{L:evals-Lk}
Assume \eqref{E:ass-perm}, let $k_0\in \R^{N-1}$  and let $\{v^{(1)}_\pm, v^{(2)}_\pm\}$ and $\{w^{(1)}_\pm, w^{(2)}_\pm\}$ denote fundamental systems of \eqref{3} and \eqref{4} on $\R_\pm$ with $k=k_0$, respectively, satisfying the properties \eqref{E:L2-cond-v} and \eqref{E:L2-cond-w}, resp. Define
\beq\label{E:Nk0v}
N^{(v)}_{k_0}:=
\{\omega \in D(\perm)\setminus (\Omega^+_{|k_0|^2} \cup \Omega^-_{|k_0|^2}): \ \eqref{E:L2-cond-v}  \ \text{and} \ \eqref{E:det-v-eq}\ \text{hold}\}
\eeq
and
\beq\label{E:Nk0w}
N^{(w)}_{k_0}:=\{\omega \in D(\perm): \eqref{E:L2-cond-w}   \ \text{and} \ \eqref{E:det-w-eq} \ \text{hold}\}.
\eeq
Then
$$N^{(v)}_{k_0} \cup N^{(w)}_{k_0} \subset \sigma_p(\widetilde{L}_{k_0}).$$

Corresponding to $\omega \in N^{(v)}_{k_0}$  there is a solution
\beq\label{E:psi-v}
(u_1,v,w)^\top=(\tfrac{-\ri |k_0|v'}{|k_0|^2-\omega^2\perm},v,0)^\top=:\psi^{(v)}
\eeq of the homogenous version of \eqref{3}-\eqref{5}, \eqref{E:IFC-uvw}, and \eqref{E:divergence-uvw-1d} with $v$ given by
\beq\label{eq:vsol}
v:=\begin{cases}
	\alpha v^{(1)}_- & \text{on} \ \R_-,\\
	\beta v^{(2)}_+	 & \text{on} \ \R_+,
\end{cases}
\eeq
where $(\alpha, \beta)\in \C^2$ solves \eqref{E:syst-alpha-beta-v}.

Corresponding to $\omega \in N^{(w)}_{k_0}$  there is a solution
\beq\label{E:psi-w}
(u_1,v,w)^\top=(0,0,w)^\top=:\psi^{(w)}
\eeq
of the same system  with $w$ given by
\beq\label{ew:wsol}
  w:=\begin{cases}
	\alpha w^{(1)}_- & \text{on} \ \R_-,\\
	\beta w^{(2)}_+	& \text{on} \ \R_+
\end{cases} 
\eeq
where  $(\alpha, \beta)\in \C^2$ solves \eqref{E:syst-alpha-beta-w}.
\elem

\brem\label{R:k-abs}
Note that because the dependence on $k_0$ in all of   \eqref{3}-\eqref{5}, \eqref{E:IFC-uvw}, and \eqref{E:divergence-uvw-1d}  is only via the absolute value $|k_0|$, we have
$$N_{k_0}^{(v)}=N_{l}^{(v)} \ \text{and} \ N_{k_0}^{(w)}=N_{l}^{(w)} \ \text{for all} \ |k_0|=|l|.$$
\erem

\medskip

As we show next, the solution $(u_1,v,w)$ corresponding to $\omega \in N^{(v)}_{k_0}$ for some $k_0\in \R^{N-1}$ and the solution corresponding to  $\omega \in N^{(w)}_{k_0}$ for some $k_0\in \R^{N-1}$ can be used to generate a Weyl sequence for the operator pencil $\cL$. Since this can be carried out for any $k_0\in \R^{N-1}$, it makes sense to study the unions
\beq\label{E:Nv}
N^{(v)}:=\bigcup_{k_0\in \R^{N-1}}N^{(v)}_{k_0} \ \text{and } \ N^{(w)}:=\bigcup_{k_0\in \R^{N-1}}N^{(w)}_{k_0}.
\eeq

\bthm\label{T:Weyl-sp}
We have
$$N^{(v)} \cup N^{(w)} \subset \sigma_{\mathrm{Weyl}}(\cL).$$
\ethm
\bpf
With the help of Lemmas \ref{L:equiv-op} and \ref{lem:equiv} we work only with the operator $\widetilde{L}(\omega)$. We divide the proof into three parts. In the parts 1) and 2) we study the case $\omega \in N^{(v)}$ and in the last part we study $\omega \in N^{(w)}$.

1) First, we assume $\omega \in  N^{(v)}_{k_0}$ with $k_0\neq 0$. The corresponding eigenfunction of $\widetilde{L}_{k_0}(\omega)$ is denoted by $\psi^{(v)}$, see \eqref{E:psi-v}. Recall that the last component of $\psi^{(v)}$ is zero.
To construct the Weyl sequence $\xi_n=(u_{1,n}, v_n, w_n)^\top$, $n \in\N$, we first choose an arbitrary $\zeta \in \R^{N-1}\setminus\{0\}$ and define $\theta_n=(\theta_{1,n}, \theta_{v,n}, \theta_{w,n})^\top$,
\beq\label{eq:thetan}
\theta_n(x_1,k) := e^{-\ri n^2 (k-k_0)\cdot \zeta}n^\frac{N-1}{2}|k|^2\widehat\varphi(n(k-k_0)) \psi^{(v)}(x_1) + \widehat{r}_n(x_1,k),
\eeq
where $k=k_2$ if $N=2$, $k=(k_2,k_3)$ if $N=3$, $\varphi \in C^\infty_c(\R^{N-1},\R)$, and
$$
\widehat{r}_n(x_1,k):=\left(0, -n^\frac{N-1}{2}|k|(|k|-|k_0|)e^{-\ri n^2 (k-k_0)\cdot \zeta}\widehat\varphi(n(k-k_0)) \psi^{(v)}_2(x_1),0\right)^\top.
$$
Note that $\theta_n$ is the Fourier transform (with respect to $x_\parallel$) of
$$
\begin{aligned}
	\widecheck{\theta}_n(x_1,x_\parallel)=~& n^{-\frac{N-1}{2}}\left(|k_0|^2\varphi\left(\frac{x_\parallel-\zeta n^2}{n}\right) -2\ri n^{-1}k_0\cdot \nabla \varphi\left(\frac{x_\parallel-\zeta n^2}{n}\right) -n^{-2}\Delta \varphi\left(\frac{x_\parallel-\zeta n^2}{n}\right)\right)e^{\ri k_0\cdot x_\parallel}\psi^{(v)}(x_1)\\
	&+r_n(x_1,x_\parallel),
\end{aligned}
$$
where $x_\parallel=x_2$ if $N=2$ and $x_\parallel=(x_2,x_3)$ if $N=3$. The correction $r_n$  ensures that $\theta_n$ satisfies the divergence condition. Indeed, since $(\perm \psi^{(v)}_1)'+\ri |k_0| \perm \psi^{(v)}_2=0$, we get
$$(\perm \theta_{1,n})'+\ri |k|\perm \theta_{v,n}=\perm\left(\ri (|k|-|k_0|)\alpha_n \psi^{(v)}_2 + \ri |k| \widehat r_{n,2}\right)=0,$$
where 
$$\alpha_n(k):=e^{-\ri n^2 (k-k_0)\cdot \zeta}n^\frac{N-1}{2}|k|^2\widehat\varphi(n(k-k_0)).$$

The interface conditions $\llbracket \theta_{w,n}\rrbracket = \llbracket \pa_{x_1}\theta_{w,n}\rrbracket=0$ hold automatically because $\theta_{w,n}=0$. The conditions $\llbracket \theta_{v,n} \rrbracket=\llbracket \perm \theta_{1,n} \rrbracket = 0$ are satisfied too. This is because $\theta_{v,n}=\alpha_n \frac{|k_0|}{|k|}\psi^{(v)}_2$ and $\theta_{1,n} = \alpha_n \psi^{(v)}_1$ and  we have $\llbracket \psi^{(v)}_2 \rrbracket=\llbracket \perm \psi^{(v)}_1 \rrbracket = 0$.

The remaining interface condition in \eqref{E:IFC-uvw-1d} is not satisfied as
$$\llbracket \ri |k| \theta_{1,n} - \pa_{x_1}\theta_{v,n}\rrbracket = \alpha_n \frac{|k|^2-|k_0|^2}{|k||k_0|}\left\llbracket \pa_{x_1} \psi^{(v)}_2\right\rrbracket=:\gamma_n(k),$$
where we have used the fact that $\llbracket \perm \psi_1^{(v)}\rrbracket=0$ and $(|k_0|^2-W) \psi_1{(v)} + \ri|k_0|
\pa_{x_1} \psi_2^{(v)} =0$.
For future reference, we note that
\beq\label{E:gamma_n}
\gamma_n(k) = n^\frac{N-1}{2}\frac{|k|(|k|^2-|k_0|^2)}{|k_0|}\widehat{\varphi}(n(k-k_0))e^{-\ri n^2 (k-k_0)\cdot \zeta}\left\llbracket\pa_{x_1} \psi^{(v)}_2\right\rrbracket.
\eeq
The Weyl sequence $\xi_n$ will be chosen by introducing a correction term in such a way that this interface condition is satisfied (and all other conditions remain to hold). We set
\beq\label{E:ksi-n}
\xi_n=(u_{1,n}, v_n, w_n)^\top := c_n(\theta_n - \gamma_n \widehat{s}), \quad \widehat{s} = (\widehat{s}_1, \widehat{s}_2, 0 )^\top,
\eeq
with $\widehat{s}$ such that
\begin{eqnarray}
	\pa_{x_1}(\widehat{s}_1 \perm) +\ri \perm |k|\widehat{s}_2&=0, \label{E:div-s}\\
	\llbracket \widehat{s}_2 \rrbracket&=0, \label{E:ifc-s2} \\
	\llbracket \ri |k|\widehat{s}_1 -\pa_{x_1} \widehat{s}_2\rrbracket&=1, \label{E:ifc-mix-s}\\
	\llbracket \perm\widehat{s}_1 \rrbracket&=0. \label{E:ifc-ws1}
\end{eqnarray}
The constants $c_n>0$ are selected so that $\|\xi_n\|_{L^2(\R^N)^3}=1$ for all $n$.

For \eqref{E:div-s} we set
\beq\label{E:s2-hat}
\widehat{s}_2:=\frac{\ri}{|k|\perm}\pa_{x_1}(\widehat{s}_1\perm).
\eeq
The following $\widehat{s}_1$ is one possible choice satisfying \eqref{E:ifc-s2}-\eqref{E:ifc-ws1}, as one easily checks.
\beq\label{E:s1-hat}
\widehat{s}_1=\begin{cases}
	\left(\perm_-^{-1}+a(k)x_1^2\right)\eta(x_1), & x_1<0,\\
		\perm_+^{-1}\eta(x_1), & x_1>0,
\end{cases}
\eeq
where $\eta\in C^\infty_c(\R)$, $\eta(0)=1, \eta'(0)=\eta''(0)=0$ and $a(k):=-\frac{|k|}{2}\left(\ri + |k|\llbracket\perm^{-1}\rrbracket\right)$.

As a preliminary step before checking the $L^2$-conditions, we study $\widehat{s}$ and the derivatives $\pa_{x_1}\widehat{s}_1$, $\pa_{x_1}\widehat{s}_2$, and $\pa_{x_1}^2\widehat{s}_2$, which appear later in $\widetilde{L}(\omega)\widehat{s}$. First, we have
$$\|\widehat{s}_1(\cdot,k)\|_{L^2(\R_-)}\leq \|\perm_-^{-1} \eta\|_{L^2(\R_-)} + |a(k)|\|x_1^2\eta\|_{L^2(\R_-)}\leq c(1+|k|^2)$$
and
$$\|\widehat{s}_1(\cdot,k)\|_{L^2(\R_+)}\leq \|\perm_+^{-1} \eta\|_{L^2(\R_+)} \leq c$$
producing
\beq\label{E:s1-L2}
\|\widehat{s}_1(\cdot,k)\|_{L^2(\R)} \leq c(1+|k|^2).
\eeq
For $\widehat{s}_2$ we have
$$\widehat{s}_2(x_1,k)=\begin{cases} \ri|k|^{-1}\perm_-^{-1}(x_1)\left(\eta'(x_1)+a(k)\pa_{x_1}(\perm_-(x_1) x_1^2\eta(x_1))\right), & x_1<0,\\
	\ri|k|^{-1}\perm_+^{-1}(x_1)\eta'(x_1), & x_1>0,
\end{cases}
$$
hence
$$
\|\widehat{s}_2(\cdot,k)\|_{L^2(\R_-)}\leq c|k|^{-1}(1+|k|^2)=c(|k|^{-1}+|k|), \qquad \|\widehat{s}_2(\cdot,k)\|_{L^2(\R_+)}\leq c|k|^{-1},
$$
resulting in 
\beq\label{E:s2-L2}
\|\widehat{s}_2(\cdot,k)\|_{L^2(\R)}\leq c(|k|^{-1}+|k|).
\eeq

For the derivatives, we calculate
$$
\pa_{x_1} \widehat{s}_1(x_1,k) = \begin{cases}
	\left(2a(k) x_1-\frac{\perm_-'(x_1)}{\perm_-^2(x_1)}\right)\eta(x_1) + \left(\perm_-^{-1}+a(k)x_1^2\right)\eta'(x_1), & x_1<0,\\
	-\frac{\perm_+'(x_1)}{\perm_+^2(x_1)}\eta(x_1) + \perm_+^{-1}(x_1)\eta'(x_1), & x_1>0.
\end{cases}
$$
Due to  the definition of $\widehat{s}_2$ and 
$$\pa_{x_1}(\perm\widehat{s}_1)(x_1,k) =-\frac{\ri\perm'(x_1)}{|k|\perm^2(x_1)}\pa_{x_1}(\perm \widehat{s}_1)(x_1,k)+\frac{\ri}{|k|\perm(x_1)}\pa_{x_1}^2(\perm \widehat{s}_1)(x_1,k),
$$
one can easily observe that $\pa_{x_1}^m\widehat{s}_2(x_1,k), m=1,2$, have the form
$$
\pa_{x_1}\widehat{s}_2(x_1,k)=\begin{cases}
	|k|^{-1}(f_1(x_1) + g_1(x_1)a(k)), & x_1<0,\\
	|k|^{-1}h_1(x_1), & x_1>0
	\end{cases}
$$
and
$$
\pa_{x_1}^2\widehat{s}_2(x_1,k)=\begin{cases}
	|k|^{-1}(f_2(x_1) + g_2(x_1)a(k)), & x_1<0,\\
	|k|^{-1}h_2(x_1), & x_1>0
\end{cases}
$$
with some functions $f_1,f_2,g_1,g_2\in L^2(\R_-)$ and $h_1,h_2\in L^2(\R_+)$. The $L^2$-property of these functions follows from $\eta\in C^\infty_c(\R)$ and  $\perm \in W^{3,\infty}(\R_\pm)$. Indeed, derivatives of $\perm$ of order one, two, and three appear in these functions.

As a result we obtain
\beq\label{E:s12-der-L2}
\|\pa_{x_1}\widehat{s}_1(\cdot,k)\|_{L^2(\R)}\leq c(1+|k|^2) \quad \text{and} \quad \|\pa_{x_1}^m\widehat{s}_2(\cdot,k)\|_{L^2(\R)}\leq c(|k|^{-1}+|k|), m=1,2.
\eeq

For the $L^2$-estimates of $\xi_n$ (and $\widetilde{L}(\omega)\xi_n$) it is useful to understand $\||k|^p \gamma_n\|_{L^2(\R^{N-1})}$ for $p\in \Z, p \geq -1$. For such $p$ we have
 \beq\label{E:gam-kp-est}
 \begin{aligned}
 	\||k|^p\gamma_n\|_{L^2(\R^{N-1})}^2 &\leq c \left|\left\llbracket\pa_{x_1}\psi_2^{(v)}\right\rrbracket\right|^2 \frac{n^{N-1}}{|k_0|^2} \int_{\R^{N-1}} |k|^{2+2p}(|k|^2 -|k_0|^2)^2|\widehat{\varphi}(n(k-k_0))|^2  \dd k\\
 	&\leq c \int_{\R^{N-1}}(|k_0|+|\tfrac{\kappa}{n}|)^{2+2p} (2n^{-1}|\kappa||k_0|+n^{-2}|\kappa|^2)^2|\widehat{\varphi}(\kappa)|^2 \dd \kappa\\
 	&\leq cn^{-2}\|\varphi\|^2_{H^{3+p}(\R^{N-1})},
 \end{aligned}
 \eeq
where we have used the fact $\|f\|^2_{H^j(\R^{N-1})}=\int_{\R^{N-1}}(1+|k|^j)^2 |\widehat{f}(k)|^2\dd k$ for $j \in \N$. 

 Let us now check the $L^2$-conditions in $\widetilde{\mathcal{D}}(\omega)$. First, we show that $\xi_n\in L^2(\R^N)$ and that $m\leq c_n\leq M$ for all $n\in \N$ with some $m>0$ and $M>0$ independent of $n$. This will follow from $\psi^{(v)} \in L^2(\R^N)^3$, from $\widehat{r}_n, \gamma_n \widehat{s}\to 0$ in $L^2$ and from upper and lower bounds on $\|\alpha_n\psi^{(v)}\|_{L^2(\R^N)^3}$.

Since $\widehat{r}_{n,2}(x_1,k)=-n^\frac{N-1}{2}|k|(|k|-|k_0|)\widehat{\varphi}(n(k-k_0))\psi_2^{(v)}(x_1)e^{-\ri n^2(k-k_0)\cdot \zeta}$ and $\widehat{r}_{n,1} = \widehat{r}_{n,3} = 0$, we have
\beq\label{E:rn-est}
\begin{aligned}
	\|\widehat{r}_n\|_{L^2(\R^N)^3}^2&=\|\psi_2^{(v)}\|_{L^2(\R)}^2\int_{\R^{N-1}} |k_0+\tfrac{\kappa}{n}|^2(|k_0+\tfrac{\kappa}{n}|-|k_0|)^2|\widehat{\varphi}(\kappa)|^2\dd\kappa\\
	&\leq n^{-2}\|\psi_2^{(v)}\|_{L^2(\R)}^2\int_{\R^{N-1}}|k_0+\tfrac{\kappa}{n}|^2|\kappa|^2|\widehat{\varphi}(\kappa)|^2\dd\kappa\leq cn^{-2}\|\varphi\|_{H^2(\R^{N-1})}^2\to 0 \ (n\to\infty).
\end{aligned}
\eeq
 Note that $\|\varphi\|_{H^2(\R^{N-1})}<\infty$ since $\varphi\in C^\infty_c(\R^{N-1})$.

For $\gamma_n \widehat{s}$ we use \eqref{E:s1-L2}, \eqref{E:s2-L2}, and \eqref{E:gam-kp-est} to conclude
\beq\label{E:gam-s-est}
\|\gamma_n\widehat{s}\|_{L^2(\R^N)^3}  \leq cn^{-1}\|\varphi\|_{H^5(\R^{N-1})} \to 0  \ (n\to\infty).
\eeq

Next, we show the boundedness of $\|\alpha_n\|_{L^2(\R^{N-1})}$ from above and away from zero. Because $\alpha_n(k)=n^\frac{N-1}{2}|k|^2\widehat{\varphi}(n(k-k_0))e^{-\ri n^2(k-k_0)\cdot \zeta} $, using the dominated convergence, we obtain
\beq\label{E:alpha-est}
\frac{1}{2}|k_0|^4\|\widehat{\varphi}\|_{L^2(\R^{N-1})}^2\leq
\|\alpha_n\|_{L^2(\R^{N-1})}^2 = \int_{\R^{N-1}} |k_0+\tfrac{\kappa}{n}|^4|\widehat{\varphi}(\kappa)|^2\dd \kappa\leq 2|k_0|^4 \|\widehat{\varphi}\|^2_{L_2^2(\R^{N-1})}
\eeq
for all $n$ large enough.

Estimates \eqref{E:rn-est}, \eqref{E:gam-s-est}, and \eqref{E:alpha-est} imply $\xi_n\in L^2(\R^N)^3$ and $m\leq c_n\leq M$ for all $n\in \N$ with some $m$ and $M$ independent of $n$.

We continue by checking the remaining $L^2$-conditions on $\xi_n$ in order to belong to $\widetilde{\cD}_\omega$. Because $w_n=0$, it is clearly sufficient to show
$$ |k|^2 u_{1,n}, |k|\pa_{x_1} u_{1,n}, \pa_{x_1} v_{n} \pa_{x_1}^2v_{n}, |k| \pa_{x_1}v_{n} \in L^2(\R^N_\pm).$$
Note that $ |k|^2 u_{1,n}, \pa_{x_1} v_{n} \in L^2(\R^N_\pm)$ implies also $\pa_{x_1} v_{n} - \ri |k| u_{1,n} \in L^2(\R^N_\pm)$.
As $\widehat{\varphi}\in \mathcal{S}(\R^{N-1})$ and because every component in $\xi_n$ includes $\widehat{\varphi}(n(k-k_0))$ as a factor, we automatically get $|k|^2 u_{1,n} \in L^2(\R^N)$. For the same reason, $|k|\pa_{x_1} \xi_{n,j} \in L^2(\R^N_\pm)$ if $\pa_{x_1} \xi_{n,j} \in L^2(\R^N_\pm)$. Hence, we only need to verify that
$$\pa_{x_1} u_{1,n}, \pa_{x_1}^m v_{n}  \in L^2(\R^N_\pm) \quad \text{for} \quad m=1,2.$$
We have
$$\pa_{x_1}u_{1,n}=c_n(\alpha_n\pa_{x_1}\psi_1^{(v)} - \gamma_n \pa_{x_1}\widehat{s}_1),$$
where
$$\pa_{x_1}\psi_1^{(v)} = -\ri |k_0|\left(\frac{\psi^{(v)''}_2}{|k_0|^2-\omega^2\perm}+\frac{\omega^2\perm'}{(|k_0|^2-\omega^2\perm)^2}\psi^{(v)'}_2\right),$$
which is in $L^2(\R)$ due to $\psi_2^{(v)}\in H^2(\R_+)\oplus H^2(\R_-)$ (because $\psi^{(v)}\in D(\widetilde{L}_{k_0}(\omega))$),  $\perm\in W^{1,\infty}(\R_+)\oplus W^{1,\infty}(\R_-)$, and $\omega \notin \Omega$. Together with $\alpha_n \in L^2(\R^{N-1})$ this implies $\alpha_n \pa_{x_1}\psi^{(v)}_1\in L^2(\R^N)$. For the second term, i.e., $\gamma_n \pa_{x_1}\widehat{s}_1$, we use \eqref{E:s12-der-L2} and \eqref{E:gam-kp-est} to obtain $\|\gamma_n \pa_{x_1}\widehat{s}_1\|_{L^2(\R^N)}\leq c n^{-1}\|\varphi\|_{H^{5}(\R^{N-1})}$.

Next,
$$\pa_{x_1}^mv_{n}=c_n(\alpha_n\pa_{x_1}^m\psi_2^{(v)}+\pa_{x_1}^m\widehat{r}_{n,2}-\gamma_n \pa_{x_1}^m\widehat{s}_2), \ m=1,2.$$
The first term is in $L^2(\R^N)$ due to $\psi_2^{(v)}\in H^2(\R_+)\oplus H^2(\R_-)$ and  $\alpha_n \in L^2(\R^{N-1})$. For the second term we have
$$\pa_{x_1}^m \widehat{r}_{n,2}(x_1,k)=-n|k|(|k|-|k_0|)e^{-\ri n^2(k-k_0)\cdot \zeta} \widehat{\varphi}(n(k-k_0))\pa_{x_1}^m\psi^{(v)}_2(x_1),$$
which is again in $L^2(\R^N)$ using, in addition, $\widehat{\varphi}\in \mathcal{S}$. For the third term we use 
\eqref{E:s12-der-L2} and \eqref{E:gam-kp-est} to estimate $$\|\gamma_n \pa_{x_1}^m\widehat{s}_2\|_{L^2(\R^N)}\leq c n^{-1}\|\varphi\|_{H^{4}(\R^{N-1})}, m=1,2.$$
We conclude that $\xi_n\in \widetilde\cD_\omega$ for all $n\in \N$.

Next, we show $\check{\xi}_n \rightharpoonup 0$  in $L^2$. Thanks to \eqref{E:rn-est}, \eqref{E:gam-s-est}, and the boundedness of $(c_n)$ it is sufficient to show
\beq\label{E:weak-conv}
\int_{\R^N} \widecheck{\alpha}_n(x_\parallel)\psi^{(v)}_j(x_1)f(x_1,x_\parallel)\dd (x_1,x_\parallel) \to 0 \ (n\to \infty) \quad \text{for all} \quad f\in L^2(\R^N) \ \text{and} \ j=1,2,
\eeq
where $\widecheck{\alpha}_n$ is the inverse Fourier transformation of $\alpha_n$, i.e. (see the formula for $\widecheck{\theta}_n$)
$$
\begin{aligned}
	\widecheck{\alpha}_n(x_\parallel)& =n^{-\frac{N-1}{2}} e^{\ri k_0\cdot x_\parallel} \left(|k_0|^2 \varphi\left(\frac{x_\parallel-n^2\zeta}{n}\right) - 2\ri n^{-1}k_0\cdot \nabla \varphi\left(\frac{x_\parallel-n^2\zeta}{n}\right) - n^{-2} \Delta \varphi\left(\frac{x_\parallel-n^2\zeta}{n}\right) \right), 
\end{aligned}
$$
where we recall that $x_\parallel\in \R^{N-1}$. Because $\varphi \in C^\infty_c(\R^{N-1})$, we only need to show for $j=1,2$
$$I_{j,n}:=n^{-\frac{N-1}{2}}\int_{\R^N} \mu\left(\frac{x_\parallel-n^2\zeta}{n}\right) \psi^{(v)}_j(x_1)f(x_1,x_\parallel)\dd x \to 0 \ (n\to \infty) \quad \text{for all} \quad f\in L^2(\R^N), \mu \in C^\infty_c(\R^{N-1}).$$
For any $\mu \in C^\infty_c(\R^{N-1})$
\beq\label{E:Lseq-conv-v}
\begin{aligned}
	\left|n^{-\frac{N-1}{2}}\int_{\R^{N-1}} \mu\left(\frac{y-n^2\zeta}{n}\right) f(x_1,y)\dd y \right|&= \left| n^{-\frac{N-1}{2}}\int_{nK+n^2\zeta} \mu \left(\frac{y-n^2\zeta}{n}\right) f(x_1,y)\dd y\right|\\
	&\leq \|\mu\|_{L^2(K)}\|f(x_1,\cdot)\|_{L^2(nK+n^2\zeta)}\\
	& \leq \|\mu\|_{L^2(K)}\|f(x_1,\cdot)\|_{L^2(\R^{N-1}\setminus [-n,n]^{N-1})} \to 0 \ (n\to \infty),
\end{aligned}
\eeq
where $K:=\supp(\mu)$ and the last inequality holds for all $n$ large enough. Hence (with the help of Fubini's theorem)
\beq\label{E:Lseq-conv-v2}
\begin{aligned}
	|I_{j,n}|&\leq  \|\mu\|_{L^2(K)}\int_{\R}|\psi^{(v)}_j(x_1)|\|f(x_1,\cdot)\|_{L^2(\R^{N-1}\setminus [-n,n]^{N-1})}\dd x_1 \\
	&\leq  \|\mu\|_{L^2(K)} \|\psi_j^{(v)}\|_{L^2(\R)}\|f\|_{L^2(\R^N\setminus (\R\times [-n,n]^{N-1}))} \to 0
\end{aligned}
\eeq

It remains to verify $\widetilde{L}(\omega)\xi_n \to 0$ in $L^2$.  As seen in \eqref{E:ksi-n} and \eqref{eq:thetan}, $\xi_n$ consists of three terms. We study these separately.
 From
$$
\widetilde{L}(\omega) \widehat{r}_n =  -n^\frac{N-1}{2}|k|(|k|-|k_0|)e^{-\ri n^2 (k-k_0)\cdot \zeta} \widehat{\varphi}(n(k-k_0))
\begin{pmatrix}
	\ri|k|\pa_{x_1}\psi^{(v)}_2\\
	-\pa_{x_1}^2\psi^{(v)}_2-W\psi^{(v)}_2\\
	0
\end{pmatrix}
$$
one obtains similarly to \eqref{E:rn-est}
$$\|\widetilde{L}(\omega) \widehat{r}_n\|_{L^2(\R^N)^3}^2 \leq cn^{-2}\|\varphi\|^2_{H^3(\R^{N-1})}(\|\psi^{(v)}_2\|^2_{H^2(\R_+)}+\|\psi^{(v)}_2\|^2_{H^2(\R_-)}) \to 0 \quad (n\to \infty).$$
Next, 
$$\widetilde{L}(\omega)(\gamma_n \widehat{s}) = \gamma_n \begin{pmatrix}
	(|k|^2-\omega^2\perm)\widehat{s}_1 +\ri |k| \pa_{x_1}\widehat{s}_2\\
	\ri |k|\pa_{x_1}\widehat{s}_1 - (\pa_{x_1}^2+\omega^2\perm)\widehat{s}_2 \\
	0
\end{pmatrix}
$$
Using again \eqref{E:s1-L2},  \eqref{E:s2-L2}, \eqref{E:s12-der-L2}, and \eqref{E:gam-kp-est}, we get
$$\|\widetilde{L}(\omega)(\gamma_n \widehat{s})\|_{L^2(\R^N)^3}\leq c n^{-1}\|\varphi\|_{H^7(\R)} \to 0 \quad (n\to \infty).$$

The last term in $\widetilde{L}(\omega)\xi_n$ is $\widetilde{L}(\omega)(\alpha_n \psi^{(v)}) = \alpha_n \widetilde{L}(\omega) \psi^{(v)}$. Because $\widetilde{L}_{k_0}(\omega)\psi^{(v)}=0$, we have
$$\widetilde{L}(\omega)(\alpha_n \psi^{(v)})= \alpha_n\begin{pmatrix}
	(|k|^2-|k_0|^2)\psi^{(v)}_1 + \ri(|k|-|k_0|)\pa_{x_1}\psi^{(v)}_2\\
	\ri (|k|-|k_0|)\pa_{x_1}\psi^{(v)}_1\\
	0
\end{pmatrix}.$$
Since
$$\int_{\R^{N-1}}|\alpha_n(k)|^2(|k|-|k_0|)^2\dd k \leq n^{-2}\int_{\R^{N-1}} |k_0+\tfrac{\kappa}{n}|^4|\kappa|^2 |\widehat{\varphi}(\kappa)|^2\dd \kappa \to 0 \quad (n\to \infty),$$
analogously $\int_{\R^{N-1}}|\alpha_n(k)|^2(|k|^2-|k_0|^2)^2\dd k \to 0$, and  $\psi^{(v)}\in H^2(\R_+)\oplus H^2(\R_-)$, we also get that $\widetilde{L}(\omega) (\alpha_n \psi^{(v)}) \to 0$ in $L^2$.

In summary, $\widetilde{L}(\omega)\xi_n \to 0$ in $L^2$.

\medskip
2) Let us now choose $\omega \in N^{(v)}_0$. As Lemma \ref{L:evals-Lk} shows, there is a corresponding eigenfunction which has the form
$$\psi^{(v)}=(0,v_0,0)^\top,$$
where $v_0\in H^2(\R)$ solves $v_0'' + \omega^2 \perm(x,\omega)v_0 =0$. Like in part 1) we choose $\zeta \in \R^{N-1}$ and $\varphi \in C^\infty_c(\R^{N-1})$ arbitrary. The naive choice 
$$e^{-\ri n^2 k\cdot \zeta}n^{\frac{N-1}{2}}\widehat{\varphi}(nk)\psi^{(v)}(x_1)$$ 
does not satisfy the divergence condition \eqref{E:divergence-uvw-1d}. To correct for this, we choose
\beq\label{E:thetan}
\theta_n(x_1,k):=(\theta_{1,n},\theta_{v,n},\theta_{w,n})^\top(x_1,k):=\alpha_n(k)\psi^{(v)}(x_1)+\widehat{r}_n(x_1,k),
\eeq
where 
$$\alpha_n(k):=e^{-\ri n^2 k\cdot \zeta}n^{\frac{N-1}{2}}\widehat{\varphi}(nk) \quad \text{and} \quad \widehat{r}_n(x_1,k):=\left(\ri |k| \perm^{-1}(x_1)\omega^{-2}\alpha_n(k)v_0'(x_1),0,0\right)^\top.$$
Clearly, $\theta_n$ satisfies the divergence condition $(\perm \theta_{n,1})'+\ri |k|\perm \theta_{v,n}=0$
because  $v_0'' + \omega^2 \perm(x,\omega)v_0 =0$. 

The interface conditions $\llbracket \theta_{w,n}\rrbracket = \llbracket \pa_{x_1}\theta_{w,n}\rrbracket=0$ hold trivially. The condition $\llbracket \theta_{v,n} \rrbracket=0$ holds because $\llbracket \psi^{(v)}_2 \rrbracket=\llbracket v_0 \rrbracket=0$. The condition $\llbracket \perm \theta_{1,n} \rrbracket = 0$ is satisfied since $v_0\in H^2(\R)$ and hence $\llbracket v_0'\rrbracket=0$. Like in part 1) the remaining interface condition in \eqref{E:IFC-uvw-1d} is not satisfied. Indeed, using again $\llbracket v_0'\rrbracket =0$, we get
\beq\label{E:gamn}
\llbracket \ri |k| \theta_{1,n} - \pa_{x_1}\theta_{v,n}\rrbracket = \ri|k|\llbracket \theta_{1,n}\rrbracket=-|k|^2 \alpha_n(k)\omega^{-2} \llbracket \perm^{-1}v_0'\rrbracket=:\gamma_n(k)\neq 0.
\eeq
Analogously to part 1) we propose the Weyl sequence
$$\xi_n=(u_{1,n}, v_n, w_n)^\top := c_n(\theta_n - \gamma_n \widehat{s}), \quad \widehat{s} = (\widehat{s}_1, \widehat{s}_2, 0 )^\top,$$
with $\theta_n$ in \eqref{E:thetan}, $\gamma_n$ in \eqref{E:gamn}, $\widehat{s}_1$ in \eqref{E:s1-hat}, $\widehat{s}_2$ in \eqref{E:s2-hat}, and $c_n>0$  such that $\|\xi_n\|_{L^2(\R^N)}=1$ for all $n$.

For $\xi_n\in \widetilde{\cD}_\omega$ it remains to check the $L^2$-conditions in $\widetilde{\mathcal{D}}(\omega)$. We start again by showing that $\xi_n\in L^2(\R^N)^3$ and that $m\leq c_n\leq M$ for all $n\in \N$ with some $m>0$ and $M>0$ independent of $n$. First, one easily obtains
 \beq\label{E:alpha-kp-est}
 \||k|^p\alpha_n\|_{L^2(\R^{N-1})}\leq n^{-p}\|\varphi\|_{H^p(\R^{N-1})} \forall p\in \N_0.
 \eeq
This, together with $v_0'\in L^2(\R)$ and $\omega \notin \Omega_\delta$, implies
\beq\label{E:rn-est2}
\|\widehat{r}_n\|_{L^2(\R^{N})^3} \leq cn^{-1}\|v_0'\|_{L^2(\R)}\|\varphi\|_{H^1(\R^{N-1})} \to 0 \quad (n\to \infty).
\eeq
Analogously to \eqref{E:gam-kp-est}, we have for $\gamma_n$ from \eqref{E:gamn} and $p\in \Z, p\geq -1$
 \beq\label{E:gam-kp-est2}
\begin{aligned}
	\||k|^p\gamma_n\|_{L^2(\R^{N-1})}^2 &\leq c \left|\left\llbracket\perm^{-1}v_0'\right\rrbracket\right|^2 n^{N-1}\int_{\R^{N-1}} |k|^{4+2p}|\widehat{\varphi}(nk)|^2  \dd k\\
	&\leq c n^{-4-2p}\int_{\R^{N-1}}|\kappa|^{4+2p}\widehat{\varphi}(\kappa)|^2 \dd \kappa \leq cn^{-4-2p}\|\varphi\|^2_{H^{2+p}(\R^{N-1})}.
\end{aligned}
\eeq
Hence, using again the estimates \eqref{E:s1-L2} and \eqref{E:s2-L2} of $\|s_{1,2}\|_{L^2(\R)}$, we obtain
\beq\label{E:gam-s-est2}
\|\gamma_n \widehat{s}_1\|_{L^2(\R^N)} \leq cn^{-2}\|\varphi\|_{H^4(\R^{N-1})} \to 0 , \quad 
\|\gamma_n \widehat{s}_2\|_{L^2(\R^N)} \leq cn^{-1}\|\varphi\|_{H^3(\R^{N-1})} \to 0 \quad (n\to \infty).
\eeq
As a result we have 
$$\frac 12\|\varphi\|_{L^2(\R^N)}\leq \|\theta_n-\gamma_n \widehat{s}\|_{L^2(\R^N)}\leq 2 \|\varphi\|_{L^2(\R^N)}$$
for all $n$ large enough and we conclude that $m\leq c_n\leq M$ for all $n\in \N$ with some $m>0$ and $M>0$ independent of $n$. 

Just like in the case $k_0\neq 0$, the remaining $L^2$-conditions to be checked for $\xi_n\in \widetilde{\cD}_\omega$ are
$$ |k|^2 u_{1,n}, |k|\pa_{x_1} u_{1,n}, \pa_{x_1} v_{n}, \pa_{x_1}^2v_{n}, |k| \pa_{x_1}v_{n} \in L^2(\R^N_\pm).$$
These hold analogously to the case $k_0\neq 0$ using the fact $v_0\in H^2(\R)$ and the estimates \eqref{E:alpha-kp-est} as well as \eqref{E:gam-kp-est2} together with \eqref{E:s1-L2}, \eqref{E:s2-L2}, and \eqref{E:s12-der-L2}.

Next, we show $\check{\xi}_n\rightharpoonup 0$ in $L^2$ as $n\to\infty$. Due to \eqref{E:rn-est2} and \eqref{E:gam-s-est2} it remains to check $\alpha_n\psi^{(v)}\rightharpoonup 0$, i.e., to show \eqref{E:weak-conv} with $\alpha_n(k):=e^{-\ri n^2 k\cdot \zeta}n^{\frac{N-1}{2}}\widehat{\varphi}(nk)$ and $\psi^{(v)}=(0,v_0,0)^\top$.
We have 
$$
	\widecheck{\alpha}_n(x_\parallel)=n^{-\frac{N-1}{2}} \varphi\left(\frac{x_\parallel-n^2\zeta}{n}\right) 
$$
and the same argument as in \eqref{E:Lseq-conv-v2} shows $\alpha_n\psi^{(v)}\rightharpoonup 0$.

In the last step, we check that $\widetilde{L}(\omega)\xi_n\to 0$ in $L^2$ as $n\to \infty$. Firstly, we have
$$\widetilde{L}(\omega) \widehat{r}_n=\alpha_n(k) v_0' \perm^{-1}\omega^{-2} \left((|k|^2-\omega^2\perm)\ri |k|, -|k|^2,0 \right)^\top.$$
As $\omega \notin \Omega$, we get
$$\|\widetilde{L}(\omega) \widehat{r}_n\|_{L^2(\R^N)^3}\leq c n^{-1}\|v_0'\|_{L^2(\R)}\|\varphi\|_{H^3(\R^{N-1})} \to 0$$
using \eqref{E:alpha-kp-est}. Secondly,
$$\widetilde{L} (\omega)(\gamma_n\widehat{s}) = \gamma_n \left((|k|^2-\omega^2\perm)\widehat{s}_1+\ri|k|\partial_{x_1}\widehat{s}_2,\ri |k|\partial_{x_1}\widehat{s}_1-(\partial_{x_1}^2+\omega^2\perm)\widehat{s}_2,0 \right)^\top.$$
Combining again \eqref{E:s1-L2}, \eqref{E:s2-L2}, and \eqref{E:s12-der-L2} with \eqref{E:gam-kp-est}, we get
$$\|\widetilde{L}(\omega) (\gamma_n\widehat{s})\|_{L^2(\R^N)^3}\leq c n^{-1}\|\varphi\|_{H^6(\R^{N-1})} \to 0.$$
To complete this part of the proof, we show that $\widetilde{L}(\omega)(\alpha_n \psi^{(v)})\to 0$ in $L^2$. Indeed, as
$$\widetilde{L}(\omega)(\alpha_n \psi^{(v)})=\ri |k|\alpha_n\left(v_0', -v_0''-\omega^2\perm v_0, 0\right)^\top,$$
one gets with \eqref{E:alpha-kp-est}
$$\|\widetilde{L}(\omega)(\alpha_n \psi^{(v)})_1\|_{L^2(\R^N)}\leq cn^{-1}\|v_0'\|_{L^2(\R)}\|\varphi\|_{L^2(\R^{N-1})}\to 0$$
and $\widetilde{L}(\omega)(\alpha_n \psi^{(v)})_2=0$.

\medskip
3) Finally, we consider  $\omega \in  N^{(w)}$. Hence, there exists $k_0\in \R^{N-1}$ such that $\omega \in  N^{(w)}_{k_0}$. A corresponding eigenfunction of $\widetilde{L}_{k_0}(\omega)$ is denoted by $\psi^{(w)}$, see \eqref{E:psi-w}. Recall that the first two components of $\psi^{(w)}$ are zero.

We choose the Weyl sequence $\xi_n=(0,0, w_n)^\top, n\in\N,$ as
$$\xi_n(x_1,k_2,k_3):=c n^{\frac{N-1}{2}}e^{-\ri n^2 (k-k_0)\cdot \zeta}\widehat\varphi(n(k-k_0))\psi^{(w)}(x_1),$$
where $\zeta \in \R^{N-1}\setminus\{0\}$ and $\varphi\in C^\infty_c(\R^{N-1},\R)$ are arbitrary and $c=(\|\varphi\|_{L^2(\R^{N-1})}\|\psi^{(w)}\|_{L^2(\R)})^{-1}$ so that $\|\xi_n\|_{L^2(\R^N)}=1,$ which is checked by an easy calculation. Note that $\xi_n$ is the Fourier transform of
$$\check{\xi}_n(x_1,x_\parallel)=c n^{-\frac{N-1}{2}}e^{\ri k_0\cdot x_\parallel}\varphi\left(\frac{x_\parallel-\zeta n^2}{n}\right)\psi^{(w)}(x_1).$$
The divergence condition and all interface conditions are satisfied by $\xi_n$. In detail, the divergence condition as well as the interface conditions involving $u_1$ and $v$ are automatic because the first two components of $\xi_n$ are zero. And the interface conditions $\llbracket w\rrbracket=\llbracket w'\rrbracket=0$ hold because they are satisfied by the third component of $\psi^{(w)}$.

To conclude that $\xi_n\in \widetilde{\cD}_\omega$, the $L^2$-conditions in $\widetilde{\mathcal{D}}$ must be checked. $\xi_n\in L^2(\R^N)^3$ was discussed above. Since $\psi^{(w)}_{1,2}=0$, the only remaining $L^2$-properties are $\pa_{x_1}w_{n}, |k|w_{n}, \pa_{x_1}^2w_{n} - |k|^2w_{n}\in L^2(\R^N)$. These all follow from $\psi^{(w)}_3\in H^2(\R_+)\oplus H^2(\R_-)$ and from $\widehat{\varphi}\in \mathcal{S}(\R^{N-1})$.

The weak convergence $\xi_n\rightharpoonup 0$ in $L^2$ is shown analogously to  \eqref{E:Lseq-conv-v} and \eqref{E:Lseq-conv-v2} leading to
$$\left|\int_{\R^N}\check{\xi}_{n,3}(x)f(x)\dd x \right| \leq c\|\varphi\|_{L^2(K)}\|\psi^{(w)}_3\|_{L^2(\R)}\|f\|_{L^2(\R^N\setminus (\R\times [-n,n]^{N-1}))}\to 0 \quad (n\to \infty)$$
for every $f\in L^2(\R^N)$, where $K:=\supp(\varphi)$.

Finally, for $\widetilde{L}(\omega)\xi_n\to 0$ we use again that $\widetilde{L}_{k_0}(\omega)\psi^{(w)}=0$ to get
$$\widetilde{L}(\omega)\xi_n  = cn^\frac{N-1}{2}e^{-\ri n^2(k-k_0)\cdot \zeta}\widehat{\varphi}(n(k-k_0)) \begin{pmatrix}
	0\\0\\(|k|^2-|k_0|^2)\psi^{(w)}_3
\end{pmatrix}$$
and, with the substitution $\kappa = n(k-k_0)$,
$$\|\widetilde{L}(\omega)\xi_n \|_{L^2(\R^N)^3}^2 = c^2\|\psi_3^{(w)}\|_{L^2(\R)}^2\int_{\R^{N-1}}|\widehat{\varphi}(\kappa)|^2(2n^{-1}k_0\cdot \kappa + n^{-2}|\kappa|^2)\dd\kappa \to 0 \quad (n\to \infty).$$
\epf

\subsubsection{Media Periodic in $\R^N_\pm$}

Here we consider once more the example of periodic materials. In contrast to the case of radiation in the $x_1$-direction, we assume that $\perm(\cdot,\omega)$ is periodic on $\R_+$ as well as on $\R_-$ for each $\omega \in D(\perm)$.
As in Definition \ref{D:Sp}, the
%
	%
	%
	%
	discriminant corresponding to  \eqref{3} on $\R_\pm$ will be denoted by $D^{(v)}_\pm$ and the one for  \eqref{4} on $\R_\pm$ by $D^{(w)}_\pm$.
	We provide a representation of the sets $N_k^{(v)}$ and $N_k^{(w)}$ defined in \eqref{E:Nk0v} and  \eqref{E:Nk0w}, as well as the determinants $d$ and $\widetilde{d}$ from Definition \ref{D:S} using this Floquet-Bloch theory notation. The basic idea is that  a decaying solution $v$ of the homogeneous equation exists on $\R_\pm$ if $D_\pm^{(v)} \notin [-2,2]$. In this case the fundamental system $\psi_1, \psi_2$ from \eqref{eq:6} can be chosen as the fundamental system $v^{(1)}_\pm$ and $v^{(2)}_\pm$ in Definition \ref{D:S} and in \eqref{E:L2-cond-v}. This is because in this case $\Real(\kappa)>0$ in \eqref{eq:6}. An analogous statement holds  for the $w$-equation.
	
	\blem\label{L:Lk-evals-periodic}
	Assume $\perm(\cdot,\omega)$ is periodic on $\R_+$ and on $\R_-$ for each $\omega \in D(\perm)$ and satisfies \eqref{E:ass-perm}. Let $k\in \R^{N-1}$.
	
	Corresponding to the periodic problem \eqref{3} on $\R_\pm$ denote the discriminant by $D_\pm^{(v)}$ and let $d(k)$ be the quantity defined in Definition \ref{D:S}.

Analogously, for \eqref{4} on $\R_\pm$, we define $D_\pm^{(w)}$ and use $\tilde{d}(k)$ from Definition \ref{D:S}.
	
	Then, with $N^{(v)}_k$ and $N^{(w)}_k$ defined as in \eqref{E:Nk0v} and \eqref{E:Nk0w} (with $k_0$ replaced by $k$), respectively,
	$$N^{(v)}_k=\{\omega \in D(\perm)\setminus \Omega:  D^{(v)}_+, D^{(v)}_- \not\in [-2,2] \ \text{and} \ d(k)=0 \ \text{holds}\},$$
	$$N^{(w)}_k=\{\omega \in D(\perm):  D^{(w)}_+, D^{(w)}_-\not\in [-2,2] \ \text{and} \ \widetilde{d}(k)=0 \ \text{holds}\}.$$
	\elem
	\bpf
	As follows from the Floquet-Bloch theory, we get fundamental systems satisfying \eqref{E:L2-cond-v} if and only if
	\beq\label{E:Dv-eval}
	D^{(v)}_+, D^{(v)}_- \not\in [-2,2]
	\eeq
	or
	\beq\label{E:Dw-eval}
	D^{(w)}_+, D^{(w)}_-\not\in [-2,2].
	\eeq
	Note that \eqref{E:det-v-eq}  and $d(k)=0$ are the same - and similarly for \eqref{E:det-w-eq}  and $\tilde{d}(k)=0$.

	\epf
	
	\subsubsection{Media Homogenous in $\R^N_\pm$}
	
	Here we study once more the special case $\perm_\pm(x,\omega)=\perm_\pm(\omega)$ for $\omega \in D(\perm)\setminus \Omega_0$, where $\Omega_0$ is defined in \eqref{Omega0}.
	Recall the sets $N^{(v)}_k$ and $N^{(w)}_k$ defined in \eqref{E:Nk0v} and \eqref{E:Nk0w} (each with $k_0$ replaced by $k$), respectively, and $N^{(v)}$ given by \eqref{E:Nv}.
	
	\blem\label{L:Weyl-interf-hom}
	We consider the sets 
\beq
\begin{aligned}
\cN^\mathrm{red}&=\left\{\omega\in D(\perm)\setminus \Omega_0: \perm_+(\omega)+\perm_-(\omega)\neq 0, \frac{\omega^2\perm_+^2(\omega)}{\perm_+(\omega)+\perm_-(\omega)}, \frac{\omega^2\perm_-^2(\omega)}{\perm_+(\omega)+\perm_-(\omega)}\notin [0,\infty),\right. \\
&\qquad \left.  \frac{\omega^2\perm_+(\omega)\perm_-(\omega)}{\perm_+(\omega)+\perm_-(\omega)}\in (0,\infty)\right\}
\end{aligned}
\eeq
and
\beq
\begin{aligned}
N_k^\mathrm{red}= &\left\{\omega \in D(\perm)\setminus \Omega_0, |k|^2 -\omega^2\perm_+(\omega), \ |k|^2 -\omega^2\perm_-(\omega)\notin (-\infty,0], \right.\\
&\left. \ \text{and} \ \omega^2\perm_+(\omega)\perm_-(\omega)=|k|^2(\perm_+(\omega)+\perm_-(\omega))\right\}.
\end{aligned}
\eeq
We have
	\begin{align}
		N^{(v)}_k &= N_k^\mathrm{red} \quad \forall k \in \R^{N-1}\setminus \{0\},\label{E:Nvnz}\\
		N^{(v)}_0 &= \emptyset,\label{E:Nv0}\\
		N^{(w)}_k &= \emptyset \quad \forall k \in \R^{N-1}.\label{E:Nw}
	\end{align}
	Moreover, 
	$N^{(v)}=\cN^\mathrm{red} \subset \sigma_{\mathrm{Weyl}}(\cL)\setminus\Omega_0$.
	\elem
	\bpf
	For $\omega \in D(\perm)\setminus \Omega_0$ the homogenous form of \eqref{3} reduces to
	\beq\label{E:v-hommed}
	-v''+(|k|^2-\omega^2\perm)v=0.
	\eeq
	For $|k|^2-\omega^2\perm_+(\omega)\in (-\infty,0]$, equation \eqref{E:v-hommed} on $\R_+$ has only the trivial $L^2(\R_+)$ solution, i.e. \eqref{E:L2-cond-v} fails. The same argument applies for $|k|^2-\omega^2\perm_-(\omega)\in (-\infty,0]$.
	
	Assume $|k|^2-\omega^2\perm_\pm(\omega)\not\in (-\infty,0]$. Then $v^{(1)}_-(x_1)=c_-e^{\mu_-x_1}$ and $v^{(2)}_+(x_1)=c_+e^{-\mu_+ x_1}$, with $\mu_\pm :=\sqrt{|k|^2-\omega^2\perm_\pm(\omega)}$ and $c_\pm \in \C$. Hence, condition \eqref{E:L2-cond-v} is satisfied. Equation \eqref{E:det-v-eq} (for $k\neq 0$) is equivalent to
	$$\perm_+(\omega)\mu_-+\perm_-(\omega)\mu_+=0,$$
	which is equivalent to
	$$\omega^2 \perm_+(\omega)\perm_-(\omega) = |k|^2(\perm_+(\omega)+\perm_-(\omega)),$$
	see Remark 4.6 in \cite{BDPW24}. This shows \eqref{E:Nvnz}.
	
	
	Next, we show $N^{(v)}\subset \sigma_{\mathrm{Weyl}}(\cL)\setminus\Omega_0$. Once again, 
	$$N^{(v)}=\cup_{k\in \R^{N-1}}(N^{(v)}_k\setminus\Omega_0)\subset \sigma_{\mathrm{Weyl}}(\cL)\setminus (\Omega\cup\Omega_0)\subset \sigma_{\mathrm{Weyl}}(\cL)\setminus\Omega_0,$$
	where the first inclusion follows by Theorem \ref{T:Weyl-sp}.

	To prove \eqref{E:Nv0}, note that equation \eqref{E:det-v-eq} (for $k=0$) is $\mu_++\mu_-=0$, which cannot hold since $\Real \mu_\pm >0$ by the definition of $\sqrt{z},$ for  $z \in \C\setminus (-\infty,0].$
	
	The proof of \eqref{E:Nw} is completely analogous to \eqref{E:Nv0} because \eqref{4} has the same form as \eqref{E:v-hommed} with $k=0$  and \eqref{E:det-w-eq} has the same form as \eqref{E:det-v-eq} with $k=0$.

	The equality $N^{(v)}=\cN^\mathrm{red}$ follows from \eqref{E:Nvnz}, \eqref{E:Nred-eq} and the obvious fact that $N_0^\mathrm{red}=\emptyset$.
	The inclusion $N^{(v)}\subset \sigma_{\mathrm{Weyl}}(\cL)$ follows from Theorem \ref{T:Weyl-sp}.
	\epf

	
	\section{Non-Existence of Eigenvalues}
	
	Clearly, all gradient fields are in the kernel of the curl operator. Hence, if $\perm(\cdot,\omega)$ vanishes on an open set $B\subset \R^N$, then gradient fields $\nabla f$ with a smooth $f$ and $\supp(f)\subset B$ satisfy $L(\omega)\nabla f =0$ and $\nabla f  \in \cD_\omega$. This means that such $\omega$ is an eigenvalue of infinite geometric multiplicity. 
	
	However, outside $\Omega$, we expect no eigenvalues due to the shift invariance of the problem along any direction in the $x_\parallel-$variables. In other words, there is no mechanism for a full localization in $\R^N$. Hence, we expect $\sigma_p(\cL)\cap (\C\setminus \Omega)=\emptyset.$ This statement remains a conjecture and will be addressed in a future paper.
	Here we only show the non-existence of eigenvalues with finite geometric multiplicity.
    
	\bthm\label{T:no-evals-fin-multip}
	There are no eigenvalues of $\cL$ of finite geometric multiplicity.
	\ethm
	\bpf
	Let $u\in \cD_\omega$ be a solution of $L(\omega)u=0$ for some $\omega \in D(\perm)$. Then, due to the shift invariance of $\cL$ in $x_\parallel$, also $u_\xi(x):=u(x+(0,\xi)), x \in \R^N,$ satisfies $L(\omega)u_\xi=0$ for any $\xi\in \R^{N-1}$.
	
	The following proof of the fact that $E:=\{u_\xi: \xi\in \R^{N-1}\}\subset \cD_\omega$ is infinite dimensional can be found also in \cite{Eastham1973}, see Theorem 6.10.1. Let us assume that $\dim E<\infty$. For a fixed $\zeta \in \R^{N-1}$ the operator $S:f\mapsto f_\zeta$ is linear and satisfies $S:E\to E$. Hence, $S$ has an eigenvalue $\lambda$ and an eigenfunction $\varphi\in E$, i.e., $\varphi(x+(0,\zeta))=\lambda \varphi(x)$ for all $x\in \R^N$.
	
	As $$\|\varphi(\cdot+(0,\zeta))\|_{L^2(\R^N)^3}=\|\varphi\|_{L^2(\R^N)^3},$$
	and $S$ is symmetric, it follows that $\lambda\in \{-1,1\}$. We obtain $\varphi(x+(0,\zeta))=\varphi(x)$ for all $x\in \R^N$, i.e., $\varphi$ is periodic or anti-periodic. This is a contradiction to $\varphi\in L^2(\R^N)^3$.
	\epf
	
	\subsection{Media Homogeneous in $\R^N_\pm$}
	Let us first recall that
	$$ \Omega_0=\{\omega \in D(\perm): \omega^2\perm_+(\omega)=0 \quad \text{or} \quad \omega^2\perm_-(\omega)= 0\}.$$

	\blem\label{L:pt-spec-hom}
	Assume $\perm_\pm(\cdot,\omega) =\perm_\pm(\omega)$. Then $\sigma_p(\cL)\setminus \Omega_0=\emptyset.$
	\elem
	\bpf
	Let $\omega \in \sigma_p(\cL)\setminus\Omega_0$ and $(u_1,v,w)$ be the Fourier transform (in the variables $x_\parallel$) of a corresponding eigenfunction. Note that $(u_1,v,w)(k)\in L^2(\R)^3$ for almost all $k \in \R^{N-1}$ and ${\rm meas}(\supp(u_1,v,w))\neq 0$. Recall from Lemma \ref{lem:equiv}  that $u_1=-\frac{\ri|k|}{|k|^2-\omega^2\perm}v'$ and for homogeneous $\perm_\pm$ the components  $v(k)$ and $w(k)$ satisfy the same type of equations $-\varphi'' +(|k|^2-\omega^2\perm_\pm)\varphi =0$ on $\R_\pm$. As explained in the proof of Lemma \ref{L:Weyl-interf-hom}, the condition $\llbracket w\rrbracket=\llbracket w'\rrbracket=0$ and condition \eqref{E:det-w-eq} together with $w\in L^2(\R)$, imply $w=0$.
	
	For the $v$-component we saw in the proof of Lemma \ref{L:Weyl-interf-hom} that
	$$v(x_1,k)=\begin{cases}
		c_+(k)e^{-\mu_+(k)x_1}, & x_1>0,\\
		c_-(k)e^{\mu_-(k)x_1}, & x_1<0
	\end{cases}$$
	with $\Real \mu_\pm >0$ and $c_\pm:\R^{N-1}\to \C.$ The interface conditions for $v$ and $u_1$ are equivalent to
	\beq\label{E:DR}
	\omega^2 \perm_+(\omega)\perm_-(\omega) =|k|^2(\perm_+(\omega) +\perm_-(\omega))
	\eeq
	as shown in the introduction to Sec. \ref{S:Weyl-interf} and the proof of Lemma \ref{L:Weyl-interf-hom}. Equation \eqref{E:DR} is satisfied at almost every $k\in \supp (c_+)\cup \supp (c_-)$. However, because for $\omega \notin\Omega_0$ equation \eqref{E:DR}  holds only for two values of $k\in\R$ if $N=2$ and for $k$ on a circle in $\R^2$ if $N=3$, i.e., on a set of measure zero, we get a contradiction with $(u_1,v,w)$ being non-null.
	\epf
	%
%
%
\appendix
\section{Integration by Parts for  $\nabla\times $}

The following two results are possibly well known but for readers' convenience we provide them including the proofs.
\blem\label{L:PI-3D}
Let $N=2,3$. Then
$$\int_{\R^N} (\nabla \times u)\cdot \overline{v} \dd x = \int_{\R^N} u\cdot \overline{(\nabla \times v)} \dd x \quad \text{  for all } \quad u,v\in H(\Curl,\R^N):=\{f\in L^2(\R^N)^3: \curl f \in L^2(\R^N)^3\},$$
where we recall that $\nabla = (\pa_{x_1},\pa_{x_2},\pa_{x_3})^\top$ if $N=3$ and $\nabla = (\pa_{x_1},\pa_{x_2},0)^\top$ if $N=2$.
\elem
\bpf
For any open Lipschitz set $\Omega \subset \R^N$ we define
$$\Curl_0: H_0(\Curl,\Omega)\subset L^2(\Omega)^3\to L^2(\Omega)^3, u\mapsto \nabla \times u,$$
where  $$H_0(\Curl,\Omega):=\overline{C^\infty_c(\Omega)^3}^{H(\Curl,\Omega)},$$
i.e. with the closure with respect to the graph norm of $\Curl$. Then one has $\Curl = \Curl_0^*$ and hence, automatically,
$$\langle \Curl_0 u, v\rangle_{L^2(\Omega)}=\langle u, \Curl v\rangle_{L^2(\Omega)} \quad \forall u \in H_0(\Curl,\Omega), v\in H(\Curl,\Omega).$$
Lemma \ref{H0curl-fullspace} shows that $H_0(\Curl,\R^N)=H(\Curl,\R^N)$ and thus the statement follows.
\epf

\blem\label{H0curl-fullspace}
Let $N=2,3$. Then
$$H_0(\Curl,\R^N)=H(\Curl,\R^N).$$
\elem
\bpf
The proof is analogous to that for showing $H^1_0(\R^N)=H^1(\R^N)$.
Clearly, $H_0(\Curl,\R^N)\subset H(\Curl,\R^N).$ For the other direction let $u\in H(\Curl,\R^N)$ be given. For any $n\in \N$ we define
$$u^{(n)}:=(\eta_{\frac{1}{n}}*u)\theta_n.$$
Here $\eta_{\frac{1}{n}}:= n^N \eta(n\cdot),$ $\eta \in C^\infty_c(B_1(0)), \int_{\R^N}\eta(x)\dd x =1$ generates a Dirac sequence $(\eta_{\frac{1}{n}})_n$ and $\theta_n(x):=\theta(\tfrac{|x|}{n})$ is a cut-off function with $\theta\in C^\infty_c(\R)$, $\theta(s)=1$ for $|s|\leq 1,$ and $\theta(s)=0$  for $|s|\geq 2$. Because $(u^{(n)})_n \subset C^\infty_c(\R^N)$, it remains to prove that $u^{(n)} \to u$ in $H(\Curl,\R^N)$. We have
\beq\label{eq:conv}
\|u^{(n)}-u\|_{L^2}\leq \|(\eta_{\frac{1}{n}}*u-u)\theta_n\|_{L^2}+\|\theta_nu-u\|_{L^2}\leq \|\eta_{\frac{1}{n}}*u-u\|_{L^2}+\|\theta_nu-u\|_{L^2} \to 0 \quad (n\to \infty),
\eeq
where the convergence of the first term is shown, e.g., in Theorem 2.29 of \cite{AF2003} and the convergence of the second term is obvious.
For $N=2$, to show $\|\nabla\times u^{(n)}-\nabla\times u\|_{L^2(\R^2)}\to 0$  it suffices to show that
$\|\pa_2 u_3^{(n)}-\pa_2 u_3\|_{L^2(\R^2)}\to 0$ and similarly for $\pa_1u_3$ and $\pa_2u_1-\pa_1u_2$. For $N=3$ we need to show $\|\pa_2 u_3^{(n)}-\pa_3 u_2^{(n)}-(\pa_2 u_3-\pa_3 u_2)\|_{L^2(\R^3)}\to 0$ and similarly for $\pa_3u_1-\pa_1u_3$ and $\pa_2u_1-\pa_1u_2$. 

Let us show in detail $\|\pa_2 u_3^{(n)}-\pa_3 u_2^{(n)}-(\pa_2 u_3-\pa_3 u_2)\|_{L^2(\R^3)}\to 0$. All other convergences are analogous.
Using $\pa_j u^{(n)} = (\eta_{\frac{1}{n}}*u)\pa_j\theta_n+(\eta_{\frac{1}{n}}*\pa_j u)\theta_n$, this follows from
$$
\begin{aligned}
	\|\pa_2 u_3^{(n)}&-\pa_3 u_2^{(n)}-(\pa_2 u_3-\pa_3 u_2)\|_{L^2} \\
	&\leq \left\|(\eta_{\frac{1}{n}}*u_3)\pa_2\theta_n\right\|_{L^2}+\left\| (\eta_{\frac{1}{n}}*u_2)\pa_3\theta_n\right\|_{L^2}+\left\|\left(\eta_{\frac{1}{n}}*(\pa_2 u_3-\pa_3 u_2)-(\pa_2 u_3-\pa_3 u_2)\right)\theta_n\right\|_{L^2}\\
	&\quad + \left\|(\pa_2 u_3-\pa_3 u_2)\theta_n-(\pa_2 u_3-\pa_3 u_2)\right\|_{L^2}\\
	& \leq \frac{c}{n}\sum_{j=2}^3\|\eta_{\frac{1}{n}}*u_j\|_{L^2}+\left\|\eta_{\frac{1}{n}}*(\pa_2 u_3-\pa_3 u_2)-(\pa_2 u_3-\pa_3 u_2)\right\|_{L^2}+ \left\|(\pa_2 u_3-\pa_3 u_2)\theta_n-(\pa_2 u_3-\pa_3 u_2)\right\|_{L^2}\\
	& \to 0,
\end{aligned}
$$
where the convergence of the last two terms is analogous to \eqref{eq:conv}.
\epf

\section{Trace Operators}\label{S:traces}

The interface conditions \eqref{E:IFC} and \eqref{E:IFC-u} were formulated using trace operators, which we define below. First, for an open set $\Omega \subset \R^N$  we use the standard notation
	$$\Hdiv(\Omega):=\{E \in L^2(\Omega)^3: \nabla \cdot E \in L^2(\Omega)\}, \quad \Hcurl(\Omega):=\{E \in L^2(\Omega)^3: \nabla \times E \in L^2(\Omega)^3\},$$
where we recall that $\nabla=(\pa_{x_1},\pa_{x_2},0)^\top$ if $N=2$ and $\nabla=(\pa_{x_1},\pa_{x_2},\pa_{x_3})^\top$ if $N=3$. We also define $\Gamma := \{x\in \R^N:x_1=0\}$ and $\nu:= \rm{e}_1\in\R^N$, i.e., the unit normal on $\Gamma$ pointing outward from $\R^N_-$. The classical trace operator
$$T_0:H^1(\R^N)^m \to H^{1/2}(\Gamma)^m \quad (\text{such that  } T_0f=f|_\Gamma \text{ for } f\in C^1(\R^N)^m) $$
with $m \in \N$ is used below. It is known that $T_0$ is continuous and surjective, see Theorem 7.39 and Sec.~7.67 in \cite{AF2003}.  

The trace operators $T_\pm^n$ and $T_\pm^t$ in $\R^2_\pm$ used in \eqref{E:IFC} were defined in \cite{BDPW24}. The definition in $\R^3_\pm$, i.e. for $N=3$, is completely analogous. In summary, we have for $N=2,3$
$$
T^n_\pm: \Hdiv(\R^N_\pm) \to H^{-1/2}(\Gamma),  T^n_\pm E[\varphi]:= \int_{\R^N_\pm} E \cdot \nabla \tilde{\varphi} \dd x + \int_{\R^N_\pm} \tilde{\varphi} \nabla \cdot E  \dd x \qquad \forall \varphi \in H^{1/2}(\Gamma),
$$
where $\tilde{\varphi} \in H^1(\R^N)$ is such that $T_0 \tilde{\varphi}= \varphi$. The space $H^{-1/2}(\Gamma)$ is the dual space of $H^{1/2}(\Gamma)$. 

Similarly,
$$
T^t_\pm: \Hcurl(\R^N_\pm) \to (H^{1/2}(\Gamma)^3)',  T^t_\pm E[\Phi]:= \int_{\R^N_\pm} E \cdot \nabla \times \tilde{\Phi} \dd x - \int_{\R^N_\pm} \tilde{\Phi}\cdot \nabla \times E  \dd x \qquad \forall \Phi \in  H^{1/2}(\Gamma)^3,
$$
where $\tilde{\Phi} \in H^1(\R^N)^3$ is such that $T_0 \tilde{\Phi} = \Phi$.

\medskip
\noindent\textbf{Trace Operators and the Fourier Transform}

\noindent The trace operators $\widehat{T_\pm^n}$ in $\R^N_\pm$ used in \eqref{E:IFC-u} are defined via 
$$
\widehat{T^n_\pm}: \cF_t\Hdiv(\R^N_\pm) \to \cF_t H^{-1/2}(\Gamma),  u\mapsto \widehat{T^n_\pm}u:=\cF_t(T^n_\pm E),
$$
where $u=\cF_t E$ with $E\in \Hdiv(\R^N_\pm)$. This definition clearly satisfies
$$
\widehat{T^n_\pm}(\cF_t E) = \cF_t (T_\pm^n E) \qquad \forall E\in \Hdiv(\R^N_\pm).
$$

For $E \in C^1(\overline{\R^N_\pm})\cap L^2(\R^N_\pm)$ we have $\widehat{T^n_\pm} u (k) = \mp \nu\cdot u(0,k)=\mp u_1(0,k)$ for all $k\in \R^{N-1}$, where again $u = \cF_t E$. This follows from the fact that $T^n_\pm E = \mp E_1|_\Gamma$ (shown by the divergence theorem, see \cite{BDPW24}) and from $\cF_t(E_1|_\Gamma)(k)=u_1(0,k)$.

Note that elements of $ \cF_t H^{-1/2}(\Gamma)$ act on test functions $\varphi \in \cF_t^{-1}H^{1/2}(\Gamma)$. For $\varphi =\cF_t^{-1}\psi$ with $\psi \in H^{1/2}(\Gamma)$ we have
$$\widehat{T^n_\pm}u[\varphi] = \cF_t(T^n_\pm E)[\cF_t^{-1}\psi]= T_\pm^nE[\psi].$$

Similarly, $\widehat{T_\pm^t}$ are defined via 
$$
\widehat{T^t_\pm}: \cF_t\Hcurl(\R^N_\pm) \to (\cF_t H^{1/2}(\Gamma)^3)',  u\mapsto \widehat{T^t_\pm}u:=\cF_t(T^t_\pm E),
$$
where $u=\cF_t E$ with $E\in \Hcurl(\R^N_\pm)$. We have
$$
\widehat{T^t_\pm}(\cF_t E) = \cF_t (T_\pm^t E) \qquad \forall E\in \Hcurl(\R^N_\pm)
$$
and for $E \in C^1(\overline{\R^N_\pm})\cap L^2(\R^N_\pm)$ we get $\widehat{T^t_\pm} u (k) = \mp \nu \times u(0,k)= \mp (-u_2(0,k),u_1(0,k),0)^\top$ for all $k\in \R^{N-1}$, where again $u = \cF_t E$.

Elements of $ (\cF_t H^{1/2}(\Gamma)^3)'$ act on test functions $\Phi \in \cF_t^{-1}H^{1/2}(\Gamma)^3$. For $\Phi =\cF_t^{-1}\Psi$ with $\Psi \in H^{1/2}(\Gamma)^3$ we have
$$\widehat{T^t_\pm}u[\Phi] = \cF_t(T^t_\pm E)[\cF_t^{-1}\Psi]= T_\pm^tE[\Psi].$$

\section*{Acknowledgements}
Michael Plum acknowledges funding by the DFG, Project-ID 258734477 - SFB 1173. All authors acknowledge
funding from the EPSRC projects EP/W007037/1 and EP/W006553/1 Spectral properties of interface problems for Maxwell systems. For the purpose of open access, the authors have applied a CC BY public copyright licence (where permitted by UKRI, an Open Government Licence or CC BY ND public copyright licence may be used instead) to any Author Accepted Manuscript version arising.

\bibliographystyle{plain}
\bibliography{bibliography}

\end{document}